\documentclass[a4paper, 11pt, reqno]{amsart}

\let\oldtocsection=\tocsection
\let\oldtocsubsection=\tocsubsection
\let\oldtocsubsubsection=\tocsubsubsection
\renewcommand{\tocsection}[2]{\hspace{0em}\oldtocsection{#1}{#2}}
\renewcommand{\tocsubsection}[2]{\hspace{1em}\oldtocsubsection{#1}{#2}}
\renewcommand{\tocsubsubsection}[2]{\hspace{2em}\oldtocsubsubsection{#1}{#2}}

\usepackage{amsthm,amssymb,amsmath,enumerate,graphicx,psfrag,enumitem,mathrsfs}

\usepackage[margin=2.7cm]{geometry}
\usepackage{hyperref}
\hypersetup{colorlinks=true,
   citecolor=blue,
   filecolor=blue,
   linkcolor=blue,
   urlcolor=blue
}

\usepackage[textsize=footnotesize]{todonotes}

\usepackage{algorithm}
\usepackage{tikz}
\usetikzlibrary{arrows,snakes}

\newtheorem{definition}{Definition}[section]
\newtheorem{claim}{Claim}
\newtheorem{step}{Step}
\newtheorem{subclaim}{Subclaim}
\newtheorem{proposition}[definition]{Proposition}
\newtheorem{theorem}[definition]{Theorem}
\newtheorem{corollary}[definition]{Corollary}
\newtheorem{lemma}[definition]{Lemma}

\numberwithin{equation}{section}

\newcommand{\comment}[1]{}
\newcommand{\N}{\mathbb N}
\newcommand{\R}{\mathbb R}

\newcommand{\es}{\emptyset}

\newcommand{\cl}{{\rm cl}}

\newcommand{\cA}{\mathcal{A}}
\newcommand{\cB}{\mathcal{B}}

\newcommand{\cE}{\mathcal{E}}
\newcommand{\cF}{\mathcal{F}}

\newcommand{\cH}{\mathcal{H}}

\newcommand{\cK}{\mathcal{K}}

\newcommand{\cO}{\mathcal{O}}
\newcommand{\cP}{\mathcal{P}}

\newcommand{\cR}{\mathcal{R}}

\newcommand{\bP}{\mathbf{P}}

\newcommand{\bS}{\mathbf{S}}

\newcommand{\ba}{\mathbf{a}}
\newcommand{\bb}{\mathbf{b}}

\newcommand{\bd}{\mathbf{d}}

\newcommand{\bx}{\mathbf{x}}
\newcommand{\by}{\mathbf{y}}
\newcommand{\bz}{\mathbf{z}}
\newcommand{\bw}{\mathbf{w}}
\newcommand{\bu}{\mathbf{u}}
\newcommand{\bv}{\mathbf{v}}

\newcommand{\dist}{{\rm dist}}

\newcommand{\sA}{\mathscr{A}}
\newcommand{\sB}{\mathscr{B}}

\newcommand{\sG}{\mathscr{G}}
\newcommand{\sH}{\mathscr{H}}

\newcommand{\sJ}{\mathscr{J}}

\newcommand{\sL}{\mathscr{L}}
\newcommand{\sM}{\mathscr{M}}

\newcommand{\sO}{\mathscr{O}}
\newcommand{\sP}{\mathscr{P}}
\newcommand{\sQ}{\mathscr{Q}}
\newcommand{\sR}{\mathscr{R}}

\newcommand{\kk}{{(k)}}

\newcommand{\norm}[1]{\|#1\|_\infty}

\renewcommand{\epsilon}{\varepsilon}

\newcommand{\sm}{\setminus}
\newcommand{\sub}{\subseteq}

\newcommand{\COMMENT}[1]{}

\title{Hypergraph regularity and random sampling}

\author[F.~Joos]{Felix Joos}
\address{Institut f\"ur Informatik, Universit\"at Heidelberg, Germany}
\email{joos@informatik.uni-heidelberg.de}

\author[J.~Kim]{Jaehoon Kim}
\address{Department of Mathematical Sciences, KAIST, South Korea 34141}
\email{jaehoon.kim@kaist.ac.kr}

\author[D.~K\"uhn]{Daniela K\"uhn}
\author[D.~Osthus]{Deryk Osthus}
\address{School of Mathematics, University of Birmingham, Edgbaston, Birmingham, B15 2TT, United Kingdom}
\email{d.kuhn@bham.ac.uk, d.osthus@bham.ac.uk}

\thanks{
 The research leading to these results was partially supported by the EPSRC, grant nos. EP/M009408/1 (F.~Joos, D.~K\"uhn and D.~Osthus), 
and EP/N019504/1 (D.~K\"uhn).
The research was  also partially supported by the European Research Council under the European Union's Seventh Framework Programme (FP/2007--2013) / ERC Grant 306349 (J.~Kim and D.~Osthus).}

\date{\today}

\begin{document}

\begin{abstract}
Suppose that a $k$-uniform hypergraph $H$ satisfies a certain regularity instance
(that is, there is a partition of $H$ given by the hypergraph regularity lemma into a bounded number of quasirandom subhypergraphs of prescribed densities).
We prove that with high probability a large enough uniform random sample of the vertex set of $H$ also admits  the same regularity instance.
Here the crucial feature is that the error term measuring the quasirandomness of the subhypergraphs requires only an arbitrarily small additive correction. 
This has applications to combinatorial property testing. 
The graph case of the sampling result was proved by Alon, Fischer, Newman and Shapira.
\end{abstract}
\maketitle

\section{Introduction}
Szemer\'edi's regularity lemma \cite{Sz78} is one of the most important results in discrete mathematics and has numerous applications.
Roughly speaking, it states that every graph can be partitioned into bounded number of vertex sets so that most of bipartite graphs between the parts are random-like.
To be more precise, it says that for every $\epsilon>0$ and $m\in \N$, 
there exists an $M(\epsilon,m)$ such that every large enough graph $G$ admits an equipartition $V_1,\ldots, V_t$ of its vertex set such that $m\leq t \leq M$ and all but at most $\epsilon t^2$ pairs $V_i,V_j$ induce an $\epsilon$-regular pair in $G$;
as usual, a bipartite graph with vertex sets $A,B$ is $\epsilon$-regular if it is $(\epsilon,d)$-regular for some $d\in [0,1]$. The latter means that
$|d-e(A',B')(|A'||B'|)^{-1}|\leq \epsilon$ for all subsets $A'\subseteq A,B'\subseteq B$ with $|A'|\geq \epsilon|A|,|B'|\geq \epsilon|B|$.

Being $\epsilon$-regular encapsulates random-like behaviour.
Much of the information about the graph is captured by the vertex partition $V_1,\ldots,V_t$ together with the densities $d(V_i,V_j):=e(V_i,V_j)(|V_i||V_j|)$ between the pairs $ij\in \binom{[t]}{2}$ of parts.
Consequently, it turned out to be interesting in many occasions to decide whether a graph $G$ on $n$ vertices satisfies a particular \emph{regularity instance} $R_{\epsilon}(\mathbf{d})$ which is defined as follows.

\begin{definition}
For given $\epsilon>0$, $t\in \N\setminus \{1\}$ and $\mathbf{d} \in [0,1]^{\binom{[t]}{2}}$, we say that a graph $G$ satisfies a regularity instance $R_{\epsilon}(\mathbf{d})$ if $G$ admits an equipartition $V_1,\ldots, V_t$ 
such that $V_i,V_j$ is an $(\epsilon,\mathbf{d}(ij))$-regular pair for each $ij\in \binom{[t]}{2}$.
\end{definition}

The regularity lemma was further extended to hypergraphs in the ground-breaking work of R\"odl and Frankl~\cite{FR92} (who proved the $3$-uniform case),
R\"odl and Skokan~\cite{RS04} (who proved the $r$-uniform case, with the corresponding counting lemma proved by R\"odl and Schacht \cite{RS07count}) as well as by Gowers~\cite{Gow06, Gow07}. This theory was further developed in e.g. R\"odl and Schacht~\cite{RS07}, Tao~\cite{Tao06}, Allen, B\"ottcher, Cooley and Mycroft~\cite{ABCM:17}, Nagle, R\"odl and Schacht~\cite{NRS18}, Allen, Davies and Skokan~\cite{ADS:19}, Moshkovitz and Shapira~\cite{MS:19}. In particular, in~\cite{RS07}, R\"odl and Schacht proved the so-called `regular approximation lemma', which is a powerful version of the hypergraph regularity lemma and which is of central importance to our proof.

The hypergraph regularity lemma guarantees that every (large) $k$-uniform hypergraph admits a partition of its edge set, where most classes of the partition consist of `regularly distributed' edges.
The appropriate notion for being regularly distributed is significantly more complicated than in the graph case and it took two decades
until a suitable notion was found and a corresponding theory was established.
At a high level,
many concepts from graph regularity carry over to hypergraphs.
In particular, one can define regularity instances for hypergraphs as a sequence of densities recorded from the cluster structure provided by the hypergraph regularity lemma.
Again in many places it has turned out to be fruitful to work with the much less complex regularity instance instead of the hypergraph itself.


\medskip

Another very natural way to study the properties of a given discrete structure is via random sampling. 
By examining a small random sample $S$ of a combinatorial object $\cO$,
can we determine (with high probability) whether $\cO$ has a specific property $\bP$
or whether it is far from satisfying $\bP$?
Questions of this type are known as property testing and have been intensively studied.

Rubinfeld and Sudan~\cite{RS96} introduced property testing and Goldreich, Goldwasser and Ron~\cite{GGR98} obtained results regarding 
$k$-colourability, max-cut and more general graph partitioning problems. 
(In fact, these results are preceded by the famous triangle removal lemma of Ruzsa and Szemer\'edi~\cite{RS78}, which can be rephrased in terms of testability of triangle-freeness.)
This list of problems was greatly extended, see \cite{AFKS00, AFNS09, AS08, AS08a, Fish05, GT03}. 
This sequence of results in property testing culminated in the result of Alon, Fischer, Newman and Shapira~\cite{AFNS09} who obtained a complete combinatorial characterization of all testable graph properties.
This solved a problem posed in~\cite{GGR98}, which was regarded as one of the main open problems in the area.
In fact, the combinatorial characterization says that a property is testable if and only if 
it can be decided whether a graph has the property solely by considering the regularity instances it satisfies.
A key step in their approach is the result that one can efficiently test whether a graph satisfies a particular regularity instance.

Therefore,
this suggests that is also of high importance to prove a similar result for hypergraphs.
Duke and R\"odl \cite{DR85} proved that randomly sampled subgraphs of a dense and regular pair $(V_i,V_j)$ almost surely span an edge.
Alon, de la Vega, Kannan and Karpinski ~\cite{AVKK03} as well as
Mubayi and R\"odl~\cite{MR03} then proved a stronger `inheritance' result in the setting of uniform hypergraphs, i.e.~with high probability uniform edge-distribution is inherited by random samples. Further generalizing these results, Czygrinow and Nagle~\cite{CN11} proved that if a hypergraph satisfies a regularity instance, 
then with high probability a randomly sampled hypergraph also satisfies `the same' regularity instance.
However, their result has the disadvantage that their argument gives polynomial regularity dependence which we sharpen to nearly best possible, up to an additive correction. We also reverse the inference in the result of Czygrinow and Nagle~\cite{CN11}, that regularity instances of sampled subhypergraphs predict that of the host hypergraph.

As the precise definition of a regularity instance for hypergraphs is fairly involved,
let us now only state our result for graphs; 
in fact, Theorem~\ref{thm:main_graphs} is proved by Alon, Fischer, Newman and Shapira~\cite{AFNS09} as a crucial tool for their characterization of testable graph properties.

\begin{theorem}[\cite{AFNS09}]\label{thm:main_graphs}
For all $\epsilon>0$ and all $\delta>0$ that are small in terms of $\epsilon$,
there exists $c>0$ such that for all sufficiently large $n$ and $q$ with $n\geq q$, the following holds.
Suppose $(\epsilon,t,d_t)$ is a regularity instance and
suppose $G$ is a graph on $n$ vertices with vertex set $V$.
Let $Q \in \binom{V}{q}$ be chosen uniformly at random. 
Then with probability at least $ 1 - e^{-cq}$ the following hold.
\begin{itemize}
	\item If $G$ satisfies the regularity instance $(\epsilon,t,d_t)$,
	then $G[Q]$ satisfies the regularity instance $(\epsilon+\delta,t,d_t)$.
	\item If $G[Q]$ satisfies the regularity instance $(\epsilon,t,d_t)$,
	then $G$ satisfies the regularity instance $(\epsilon+\delta,t,d_t)$.
\end{itemize}
\end{theorem}

The main result of this paper is a hypergraph version for Theorem~\ref{thm:main_graphs}.
In another paper~\cite{JKKO1}, we exploit this theorem to prove a combinatorial characterization for testable hypergraph properties. 
For this it is again crucial that a regularity instance can be tested with an arbitrarily small additive error.
It is by no means clear that one can generalize the results of~\cite{AFNS09} to hypergraphs and obtain such a characterization of testable hypergraph properties: Austin and Tao~\cite{AT10} showed that for the stronger but related notion of `repairability' the graph results do not extend to hypergraphs; see~\cite{AT10,JKKO1} for a more detailed discussion.

It is beyond the scope of an introduction to present the precise statement of such a result as first several notions concerning hypergraph regularity have to be introduced.
We therefore defer the statement of our main result to Section~\ref{sec: more concepts and tools}.

In order to prove our main result, we prove Lemma~\ref{lem:refreg}, which is a strengthening of a version of the hypergraph regularity lemma. 
This is a strengthening of a regular approximation lemma (Lemma~\ref{RAL(k)}) proved by R\"odl and Schacht~\cite{RS07}.
The latter is a version of the regularity lemma which of a given hypergraph $H$ guarantees that by modifying a small proportion of the hyperedges one can obtain a hypergraph $H'$ which has a `high' quality regularity lemma partition, that is, with very small error terms.
We believe that Lemma~\ref{lem:refreg} is also of independent interest and will have additional applications. 
In particular, we apply it in~\cite{JKKO1} to derive from our testability characterization that the max cut problem is testable.


\section{Concepts and tools}\label{sec:concepts}

In this section we introduce the main concepts and tools (mainly concerning hypergraph regularity partitions) which form the basis of our approach. The constants in the hierarchies used to state our results have to be chosen from right to left. More precisely, if we claim that a result holds whenever $1/n \ll a \ll b \leq 1$ (where $n\in \N$ is typically the number of vertices of a hypergraph), 
then this means that there are non-decreasing functions $f : (0, 1] \rightarrow (0, 1]$ and $g : (0, 1] \rightarrow (0, 1]$ such that the result holds for all $0 < a, b \leq 1 $ and all $n \in \mathbb{N}$ with $a \leq f(b)$ and $1/n \leq g(a)$.
For a vector $\bx =(\alpha_1,\dots, \alpha_\ell)$, we let $\bx_*:= \{ \alpha_1,\dots, \alpha_\ell\}$ and write $\norm{\bx} = \max_{i\in [\ell]}\{ \alpha_i\}$. 
We say a set $E$ is an \emph{$i$-set} if $|E|=i$.
Unless stated otherwise, in the partitions considered in this paper, we allow some of the parts to be empty.

\subsection{Hypergraphs}\label{sec: basic notation}

In the following we introduce several concepts about a hypergraph $H$.
We typically refer to $V=V(H)$ as the vertex set of $H$
and usually let $n:=|V|$.
Given a hypergraph $H$ and a set $Q\subseteq V(H)$, we denote by $H[Q]$ the hypergraph induced on $H$ by $Q$.
For two $k$-graphs $G,H$ on the same vertex set, we often refer to $|G\triangle H|$  as the \emph{distance} between $G$ and $H$.
If the vertex set of $H$  has a partition $\{V_1,\ldots,V_\ell\}$,
we simply refer to $H$ as a hypergraph on $\{V_1,\ldots,V_\ell\}$.

A partition $\{V_1,\ldots,V_\ell\}$ of $V$ is an \emph{equipartition} if $|V_i|=|V_j|\pm 1$ for all $i,j\in [\ell]$. For a partition $\{V_1,\ldots,V_\ell\}$ of $V$ and $k\in [\ell]$,
we denote by $K^{(k)}_\ell(V_1,\dots, V_\ell)$ the \emph{complete $\ell$-partite} $k$-graph with vertex classes $V_1,\ldots, V_\ell$.
Let $0\leq \lambda<1$.
If $|V_i| = (1\pm \lambda)m$ for every $i\in [\ell]$, 
then an \emph{$(m,\ell,k,\lambda)$-graph} $H$ on $\{V_1, \dots, V_{\ell}\}$ 
is a spanning subgraph of $K^\kk_{\ell}(V_1,\dots, V_{\ell})$.
For notational convenience, we consider the vertex partition $\{V_1,\ldots ,V_\ell\}$ as an $(m,\ell,1,\lambda)$-graph. 
If $|V_i| \in \{m,m+1\}$, 
we drop $\lambda$ and simply refer to $(m,\ell,k)$-graphs.
Similarly, if the value of $\lambda$ is not relevant, then we say $H\subseteq K^\kk_{\ell}(V_1,\dots, V_{\ell})$ is an $(m,\ell,k,*)$-graph.

Given an $(m,\ell,k,*)$-graph $H$ on $\{V_1,\dots, V_\ell\}$, an integer $k\leq i\leq \ell$ and a set $\Lambda_i \in \binom{[\ell]}{i}$, 
we set $H[\Lambda_i] := H[\bigcup_{\lambda'\in \Lambda_i} V_{\lambda'}]$. If $2\leq k\leq i \leq \ell$ and $H$ is an $(m,\ell,k,*)$-graph,
we denote by $\cK_i(H)$ the family of all $i$-element subsets $I$ of $V(H)$ for which $H[I]\cong K^{\kk}_i$,
where $K^{\kk}_i$ denotes the complete $k$-graph on $i$ vertices.

If $H^{(1)}$ is an $(m,\ell,1,*)$-graph and $i\in [\ell]$, 
we denote by $\cK_i(H^{(1)})$ the family of all $i$-element subsets $I$ of $V(H^{(1)})$ which `cross' the partition $\{V_1,\dots,V_\ell\}$; 
that is, $I\in \cK_i(H^{(1)})$ if and only if $|I\cap V_s|\leq 1$ for all $s\in [\ell]$.

We will consider hypergraphs of different uniformity on the same vertex set.
Given an $(m,\ell,k-1,\lambda)$-graph $H^{(k-1)}$ and an $(m,\ell,k,\lambda)$-graph $H^\kk$ on the same vertex set, 
we say $H^{(k-1)}$ \emph{underlies} $H^\kk$ if $H^\kk\subseteq \cK_k(H^{(k-1)})$;
that is, for every edge $e\in H^\kk$ and every $(k-1)$-subset $f$ of $e$,
we have $f\in H^{(k-1)}$.
If we have an entire cascade of underlying hypergraphs we refer to this as a complex.
More precisely, 
let $m\geq 1$ and $\ell\geq k\geq 1$ be integers. 
An \emph{$(m,\ell,k,\lambda)$-complex}  $\cH$ on $\{V_1,\ldots,V_\ell\}$ 
is a collection of $(m,\ell,j,\lambda)$-graphs $\{H^{(j)}\}_{j=1}^{k}$ on $\{V_1,\ldots,V_\ell\}$ such that $H^{(j-1)}$ underlies $H^{(j)}$ for all $i\in [k]\sm\{1\}$, that is, $H^{(j)} \subseteq \cK_{j}(H^{(j-1)})$.
Again, if $|V_i| \in \{m,m+1\}$, then we simply drop $\lambda$ and refer to such a complex as an $(m,\ell,k)$-complex. 
If the value of $\lambda$ is not relevant, then we say that $\{H^{(j)}\}_{j=1}^{k}$  is an $(m,\ell,k,*)$-complex. A collection of hypergraphs is a \emph{complex} if it is an $(m,\ell,k,*)$-complex for some integers $m,\ell,k$.

When $m$ is not of primary concern, we refer to $(m,\ell,k,\lambda)$-graphs and $(m,\ell,k,\lambda)$-complexes simply as 
$(\ell,k,\lambda)$-graphs and $(\ell,k,\lambda)$-complexes, respectively.
Again, we also omit $\lambda$ if $|V_i|\in \{m,m+1\}$ and refer to $(\ell,k)$-graphs and $(\ell,k)$-complexes and we write the symbol `$*$' instead of $\lambda$ if $\lambda$ is not relevant.

Note that there is no ambiguity between an $(\ell,k,\lambda)$-graph and an $(m,\ell,k)$-graph (and similarly for complexes) as $\lambda<1$.

Suppose $n\geq \ell\geq k$ and
suppose $H$ is an $n$-vertex $k$-graph and $F$ is an $\ell$-vertex $k$-graph.
We define $\mathbf{Pr}(F,H)$ 
such that $\mathbf{Pr}(F,H)\binom{n}{\ell}$ equals the number of induced copies of $F$ in~$H$. 
For a collection $\cF$ of $\ell$-vertex $k$-graphs, 
we define $\mathbf{Pr}(\cF, H)$ 
such that $\mathbf{Pr}(\cF,H)\binom{n}{\ell}$ equals the number of induced $\ell$-vertex $k$-graphs $F$ in $H$ such that $F\in \cF$.

\subsection{Hypergraph regularity}\label{sec: 2 hypergraph regularity}
In this subsection we introduce $\epsilon$-regularity for hypergraphs.
Suppose $\ell \geq k\geq 2$ and $V_1,\dots, V_{\ell}$ are pairwise disjoint vertex sets.
Let $H^\kk$ be an $(\ell,k,*)$-graph on $\{V_1,\dots, V_{\ell}\}$, 
let $\{i_1,\dots, i_{k} \} \in \binom{[\ell]}{k}$, 
and let $H^{(k-1)}$ be a $(k,k-1,*)$-graph on $\{V_{i_1},\dots, V_{i_k}\}$. 
We define the \emph{density of $H^\kk$ with respect to $H^{(k-1)}$} as
$$d(H^{(k)} \mid H^{(k-1)}) := \left\{ \begin{array}{ll} \frac{|H^{(k)}\cap \cK_k(H^{(k-1)})|}{|\cK_k(H^{(k-1)})| } & \text{ if } |\cK_k(H^{(k-1)})|>0, \\
0 & \text{otherwise.} 
\end{array}\right.$$
Suppose $\epsilon>0$ and $d\geq 0$.
We say $H^{(k)}$ is \emph{$(\epsilon,d)$-regular with respect to $H^{(k-1)}$} 
if for all $Q^{(k-1)}\subseteq H^{(k-1)}$ with  
$$| \cK_k(Q^{(k-1)})| \geq \epsilon |\cK_k(H^{(k-1)})|, 
\text{ we have } |H^{(k)}\cap \cK_k(Q^{(k-1)})| = (d\pm \epsilon) |\cK_k(Q^{(k-1)})|.$$\COMMENT{We don't want to write $d(H^{(k)}\mid Q^{(k-1)}) = d\pm \epsilon$ here since we later want that if $\cK_{k}(H^{(k-1)})=\emptyset$ then any $H^{(k)}$  is $(\epsilon,d)$-regular with respect to $H^{(k-1)}$ for all $d,\epsilon \in [0,1]$.}
Note that if $H^{(k)}$ is $(\epsilon,d)$-regular with respect to $H^{(k-1)}$ and $H^{(k-1)}\neq \emptyset$, then we have $d(H^{(k)} \mid H^{(k-1)}) = d\pm \epsilon$. We say $H^{(k)}$ is \emph{$\epsilon$-regular with respect to $H^{(k-1)}$} if it is $(\epsilon,d)$-regular with respect to $H^{(k-1)}$ for some $d\geq 0$.

We say an $(\ell,k,*)$-graph $H^{(k)}$ 
on $\{V_1,\dots, V_\ell\}$ is \emph{$(\epsilon,d)$-regular with respect to an $(\ell,k-1,*)$-graph $H^{(k-1)}$} 
on $\{V_1,\dots, V_\ell\}$
if for every $\Lambda \in \binom{[\ell]}{k}$ 
$H^{(k)}$ is $(\epsilon,d)$-regular with respect to the restriction $H^{(k-1)}[\Lambda]$. 

Let $\bd=(d_2,\dots, d_{k})\in \R^{k-1}_{\geq 0}$.
We say an $(\ell,k,*)$-complex $\cH = \{H^{(j)}\}_{j=1}^{k}$ is \emph{$(\epsilon,\bd)$-regular} if $H^{(j)}$ is $(\epsilon,d_j)$-regular with respect to $H^{(j-1)}$ for every $j\in [k]\sm \{1\}$. 
We sometimes simply refer to a complex as being $\epsilon$-regular if it is $(\epsilon,\bd)$-regular for some vector $\bd$.

\subsection{Partitions of hypergraphs and the regular approximation lemma}\label{sec: partitions of hypergraphs and RAL}

The regular approximation lemma of R\"odl and Schacht implies that 
for all $k$-graphs $H$, there exists a $k$-graph $G$
which is very close to $H$ and so that $G$ has a very `high quality' partition into $\epsilon$-regular subgraphs.
To state this formally we need to introduce further concepts involving partitions of hypergraphs.

Suppose $A\supseteq B$ are finite sets, 
$\sA$ is a partition of $A$, 
and $\sB$ is a partition of $B$.
We say $\sA$ \emph{refines} $\sB$ and write $\sA \prec \sB$ 
if for every $\cA\in \sA$ there either exists $\cB \in \sB$ such that $\cA \subseteq \cB$ or $\cA \subseteq A\setminus B$.
The following definition concerns `approximate' refinements.
Let $\nu\geq 0$.
We say that $\sA$ \emph{$\nu$-refines} $\sB$ and write $\sA \prec_{\nu} \sB$ 
if there exists a function $f: \sA \rightarrow \sB\cup \{A\setminus B\}$ such that 
$$\sum_{\cA \in \sA}|\cA \setminus f(\cA)| \leq \nu |A|.$$
We make the following observations.
\begin{equation}\label{eq: prec triangle}
\begin{minipage}[c]{0.9\textwidth}\em
\begin{itemize}
\item  $\sA\prec \sB$ if and only if $\sA\prec_{0} \sB$. 
\item Suppose $\sA,\sA',\sA''$ are partitions of $A,A',A''$ respectively and $A''\subseteq A' \subseteq A$. 
If $\sA \prec_{\nu} \sA'$ and $\sA' \prec_{\nu'} \sA''$, 
then $\sA\prec_{\nu+\nu'}\sA''$.
\end{itemize} 
\end{minipage}\ignorespacesafterend 
\end{equation}

We now introduce the concept of a polyad.
Roughly speaking, given a vertex partition $\sP^{(1)}$, an $i$-polyad is an $i$-graph which arises from a partition $\sP^{(i)}$ of the complete partite $i$-graph $\cK_i(\sP^{(1)})$.
The $(i+1)$-cliques spanned by all the $i$-polyads give rise to a partition $\sP^{(i+1)}$ of $\cK_{i+1}(\sP^{(1)})$ (see Definition~\ref{def: family of partitions}).
Such a `family of partitions' then provides a suitable framework for describing a regularity partition (see Definition~\ref{def: equitable family of partitions}).

Suppose we have a vertex partition $\sP^{(1)} = \{V_1,\ldots,V_\ell\}$ and $\ell\geq k$.
For integers $k\leq \ell'\leq \ell$, we say that a hypergraph $H$ is an \emph{$(\ell',k,*)$-graph with respect to $\sP^{(1)}$} if it is an $(\ell',k,*)$-graph on $\{ V_i : i\in \Lambda\}$ for some $\Lambda \in \binom{[\ell]}{\ell'}$.

Recall that $\cK_{j}(\sP^{(1)})$ is the family of all crossing $j$-sets with respect to $\sP^{(1)}$.
Suppose that for all  $i\in[k-1]\sm \{1\}$, 
we have partitions $\sP^{(i)}$ of $\cK_{i}(\sP^{(1)})$ such that each part of $\sP^{(i)}$ is an $(i,i,*)$-graph with respect to $\sP^{(1)}$.
By definition,
for each $i$-set $I\in \cK_{i}(\sP^{(1)})$, 
there exists exactly one $P^{(i)}=P^{(i)}(I)\in \sP^{(i)}$ so that $I\in P^{(i)}$.
Consider $j\in [\ell]$ and any $J\in \cK_j(\sP^{(1)})$.
For each $i\in[ \max\{j,k-1\}]$,
the \emph{$i$-polyad} $\hat{P}^{(i)}(J)$ of $J$ is defined by
\begin{align}\label{def:polyad}
	\hat{P}^{(i)}(J) := \bigcup\left\{P^{(i)}(I) : I\in \binom{J}{i}\right\}.
\end{align}
Thus $\hat{P}^{(i)}(J)$ is a $(j,i)$-graph with respect to $\sP^{(1)}$.
Moreover, let
\begin{align}\label{eq: hat cP}
\hat{\cP}(J):=\left\{\hat{P}^{(i)}(J)\right\}_{i=1}^{\max\{j,k-1\}},\end{align}
\COMMENT{Until now, this does not have to be complex, since we do not know that $\cK_{i}(\hat{P}^{(i-1)})$ underlies $P^{(i)} \in \sP^{(i)}$ or not.
Because since $\sP$ is arbitrary, $\hat{P}^{(i)}(J) = \hat{P}^{(i)}(J')$ while $\hat{P}^{(i-1)}(J) \neq \hat{P}^{(i-1)}(J')$ might happen.
}
and for $j\in [k-1]$, let
\begin{align}\label{eq:sP}
	\hat{\sP}^{(j)} := \left\{\hat{P}^{(j)}(J): J\in \cK_{j+1}(\sP^{(1)})\right\}.
\end{align}
We note that $\hat{\sP}^{(1)}$ is the set consisting of all $(2,1)$-graphs with vertex classes $V_s, V_t$ (for all distinct $s,t\in [\ell]$). In other words, each element of $\hat{\sP}^{(1)}$ is a $1$-graph with the vertex set $V_s\cup V_t$ and the edge set $V_s\cup V_t$. 
Also note that for any $\hat{P}^{(j)}\in \hat{\sP}^{(j)}$, we have $\cK_{j+1}(\hat{P}^{(j)})\neq \emptyset$. Indeed, if $\hat{P}^{(j)} \in \hat{\sP}^{(j)}$, it follows that there is a set $J\in \cK_{j+1}(\sP^{(1)})$ such that $\hat{P}^{(j)}=\hat{P}^{(j)}(J)$ and $J\in \cK_{j+1}(\hat{P}^{(j)}(J))$.


The above definitions apply to arbitrary partitions $\sP^{(i)}$ of $\cK_i(\sP^{(1)})$.
However, it will be useful to consider partitions with more structure.
\begin{definition}[Family of partitions]\label{def: family of partitions}
Suppose $k\in \N\sm \{1\}$ and $\ba=(a_1,\dots, a_{k-1})\in \N^{k-1}$.
We say $\sP= \sP(k-1,\ba)= \{\sP^{(1)},\dots, \sP^{(k-1)}\}$ is a \emph{family of partitions on $V$} 
if it satisfies the following for each $j\in [k-1]\setminus \{1\}$:
\begin{enumerate}[label={ \rm(\roman*)}]
\item $\sP^{(1)}$ is a partition of $V$ into $a_1 \geq k$ nonempty classes,
\item $\sP^{(j)}$ is a partition of $\cK_{j}(\sP^{(1)})$ into nonempty\COMMENT{Here, I added the word `nonempty' because of the injectivity of $\ba$-labelling $\phi^{(j)}$ we use in section 2.6.1. Since the parts of $\sP^{(k-1)}$ are non-empty this also implies that $a_1\geq k-1$.} 
$j$-graphs such that
\begin{itemize}
	\item $\sP^{(j)}\prec \{\cK_j(\hat{P}^{(j-1)}): \hat{P}^{(j-1)}\in \hat{\sP}^{(j-1)}\}$ and
	\item $|\{P^{(j)}\in \sP^{(j)} : P^{(j)} \subseteq \cK_j(\hat{P}^{(j-1)})\}|=a_j$ for every $\hat{P}^{(j-1)}\in \hat{\sP}^{(j-1)}$.
\end{itemize}
\end{enumerate}
\end{definition}
We say $\sP = \sP(k-1,\ba)$ is \emph{$T$-bounded} if $\norm{\ba}\leq T$. 
For two families of partitions $\sP = \sP(k-1,\ba^{\sP})$ and $\sQ =\sQ(k-1,\ba^{\sQ})$, we say \emph{$\sP \prec \sQ$} if $\sP^{(j)} \prec \sQ^{(j)}$ for all $j\in [k-1]$. We say \emph{$\sP \prec_{\nu} \sQ$} if $\sP^{(j)} \prec_{\nu} \sQ^{(j)}$ for all $j\in [k-1]$.

As the concept of polyads is central to this paper,
we emphasize the following:
\begin{proposition}\label{prop: (i,i)-hypergraph}
Let $k\in \mathbb{N}\setminus \{1\}$, $\ba\in \mathbb{N}^{(k-1)}$ and $\sP = \sP(k-1,\ba)$ be a family of partitions. Then for all $i\in [k-1]$ and $j\in [a_1]$, the following hold.
\begin{enumerate}[label={ \rm(\roman*)}]
\item if $i>1$, then $\sP^{(i)}$ is a partition of $\cK_{i}(\sP^{(1)})$ into $(i,i,*)$-graphs with respect to $\sP^{(1)}$, 
\item each $\hat{P}^{(i)}\in \hat{\sP}^{(i)}$ is an $(i+1,i,*)$-graph with respect to $\sP^{(1)}$,
\item for each $j$-set $J\in \cK_j(\sP^{(1)})$, $\hat{\cP}(J)$ as defined in \eqref{eq: hat cP} is a complex.
\end{enumerate}
\end{proposition}

We now extend the concept of $\epsilon$-regularity to families of partitions.
\begin{definition}[Equitable family of partitions]\label{def: equitable family of partitions}
Let $k\in \N\sm \{1\}$.
Suppose $\eta>0$ and $\ba=(a_1,\dots, a_{k-1})\in \N^{k-1}$.
Let $V$ be a vertex set of size $n$.
We say a family of partitions $\sP= \sP(k-1,\ba)$ on $V$ is \emph{$(\eta,\epsilon,\ba,\lambda)$-equitable} if it satisfies the following:
\begin{enumerate}[label={ \rm(\roman*)}]
\item $a_1 \geq \eta^{-1}$,
\item $\sP^{(1)} = \{V_i : i\in [a_1]\}$ satisfies $|V_i|= (1\pm \lambda)n/a_1$ for all $i\in[a_1]$, and
\item if $k\geq 3$, then for every $k$-set $K\in \cK_{k}(\sP^{(1)})$ 
the collection $\hat{\cP}(K)=\{\hat{P}^{(j)}(K)\}_{j=1}^{k-1}$ is an $(\epsilon,\bd)$-regular $(k,k-1,*)$-complex, where $\bd= (1/a_2,\dots, 1/a_{k-1})$. 
\end{enumerate}
\end{definition}
As before we drop $\lambda$ if $|V_i| \in \{ \lfloor n/a_1\rfloor, \lfloor n/a_1\rfloor +1\}$ and say $\sP$ is $(\eta,\epsilon,\ba)$-equitable.
Note that for any $\lambda\leq 1/3$, every $(\eta,\epsilon,\ba,\lambda)$-equitable family of partitions $\sP$ satisfies
\begin{align}\label{eq: eta a1}
\left|\binom{V}{k}\setminus \cK_{k}(\sP^{(1)})\right| \leq k^2 \eta \binom{n}{k}.
\end{align}
\COMMENT{A $k$-set $K$ is not in $\cK_{k}(\sP^{(1)})$ if $|K\cap V_i|\geq 2$ for some $i\in [a_1]$.
Thus we choose $i$ with $\leq a_1$ ways, and choose two vertices in $\binom{V_i}{2}$ with $\leq \binom{(1+\lambda)n/a_1}{2}$ ways, and choose $k-2$ other vertices arbitrarily with $\leq \binom{n}{k-2}$ ways.
Thus we have 
$$\left|\binom{V}{k}\setminus \cK_{k}(\sP^{(1)})\right| \leq a_1 \cdot \binom{(1+\lambda)n/a_1}{2} \cdot \binom{n}{k-2} \leq  \frac{1}{2 a_1} (1+\lambda)^2k(k-1)\binom{n}{k} \leq k^2 \eta \binom{n}{k}.$$
}

We next introduce the concept of perfect $\epsilon$-regularity with respect to a family of partitions.
\begin{definition}[Perfectly regular]
Suppose $\epsilon>0$ and $k\in \N\sm\{1\}$. 
Let $H^{(k)}$ be a $k$-graph with vertex set $V$ and let $\sP= \sP(k-1,\ba)$ be a family of partitions on $V$. 
We say $H^{(k)}$ is perfectly $\epsilon$-regular with respect to $\sP$ 
if for every $\hat{P}^{(k-1)}\in \hat{\sP}^{(k-1)}$ the graph $H^{(k)}$ is $\epsilon$-regular with respect to $\hat{P}^{(k-1)}$.
\end{definition}

Having introduced the necessary notation, we are now ready to state the regular approximation lemma due to R\"odl and Schacht.
It states that for every $k$-graph $H$, 
there is a $k$-graph $G$ that is close to $H$ and that has very good regularity properties.

\begin{theorem}[Regular approximation lemma~\cite{RS07}]\label{thm: RAL}
Let $k\in \N\sm\{1\}$. 
For all $\eta,\nu>0$ and every function $\epsilon: \mathbb{N}^{k-1}\rightarrow (0,1]$, 
there are integers $t_0:= t_{\ref{thm: RAL}}(\eta,\nu,\epsilon)$ and $n_0:=n_{\ref{thm: RAL}}(\eta,\nu,\epsilon)$ 
so that the following holds:

For every $k$-graph $H$ on at least $n\geq n_0$ vertices,
there exists a $k$-graph $G$ on $V(H)$ and a family of partitions $\sP=\sP(k-1,\ba^{\sP})$ on $V(H)$ so that 
\begin{enumerate}[label={\rm (\roman*)}]
\item $\sP$ is $(\eta,\epsilon(\ba^{\sP}),\ba^{\sP})$-equitable and $t_0$-bounded,
\item $G$ is perfectly $\epsilon(\ba^{\sP})$-regular with respect to $\sP$, and
\item $|G\triangle H|\leq \nu \binom{n}{k}$.
\end{enumerate}
\end{theorem}

The crucial point here is that in applications we may apply Theorem~\ref{thm: RAL} with a function $\epsilon$ such that $\epsilon(\ba^{\sP})\ll \norm{\ba^\sP}^{-1}$.
This is in contrast to other versions (see e.g.~\cite{Gow07,RS04,Tao06}) where (roughly speaking) in (iii) we have $G=H$ but in (ii) we have an error parameter $\epsilon'$ which may be large compared to $\norm{\ba^\sP}^{-1}$.

\subsection{The address space}\label{sec: address space}
Later on, we will need to explicitly refer to the densities arising, for example, in Theorem~\ref{thm: RAL}(ii).
For this (and other reasons) it is convenient to consider the `address space'.
Roughly speaking the address space consists of a collection of vectors where each vector identifies a polyad.

For $a,s\in \N$, 
we recursively define 
$[a]^{s}$ by $[a]^{s} := [a]^{s-1}\times [a]$ and $[a]^{1}:= [a]$. 
To define the address space,
let us write $\binom{[a_1]}{\ell}_{<} := \{(\alpha_1,\dots, \alpha_\ell)\in[a_1]^\ell: \alpha_1 < \dots <\alpha_\ell \}$.

Suppose $k',\ell,p\in \N$, $\ell\geq k'$, and $p\geq\max\{ k'-1,1\}$, and $\ba=(a_1,\dots,a_p )\in \N^{p}$. 
We define 
$$\hat{A}(\ell,k'-1,\ba):= 
\binom{[a_1]}{\ell}_{<} \times \prod_{j=2}^{k'-1}[a_j]^{\binom{\ell}{j}}$$ to be the \emph{$(\ell,k')$-address space}. 
Observe that $\hat{A}(1,0,\ba)=[a_1]$ and $\hat{A}(2,1,\ba)=\binom{[a_1]}{2}_{<}$. 
Recall that for a vector $\bx$, the set $\bx_*$ was defined at the beginning of Section~\ref{sec:concepts}.
Note that if $k'>1$, then each $\hat{\bx}\in \hat{A}(\ell,k'-1,\ba)$ can be written as $\hat{\bx} = (\bx^{(1)},\dots, \bx^{(k'-1)})$, 
where $\bx^{(1)}\in \binom{[a_1]}{\ell}_{<}$ and $\bx^{(j)} \in [a_j]^{\binom{\ell}{j}}$ for each $j\in [k'-1]\sm\{1\}$. 
Thus each entry of the vector $\bx^{(j)}$ corresponds to (i.e.~is indexed by) a subset of $\binom{[\ell]}{j}$. 
We order the elements of both $\binom{[\ell]}{j}$ and $\binom{\bx_*^{(1)}}{j}$ lexicographically and consider the bijection $g: \binom{\bx_*^{(1)}}{j} \rightarrow \binom{[\ell]}{j}$ which preserves this ordering.
For each $\Lambda \in \binom{\bx^{(1)}_*}{j}$ and $j\in [k'-1]$, 
we denote by $\bx^{(j)}_{\Lambda}$ the entry of $\bx^{(j)}$ which corresponds to the set $g(\Lambda)$.

\subsubsection{Basic properties of the address space}\label{sec: Basic properties of the address space}
Let $k\in\N\sm\{1\}$ and let $V$ be a vertex set of size $n$. 
Let $\sP(k-1,\ba)$ be a family of partitions on $V$. 
For each crossing $\ell$-set $L\in \cK_{\ell}(\sP^{(1)})$, 
the address space allows us to identify (and thus refer to) the set of polyads `supporting' $L$. 
We will achieve this by defining a suitable operator $\hat{\bx}(L)$ which maps $L$ to the address space.

To do this, write $\sP^{(1)}= \{V_i: i\in [a_1]\}$.  
Recall from Definition~\ref{def: family of partitions}(ii) that for $j\in[k-1]\setminus\{1\}$, we partition $\cK_j(\hat{P}^{(j-1)})$ of every $(j-1)$-polyad $\hat{P}^{(j-1)}\in\hat{\sP}^{(j-1)}$ into $a_j$ nonempty parts in such a way that $\sP^{(j)}$ is the collection of all these parts.
Thus, there is a labelling $\phi^{(j)}:\sP^{(j)} \rightarrow [a_j]$ such that for every polyad $\hat{P}^{(j-1)} \in \hat{\sP}^{(j-1)}$, the restriction of $\phi^{(j)}$ to $\{P^{(j)} \in \sP^{(j)}: P^{(j)}\subseteq \cK_j(\hat{P}^{(j-1)})\}$ is injective.\COMMENT{ To make this true, we added the work `nonempty' in Definition~\ref{def: family of partitions} which ensures that every element of $\sP^{(j)}$ are all distinct. }
The set $\mathbf{\Phi} := \{\phi^{(2)},\dots, \phi^{(k-1)}\}$ is called an \emph{$\ba$-labelling} of $\sP(k-1,\ba)$.
For a given set $L\in \cK_{\ell}(\sP^{(1)})$, we denote $\cl(L):= \{ i: V_i\cap L \neq \emptyset\}.$

Consider any $\ell \in [a_1]$. 
Let $j':=\min\{k-1,\ell-1\}$  and let $j'':=\max\{j',1\}$. 
For every $\ell$-set $L\in \cK_{\ell}(\sP^{(1)})$ we define an integer vector $\hat{\bx}(L) = (\bx^{(1)}(L),\dots, \bx^{(j'')}(L))$ by
\begin{equation}\label{eq: bx J def}
\begin{minipage}[c]{0.9\textwidth}\em
\begin{itemize}
\item $\bx^{(1)}(L) := (\alpha_1,\ldots,\alpha_\ell)$, where  $\alpha_1<\ldots<\alpha_\ell$ and $L \cap V_{\alpha_i}=\{v_{\alpha_i}\}$,
\item and for $i\in [j']\setminus\{1\}$ we set
$$\bx^{(i)}({L}) := \left(\phi^{(i)}(P^{(i)}): \{v_{\lambda} :\lambda \in \Lambda\}\in P^{(i)}, P^{(i)}\in \sP^{(i)}\right)_{\Lambda \in \binom{\cl(L)}{i}}.$$
\end{itemize}
\end{minipage}
\end{equation}
Here, we order $\binom{\cl(L)}{i}$ lexicographically. In particular, $\bx^{(i)}(L)$ is a vector of length $\binom{\ell}{i}$.

By definition, $\hat{\bx}(L) \in \hat{A}(\ell,j',\ba)$ for every $L\in \cK_{\ell}(\sP^{(1)})$ with $\ell, j'$ as above. 
Our next aim is to define an operator $\hat{\bx}(\cdot)$ which maps the set $\hat{\sP}^{(j-1)}$ of $(j-1)$-polyads injectively into the address space $\hat{A}(j,j-1,\ba)$ (see \eqref{eq: hat bx maps}). We will then extend this further into a bijection between elements of the address spaces and their corresponding hypergraphs.
However, before we can define $\hat{\bx}(\cdot)$, we need to introduce some more notation.

Suppose $j\in [k'-1]$. 
For $\hat{\bx} \in \hat{A}(\ell,k'-1,\ba)$ and $J\in \cK_{j}(\sP^{(1)})$ with $\cl(J) \subseteq \bx^{(1)}_{*}$, we define $\bx^{(j)}_{J} := \bx^{(j)}_{\cl(J)}$. Thus from now on, we may refer to the entries of $\bx^{(j)}$ either by an index set $\Lambda \in \binom{\bx^{(1)}_*}{j}$ or by a set $J \in \cK_{j}(\sP^{(1)})$.

Next we introduce a relation on the elements of (possibly different) address spaces.
Consider $\hat{\bx}= (\bx^{(1)}, \bx^{(2)},\dots,\bx^{(k'-1)}) \in \hat{A}(\ell,k'-1,\ba)$ with $\ell'\leq \ell$ and $k''\leq k'$.
We define $\hat{\by} \leq_{\ell',k''-1} \hat{\bx}$ if
\begin{itemize}
\item $\hat{\by} = (\by^{(1)}, \by^{(2)},\dots,\by^{(k''-1)})\in \hat{A}(\ell', k''-1,\ba)$,
\item $\by^{(1)}_*\subseteq \bx^{(1)}_*$ and 
\item $\bx^{(j)}_{\Lambda}=\by^{(j)}_{\Lambda}$ for any $\Lambda \in \binom{\by^{(1)}_*}{j}$ and $j\in [k''-1]\sm\{1\}$.
 \end{itemize}
Thus any $\hat{\by}\in \hat{A}(\ell', k''-1,\ba)$ with $\hat{\by} \leq_{\ell',k''-1} \hat{\bx}$ 
can be viewed as the restriction of $\hat{\bx}$ to an $\ell'$-subset of the $\ell$-set $\bx^{(1)}_*$.
Hence for $\hat{\bx} \in \hat{A}(\ell,k'-1,\ba)$, there are exactly $\binom{\ell}{\ell'}$ distinct integer vectors $\hat{\by}\in \hat{A}(\ell',k''-1,\ba)$ such that $\hat{\by} \leq_{\ell',k''-1} \hat{\bx}.$ 
Also it is easy to check the following properties.
\begin{proposition}\label{eq: I subset J then leq}
Suppose $\sP= \sP(k-1,\ba)$ is a family of partitions, $i\in [a_1]$ and $i':= \min\{i,k\}$. 
\begin{enumerate}[label={\rm (\roman*)}]
	\item Whenever $I\in \cK_{i}(\sP^{(1)})$ and $J\in \cK_{j}(\sP^{(1)})$ with $I\subseteq J$, then $\hat{\bx}(I) \leq_{i,i'-1} \hat{\bx}(J)$.
	\item If $J \in \cK_j(\sP^{(1)})$ and $\hat{\by} \leq_{i,i'-1} \hat{\bx}(J)$, then there exists a unique $I\in \binom{J}{i}$ such that $\hat{\by} = \hat{\bx}(I)$.
\end{enumerate}
\end{proposition}

Now we are ready to introduce the promised bijection between the elements of address spaces and their corresponding hypergraphs.

Consider $j\in [k]\setminus\{1\}$. Recall that for every $j$-set $J\in \cK_{j}(\sP^{(1)})$, 
we have $\hat{\bx}(J) \in \hat{A}(j,j-1,\ba)$. 
Moreover, recall that $\cK_{j}(\hat{P}^{(j-1)})\neq \emptyset$ for any $\hat{P}^{(j-1)}\in \hat{\sP}^{(j-1)}$, and note that $\hat{\bx}(J) = \hat{\bx}(J')$ for all $J,J' \in \cK_{j}(\hat{P}^{(j-1)})$ and all $\hat{P}^{(j-1)}\in \hat{\sP}^{(j-1)}$.
Hence, for each $\hat{P}^{(j-1)} \in \hat{\sP}^{(j-1)}$ we can define 
\begin{align}\label{eq: hat bx maps}
\hat{\bx}(\hat{P}^{(j-1)}):= \hat{\bx}(J)\text{ for some } J\in \cK_{j}(\hat{P}^{(j-1)}).
\end{align}
Let 
\begin{align*}
\hat{A}(j,j-1,\ba)_{\neq \emptyset} &:= \{ \hat{\bx}\in \hat{A}(j,j-1,\ba): \exists \hat{P}^{(j-1)}\in \hat{\sP}^{(j-1)} \text{ such that }\hat{\bx}(\hat{P}^{(j-1)}) = \hat{\bx}\}\\
&= \{ \hat{\bx} \in \hat{A}(j,j-1,\ba) : \exists J\in \cK_j(\sP^{(1)}) \text{ such that } \hat{\bx} = \hat{\bx}(J)\},\\
\hat{A}(j,j-1,\ba)_{\emptyset}&:= \hat{A}(j,j-1,\ba)\setminus \hat{A}(j,j-1,\ba)_{\neq \emptyset}.
\end{align*}
Clearly \eqref{eq: hat bx maps} gives rise to a bijection between $\hat{\sP}^{(j-1)}$ and $\hat{A}(j,j-1,\ba)_{\neq \emptyset}$.
Thus for each $\hat{\bx} \in \hat{A}(j,j-1,\ba)_{\neq \emptyset}$, we can define the polyad  $\hat{P}^{(j-1)}(\hat{\bx})$ of $\hat{\bx}$ by
\begin{align}\label{eq: hat P def}
\hat{P}^{(j-1)}(\hat{\bx}):= \hat{P}^{(j-1)} \text{ such that }\hat{P}^{(j-1)} \in \hat{\sP}^{(j-1)} \text{ with } \hat{\bx}= \hat{\bx}(\hat{P}^{(j-1)}).
\end{align}
Note that for any $J\in \cK_j(\sP^{(1)})$, we have $\hat{P}^{(j-1)}(\hat{\bx}(J)) = \hat{P}^{(j-1)}(J).$ 

We will frequently make use of an explicit description of a polyad in terms of the partition classes it contains (see \eqref{eq:hatPconsistsofP(x,b)}). For this, we proceed as follows.
For each $b\in [a_1]$, let $P^{(1)}(b,b):= V_b$.
For each $j\in [k-1]\setminus\{1\}$ and $(\hat{\bx},b)  \in \hat{A}(j,j-1,\ba)_{\neq\emptyset}\times [a_j]$, we let 
\begin{align}\label{eq: Pj operator define}
	P^{(j)}(\hat{\bx},b):=P^{(j)}\in \sP^{(j)} \text{ such that } \phi^{(j)}(P^{(j)})=b \text{ and } P^{(j)} \subseteq \cK_{j}(\hat{P}^{(j-1)}(\hat{\bx})). 
\end{align}
Using Definition~\ref{def: family of partitions}(ii), 
we conclude that so far $P^{(j)}(\hat{\bx},b)$ is well-defined for each $(\hat{\bx},b)  \in \hat{A}(j,j-1,\ba)_{\neq\emptyset}\times [a_j]$ and all $j \in [k-1]\setminus\{1\}$.

For convenience we now extend the domain of the above definitions to cover the `trivial' cases.
For $(\hat{\bx},b) \in  \hat{A}(j,j-1,\ba)_{\emptyset}\times [a_j]$, we let
\begin{align}\label{eq: Pj operator define empty}
P^{(j)}(\hat{\bx},b) := \emptyset.
\end{align}
We also let $P^{(1)}(a,b):= \emptyset$ for all $a,b\in [a_1]$ with $a\neq b$.
For all $j\in [k-1]$ and $\hat{\bx}\in \hat{A}(j+1,j,\ba)_{\emptyset}$, we define
\begin{align}\label{eq: hat P P relations emptyset}
\hat{P}^{(j)}(\hat{\bx}):= 
\bigcup_{\hat{\by}\leq_{j,j-1} \hat{\bx}} P^{(j)}(\hat{\by},\bx^{(j)}_{\by^{(1)}_* }).
\end{align}

To summarize, given a family of partitions $\sP=\sP(k-1,\ba)$ and an $\ba$-labelling $\mathbf{\Phi}$, 
for each $j\in [k-1]$, this defines $ P^{(j)}(\hat{\bx},b)$ for $\hat{\bx}\in \hat{A}(j,j-1,\ba)$ and $b\in [a_j]$ and $\hat{P}^{(j)}(\hat{\bx})$ for all $\hat{\bx}\in \hat{A}(j+1,j,\ba)$.
For later reference, we collect the relevant properties of these objects below.
For each $j\in [k-1]\setminus \{1\}$, it will be convenient to extend the domain of the $\ba$-labelling $\phi^{(j)}$ of $\sP^{(j)}$ to all $j$-sets $J\in \cK_j(\sP^{(1)})$ by
setting $\phi^{(j)}(J) := \phi^{(j)}(P^{(j)})$, where $P^{(j)}\in \sP^{(j)}$ is the unique $j$-graph that contains $J$.

\begin{proposition}\label{prop: hat relation}
For a given family of partitions $\sP=\sP(k-1,\ba)$ and an $\ba$-labelling~$\mathbf{\Phi}$, the following hold for all $j\in [k-1]$.
\begin{enumerate}[label={\rm (\roman*)}]
\item $\hat{P}^{(j)}(\cdot):\hat{A}(j+1,j,\ba)_{\neq\emptyset}\to \hat{\sP}^{(j)}$ is a bijection.
\item For $j\geq 2$, the restriction of $P^{(j)}(\cdot,\cdot)$ onto $\hat{A}(j,j-1,\ba)_{\neq \emptyset} \times [a_j]$ is a bijection onto $\sP^{(j)}$.

\item $\hat{\bx} \in \hat{A}(j+1,j,\ba)_{\neq\emptyset}$ if and only if $\cK_{j+1}(\hat{P}^{(j)}(\hat{\bx}))\neq \emptyset$.

\item Each $\hat{\bx}\in \hat{A}(j+1,j,\ba)$ satisfies
\begin{align}\label{eq:hatPconsistsofP(x,b)}
\hat{P}^{(j)}(\hat{\bx})= 
\bigcup_{\hat{\by}\leq_{j,j-1} \hat{\bx}} P^{(j)}(\hat{\by},\bx^{(j)}_{\by^{(1)}_* }).
\end{align}
\item $\{P^{(j)}(\hat{\bx},b) : \hat{\bx}\in \hat{A}(j,j-1,\ba), b\in [a_j]\}$ forms a partition of $\cK_{j}(\sP^{(1)})$.

\item $\{\cK_{j+1}(\hat{P}^{(j)}(\hat{\bx})):\hat{\bx}\in \hat{A}(j+1,j,\ba)\}$ forms a partition of $\cK_{j+1}(\sP^{(1)})$.
\item $\{P^{(j+1)}(\hat{\bx},b) : \hat{\bx}\in \hat{A}(j+1,j,\ba), b\in [a_{j+1}]\}\prec \{\cK_{j+1}(\hat{P}^{(j)}(\hat{\bx})):\hat{\bx}\in \hat{A}(j+1,j,\ba)\}$.
\item If $\sP(k-1,\ba)$ is $T$-bounded, then $|\hat{\sP}^{(j)}|\leq |\hat{A}(j+1,j,\ba)|\leq T^{2^{j+1}-1}$ and $|\sP^{(j)}| \leq T^{2^{j}}$.

\item If $\cK_{j+1}(\hat{P}^{(j)}(\hat{\bx}))\neq \emptyset$ for all $\hat{\bx}\in \hat{A}(j+1,j,\ba)$, then $\hat{P}^{(j)}(\cdot):\hat{A}(j+1,j,\ba)\to \hat{\sP}^{(j)}$ is a bijection and, if in addition $j<k-1$, then $P^{(j+1)}(\cdot,\cdot):\hat{A}(j+1,j,\ba)\times [a_{j+1}]\to \sP^{(j+1)}$ is also a bijection.

\item $\hat{A}(j,j-1,\ba)_{\emptyset} = \emptyset$ for all $j\in [2]$ and thus $\hat{P}^{(1)}(\cdot)$ and $P^{(2)}(\cdot,\cdot)$ are always bijections.

\item If $\sP \prec \sQ(k-1,\ba^{\sQ})$, then $\{ \cK_{j+1}(\hat{P}^{(j)}(\hat{\bx})) : \hat{\bx}\in \hat{A}(j+1,j,\ba) \} \prec \{ \cK_{j+1}(\hat{Q}^{(j)}(\hat{\bx})) : \hat{\bx}\in \hat{A}(j+1,j,\ba^{\sQ})\}$.
\end{enumerate}
\end{proposition}

We postpone the proof of Proposition~\ref{prop: hat relation} to Section~\ref{sec: more concepts and tools}.


 We remark that the counting lemma (see Lemma~\ref{lem: counting}) will enable us to restrict our attention to families of partitions as in Proposition~\ref{prop: hat relation}(ix). This is formalized in Lemma~\ref{lem: maps bijections}.

For $j\in [k-1]$, $\ell \geq j+1$ and for each $\hat{\bx} \in \hat{A}(\ell,j,\ba)$, 
we define the \emph{polyad of $\hat{\bx}$} by 
\begin{align}\label{eq: larger polyad def}
\hat{P}^{(j)}(\hat{\bx}) 
:= \bigcup_{\hat{\by}\leq_{j+1,j} \hat{\bx}} \hat{P}^{(j)}(\hat{\by})
\stackrel{(\ref{eq:hatPconsistsofP(x,b)})}{=}\bigcup_{\hat{\bz}\leq_{j,j-1} \hat{\bx}} P^{(j)}(\hat{\bz},\bx^{(j)}_{\bz^{(1)}_* }).
\end{align}\COMMENT{Here, we have this because
$\hat{\bz}\leq_{j,j-1} \hat{\by} \leq_{j+1,j} \hat{\bx}$ implies $\hat{\bz}\leq_{j,j-1}\hat{\bx}$.}
(Note that this generalizes the definition made in \eqref{eq: hat P def} for the case $\ell=j+1$.)
The following fact follows easily from the definition.
\begin{proposition}\label{prop:unique}
Let $\sP= \sP(k-1,\ba)$ be a family of partitions. 
Let $j\in [k-1]$ and $\ell\geq j+1$. Then for every $L\in \cK_{\ell}(\sP^{(1)})$, 
there exists a unique $\hat{\bx}\in \hat{A}(\ell,j,\ba)$ such that $L\in \cK_{\ell}(\hat{P}^{(j)}(\hat{\bx}))$.
\end{proposition}
\COMMENT{
\begin{proof}
Indeed, recall that for every $J\in \binom{L}{j+1}$, 
the vector $\hat{\bx}(J)$ is the unique element of $\hat{A}(j+1,j,\ba)$ which satisfies $J\in \cK_{j+1}(\hat{P}^{(j)}(\hat{\bx}(J)))$. 
Let $\hat{\bx}\in \hat{A}(\ell,j,\ba)$ be the vector which satisfies $\bx^{(i)}_{I'}= (\bx^{(i)}(J))_{I'}$ for every $I'\subsetneq J$ with $|I'|=i\leq j$ and every $J\in \binom{L}{j+1}$.\COMMENT{Here, subscript $I'$ denotes the coordinate of the vector.}
Then it is easy to see that the vector $\hat{\bx}= (\bx^{(1)},\dots, \bx^{(j)})$ is well-defined and $\hat{\bx}(J)\leq_{j+1,j} \hat{\bx}$ for all $J\in \binom{L}{j+1}$.
Thus by \eqref{eq: larger polyad def}, every $J\in \binom{L}{j+1}$ satisfies $J\in \cK_{j+1}(\hat{P}^{(j)}(\hat{\bx}))$, and hence $L\in \cK_{\ell}(\hat{P}^{(j)}(\hat{\bx}))$. 
This proves our claim. 
\end{proof}
}
Note that \eqref{eq: Pj operator define} and \eqref{eq: larger polyad def} together imply that, for all $j\in [k-1]$ and $\hat{\bx}\in \hat{A}(j+1,j,\ba)$,
\begin{align}\label{eq: complex definition by address}
\hat{\cP}(\hat{\bx}) := \left\{ \bigcup_{\hat{\by}\leq_{j+1,i} \hat{\bx}} \hat{P}^{(i)}(\hat{\by})\right\}_{i\in [j]}
\end{align} is a $(j+1,j,*)$-complex.\COMMENT{
If it is a complex, then it is easy to see that it is a $(j+1,j)$-complex.
Suppose $\hat{\by} \leq_{j+1,i} \hat{\bx}$ for $i\leq j$.
It suffice to show that if $J\in \hat{P}^{(i)}(\hat{\by})$, then $J\in \cK_{i}( \hat{P}^{(i-1)}(\hat{\bz}))$ for some $\hat{\bz}$ with $\hat{\bz} \leq_{j+1,i-1} \hat{\bx}$. \newline
Suppose $J\in \hat{P}^{(i)}(\hat{\by})$.
By \eqref{eq: larger polyad def}, $J\in P^{(i)}(\hat{\bz}, \by^{(i)}_{\bz^{(1)}_*})$ for some $\hat{\bz}$ with $\hat{\bz}\leq_{i,i-1} \hat{\by}$.
By the definition of $ P^{(i)}(\hat{\bz}, \by^{(i)}_{\bz^{(1)}_*})$, $J$ belongs to $\cK_{i}(\hat{P}^{(i-1)}(\hat{\bz}))$. Thus we get what we want.
}
Moreover, using Proposition~\ref{prop: hat relation}(iii) it is easy to check that for each $\hat{\bx}\in \hat{A}(j+1,j,\ba)$ with $\cK_{j+1}(\hat{P}^{(j)}(\hat{\bx}))\neq \emptyset$, we have (for $\hat{\cP}(J)$ as defined in~\eqref{eq: hat cP})
\begin{align}\label{eq: hat cP bx is hat cP J is for some J}
\hat{\cP}(\hat{\bx}) = \hat{\cP}(J)\text{ for some } J\in \cK_{j+1}(\sP^{(1)}).
\end{align}

\medskip

\subsubsection{Density functions of address spaces}\label{sec: subsub density function}

For $k\in \N\sm\{1\}$ 
and $\ba\in \N^{k-1}$,
we say a function $d_{\ba,k}:\hat{A}(k,k-1,\ba)\rightarrow[0,1]$ is a \emph{density function} of $\hat{A}(k,k-1,\ba)$.
For two density functions $d^1_{\ba,k}$ and $d^2_{\ba,k}$,
we define the \emph{distance} between $d^1_{\ba,k}$ and $d^2_{\ba,k}$ by
$$\dist(d^1_{\ba,k}, d^2_{\ba,k}) 
:= k!\prod_{i=1}^{k-1} a_i^{-\binom{k}{i}}\sum_{\hat{\bx}\in \hat{A}(k,k-1,\ba)} |d^1_{\ba,k}(\hat{\bx})-d^2_{\ba,k}(\hat{\bx})| .$$
Since $|\hat{A}(k,k-1,\ba)| = \binom{a_1}{k} \prod_{i=2}^{k-1} a_i^{\binom{k}{i}}$, we always have that $\dist(d^1_{\ba,k}, d^2_{\ba,k})\leq 1.$
Suppose we are given a density function $d_{\ba,k}$, a real $\epsilon>0$, and a $k$-graph $H^\kk$.
We say a family of partitions $\sP=\sP(k-1,\ba)$ on $V(H^{(k)})$ is an \emph{$(\epsilon,d_{\ba,k})$-partition} of $H^{(k)}$ 
if for every $\hat{\bx}\in \hat{A}(k,k-1,\ba)$ the $k$-graph
$H^{(k)}$ is $(\epsilon,d_{\ba,k}(\hat{\bx}))$-regular with respect to $\hat{P}^{(k-1)}(\hat{\bx})$. 
If $\sP$ is also $(1/a_1,\epsilon,\ba)$-equitable (as specified in Definition~\ref{def: equitable family of partitions}), 
we say $\sP$ is an \emph{$(\epsilon,\ba,d_{\ba,k})$-equitable partition} of $H^{(k)}$. Note that
\begin{equation}\label{eq: perfectly regular is regular}
\begin{minipage}[c]{0.9\textwidth}\em 
 if $\hat{P}^{(k-1)}(\cdot): \hat{A}(k,k-1,\ba)\rightarrow \hat{\sP}^{(k-1)}$ is a bijection, then $H^{(k)}$ is perfectly $\epsilon$-regular with respect to $\sP$ if and only if there exists a density function $d_{\ba,k}$ such that $\sP$ is an $(\epsilon,d_{\ba,k})$-partition of $H^{(k)}$.
\end{minipage}
\end{equation}

\subsection{Regularity instances}\label{sec: regularity instances}
A regularity instance $R$ encodes an address space, an associated density function and a regularity parameter.
Roughly speaking,
a regularity instance can be thought of as encoding a weighted `reduced multihypergraph' obtained from an application of the regularity lemma for hypergraphs.
To formalize this,
let $\epsilon_{\ref{def: regularity instance}}(\cdot,\cdot) : \mathbb{N} \times \mathbb{N} \rightarrow (0,1]$ be a function which satisfies the following.
\begin{itemize}
\item $\epsilon_{\ref{def: regularity instance}}( \cdot, k)$ is a decreasing function for any fixed  $k\in \mathbb{N}$ with $\lim_{t\rightarrow \infty} \epsilon_{\ref{def: regularity instance}}(t, k) = 0$, 
\item $\epsilon_{\ref{def: regularity instance}}( t, \cdot)$ is a decreasing function for any fixed $t\in \mathbb{N}$,
\item $\epsilon_{\ref{def: regularity instance}}(t,k) < t^{- 4^k} \epsilon_{\ref{lem: counting}}(1/t,1/t,k-1,k)/4$, where $\epsilon_{\ref{lem: counting}}
$ is defined in Lemma~\ref{lem: counting}.
\end{itemize}\COMMENT{Here, note that we divide $\epsilon_{\ref{lem: counting}}(1/t,1/t,k-1,k)$ by $4 t^{4^k}$. There are two reason for this. We divide it by three because
since we want to apply Lemma~\ref{lem: counting} even when $(\epsilon,\ba,d_{\ba,k})$ is not a regularity instance but  $(\epsilon/3,\ba,d_{\ba,k})$ is a regularity instance. For this purposes, in some lemmas we assume $(\epsilon/3, \ba, d_{\ba,k})$ is a regularity instance, instead of $(\epsilon,\ba,d_{\ba,k})$ is a regularity instance.
We also divide it by $t^{4^{k}}$ to make sure that $\epsilon $ is much less than $1/t$.
}
\begin{definition}[Regularity instance]\label{def: regularity instance}
A \emph{regularity instance} $R= (\epsilon,\ba,d_{\ba,k})$ 
is a triple, where
$\ba = (a_1,\dots, a_{k-1})\in \N^{k-1}$ with $0<\epsilon \leq \epsilon_{\ref{def: regularity instance}}(\|\ba\|_\infty,k)$,\COMMENT{Here, the condition that $\epsilon \ll \epsilon_{\ref{def: regularity instance}}(t)$ is required. If $\epsilon$ is too big compare to $1/t$, we lose the tack of $|\cK_{j}(\hat{P}(\hat{\bx}))|$ for each $\hat{\bx}\in \hat{A}(j,j-1,\ba)$. Then it causes a problem in the proof of Lemma~\ref{lem: slightly different partition regularity} } 
and $d_{\ba,k}$ is a density function of $\hat{A}(k,k-1,\ba)$. 
A $k$-graph $H$ satisfies 
the regularity instance $R$ if there exists a family of partitions $\sP=\sP(k-1,\ba)$ such that 
$\sP$ is an $(\epsilon,\ba,d_{\ba,k})$-equitable partition of~$H$.
The \emph{complexity} of $R$ is $1/\epsilon$.
\end{definition}

Since $\epsilon_{\ref{def: regularity instance}}$ depends only on $\norm{\ba}$ and $k$, it follows that  for given $r$ and fixed $k$, the number of vectors $\ba$ which could belong to a regularity instance $R$ with complexity $r$ is bounded by a function of $r$.

\COMMENT{

\begin{definition}[Regular reducible]\label{def: regular reducible}
A $k$-graph property $\bP$ is \emph{regular reducible} if for any $\beta>0$, 
there exists an $r=r_{\ref{def: regular reducible}}(\beta,\bP)$ such that 
for any integer $n\geq k$, 
there is a family $\cR=\cR(n,\beta,\bP)$ of at most $r$ regularity instances, each of complexity at most $r$,
such that the following hold for every $\alpha>\beta$ and every $n$-vertex $k$-graph $H$:
\begin{itemize}
\item If $H$ satisfies $\bP$, then there exists $R\in \cR$ such that $H$ is $\beta$-close to satisfying $R$.
\item If $H$ is $\alpha$-far from satisfying $\bP$, then for any $R\in \cR$  the $k$-graph $H$ is $(\alpha-\beta)$-far from satisfying $R$.
\end{itemize}
\end{definition}

Thus a property is regular reducible if it can be (approximately) encoded by a bounded number of regularity instances of bounded complexity.}
We will often make use of the fact that if we apply the regular approximation lemma (Theorem~\ref{thm: RAL}) to a $k$-graph $H$
to obtain $G$ and $\sP$,
then $\ba^\sP$ together with the densities of $G$ with respect to the polyads in $\hat{\sP}^{(k-1)}$ naturally give rise to a regularity instance $R$
where $G$ satisfies $R$ and $H$ is close to satisfying $R$.
\COMMENT{
Note that different choices of $\epsilon_{\ref{def: regularity instance}}$ lead to a different definition of regularity instances and thus might lead to a different definition of being regular reducible. However, our main result implies that for \emph{any} appropriate choice of $\epsilon_{\ref{def: regularity instance}}$, being regular reducible and testability are equivalent. In particular, if a property is regular reducible for an appropriate choice of $\epsilon_{\ref{def: regularity instance}}$, then it is regular reducible for all appropriate choices of $\epsilon_{\ref{def: regularity instance}}$, and so `regular reducibility' is well defined. 
}

\section{Main result: Sampling regularity instances}

Let us now turn to the statement of our main result
thereby extending Theorem~\ref{thm:main_graphs} to $k$-graphs.
It states that not too small random samples of vertex subsets satisfy with high probability essentially the same regularity instance;
that is, only an additive error term is needed.

\begin{theorem}\label{lem: random choice2}
	Suppose $0< 1/n< 1/q \ll c \ll \delta \ll \epsilon_0 \leq 1$ and $k\in \N\sm \{1\}$.
	Suppose $R=(2\epsilon_0/3, \ba, d_{\ba,k})$ is a regularity instance.
	Suppose $H$ is a $k$-graph on vertex set $V$ with $|V|=n$. 
	Let $Q \in \binom{V}{q}$ be chosen uniformly at random. 
	Then with probability at least $ 1 - e^{-cq}$ the following hold.
	\begin{itemize}
		\item[{\rm (Q1)$_{\ref{lem: random choice2}}$}] If there exists an $(\epsilon_0,\ba,d_{\ba,k})$-equitable partition $\sO_1$ of $H$, 
		then there exists an $(\epsilon_0+\delta,\ba,d_{\ba,k})$-equitable partition $\sO_2$ of $H[Q]$.
		\item[{\rm(Q2)$_{\ref{lem: random choice2}}$}] If there exists an $(\epsilon_0,\ba,d_{\ba,k})$-equitable partition $\sO_2$ of $H[Q]$, 
		then there exists  an $(\epsilon_0+\delta,\ba,d_{\ba,k})$-equitable partition $\sO_1$ of $H$.
	\end{itemize}
\end{theorem}

We use Theorem~\ref{lem: random choice2} to completely characterize all testable hypergraph properties in the companion paper~\cite{JKKO1}.
Similar in the graph setting, regular instances are the key objects and, roughly speaking, a property is testable if regularity instances determine the property.



\medskip
Now we illustrate a few key points in our approach.
Roughly speaking, Theorem~\ref{lem: random choice2} states the following.
\begin{equation*}
	\begin{minipage}[c]{0.90\textwidth}\em
		Suppose $H$ is a $k$-graph and $Q$ a random subset of $V(H)$.
		Then with high probability, the following hold  (where $\delta \ll \epsilon_0$).
		\begin{itemize}
			\item If $\sO_1$ is an $\epsilon_0$-equitable partition of $H$ with density function $d_{\ba,k}$,
			then there is an $(\epsilon_0+\delta)$-equitable partition of $H[Q]$ with the same density function~$d_{\ba,k}$.
			\item If $\sO_2$ is an $\epsilon_0$-equitable partition of $H[Q]$ with density function $d_{\ba,k}$,
			then there is an $(\epsilon_0+\delta)$-equitable partition of $H$ with the same density function~$d_{\ba,k}$.
		\end{itemize}
	\end{minipage}
\end{equation*}
The crucial point here is that the transfer between $H$ and $H[Q]$ incurs only an additive increase in the regularity parameter $\epsilon_0$.
In fact, this additive increase can then be eliminated by slightly adjusting $H$ (or $H[Q]$).

The key ingredient  in the proof of Theorem~\ref{lem: random choice2} is Lemma~\ref{lem: similar}. 
Roughly speaking, Lemma~\ref{lem: similar} states the following.
\begin{equation*}
	\begin{minipage}[c]{0.90\textwidth}\em
		Suppose the following hold (where $\epsilon\ll \delta\ll \epsilon_0$).
		\begin{itemize}
			\item $H_1$ is a $k$-graph on vertex set $V_1$ and $\sQ_1$ is an $\epsilon$-equitable partition of $H_1$ with density function $d_{\ba^{\sQ},k}$.
			\item $H_2$ is a $k$-graph on vertex set $V_2$ and $\sQ_2$ is an $\epsilon$-equitable partition of $H_2$ with the same density function $d_{\ba^{\sQ},k}$.
			\item $\sO_1$ is an $\epsilon_0$-equitable partition of $H_1$ with density function $d_{\ba^\sO,k}$.
		\end{itemize}
		Then there is an $(\epsilon_0+\delta)$-equitable partition $\sO_2$ of $H_2$,
		also with density function~$d_{\ba^\sO,k}$.
	\end{minipage}
\end{equation*}
One may think of this results as follows;
if two $k$-graphs both satisfy \emph{some} `high quality' regularity partition with the same parameters,
then \emph{all} `low quality' regularity partitions from one $k$-graph are also regularity partitions of the other $k$-graph to the expense of only a small additive increase in the regularity parameter. 

To prove Theorem~\ref{lem: random choice2},
we will apply Lemma~\ref{lem: similar} with $H$ playing the role of $H_1$ and with
the random sample $H[Q]$ playing the role of $H_2$ (and vice versa).
In turn,
our strengthening of the regular approximation lemma (Lemma~\ref{lem:refreg}) will be one of the main tools in the proof of Lemma~\ref{lem: similar}; see the beginning of Section~\ref{sec: Sampling a regular partition} for a more detailed sketch.

\section{More concepts and tools}\label{sec: more concepts and tools}

Here we collect some further results that we need later in our proofs but which are not needed to understand our main theorem.

\subsection{A stronger hypergraph regularity lemma}

We next state Lemma~\ref{RAL(k)} which is a generalization of the regular approximation lemma
which was also proved by R\"odl and Schacht (see Lemma~25 in~\cite{RS07}).
Lemma~\ref{RAL(k)} 
has two additional features in comparison to Theorem~\ref{thm: RAL}.
Firstly, we can prescribe a family of partitions $\sQ$ and obtain a refinement $\sP$ of $\sQ$,
and secondly, we are not only given one $k$-graph $H$ but a collection of $k$-graphs $H_i$ that partitions the complete $k$-graph.
Thus we may view Lemma~\ref{RAL(k)} as a `partition version' of Theorem~\ref{thm: RAL}.
We will use it in the proof of Lemma~\ref{lem:refreg}.

\begin{lemma}[R\"odl and Schacht~\cite{RS07}]\label{RAL(k)}
	For all $o,s\in \N$, $k\in \N\sm \{1\}$, all $\eta,\nu>0$, 
	and every function $\epsilon : \mathbb{N}^{k-1}\rightarrow (0,1]$, 
	there are $\mu=\mu_{\ref{RAL(k)}}(k,o,s,\eta,\nu,\epsilon)>0$ 
	and $t=t_{\ref{RAL(k)}}(k,o,s,\eta,\nu,\epsilon)\in \N$ and $n_0=n_{\ref{RAL(k)}}(k,o,s,\eta,\nu,\epsilon)\in\N$ such that the following hold. 
	Suppose
	\begin{itemize}
		\item[{\rm (O1)$_{\ref{RAL(k)}}$}] $V$ is a set and $|V|=n\geq n_0$,
		\item[{\rm (O2)$_{\ref{RAL(k)}}$}] $\sQ=\sQ(k,\ba^{\sQ})$ is a $(1/a_1^{\sQ},\mu,\ba^{\sQ})$-equitable $o$-bounded family of partitions on $V$,
		\item[{\rm (O3)$_{\ref{RAL(k)}}$}] $\sH^{(k)}=\{H^{(k)}_1,\dots,H^{(k)}_{s}\}$ is a partition of $\binom{V}{k}$ so that $\sH^{(k)} \prec \sQ^{(k)}$.
	\end{itemize}
	Then there exist a family of partitions $\sP = \sP(k-1,\ba^{\sP})$ and a partition $\sG^{(k)}= \{ G^{(k)}_1,\dots, G^{(k)}_s\}$ of $\binom{V}{k}$ satisfying the following for every $i\in [s]$ and $j\in [k-1]$.
	\begin{itemize}
		\item[{\rm (P1)$_{\ref{RAL(k)}}$}] $\sP$ is a $t$-bounded $(\eta,\epsilon(\ba^{\sP}),\ba^{\sP})$-equitable family of partitions, and $a^{\sQ}_j$ divides $a^{\sP}_j$, 
		\item[{\rm (P2)$_{\ref{RAL(k)}}$}] $\sP \prec \{\sQ^{(j)}\}_{j=1}^{k-1}$,
		\item[{\rm (P3)$_{\ref{RAL(k)}}$}] $G^{(k)}_i$ is perfectly $\epsilon(\ba^{\sP})$-regular with respect to $\sP$,
		\item[{\rm (P4)$_{\ref{RAL(k)}}$}] $\sum_{i=1}^{s} |G^{(k)}_i\triangle H^{(k)}_i| \leq \nu \binom{n}{k}$, and
		\item[{\rm (P5)$_{\ref{RAL(k)}}$}] $\sG^{(k)} \prec \sQ^{(k)}$ and if $H^{(k)}_i \subseteq \cK_k(\sQ^{(1)})$, then $G_i^{(k)} \sub \cK_k(\sQ^{(1)})$.
	\end{itemize}
\end{lemma}

In Lemma~\ref{RAL(k)} we may assume without loss of generality that $1/\mu,t,n_0$ are non-decreasing in $k,o,s$ and non-increasing in $\eta, \nu$.

\subsection{The proof of Proposition~\ref{prop: hat relation}}

\begin{proof}[Proof of Proposition~\ref{prop: hat relation}]
	Observe that (i) and (ii) hold by definition.
	Note that $\hat{\bx} \in \hat{A}(j+1,j,\ba)_{\neq\emptyset}$ if and only if $\hat{\bx} = \hat{\bx}(J)$ for some $J \in \cK_{j+1}(\sP^{(1)})$ if and only if there exists a set $J\in \cK_{j+1}(\hat{P}^{(j)}(\hat{\bx}))$.%
	\COMMENT{
		Claim: if $\hat{\bx}\in \hat{A}(j+1,j,\ba)$ and $J\in \cK_{j+1}(\hat{P}^{(j)}(\hat{\bx}))$,
		then $\hat{\bx}=\hat{\bx}(J)$ and $\hat{\bx}\in \hat{A}(j+1,j,\ba)_{\neq \es}$.\newline
		Proof by induction on $j$.\newline
		$j=1$ is clear. \newline
		Suppose now that $j>1$. 
		Suppose also that $\hat{\bx}\in \hat{A}(j+1,j,\ba)_{\neq \es}$.
		Thus \eqref{eq: hat P def} implies that $\hat{P}^{(j)}= \hat{P}^{(j)}(\hat{\bx})\in \hat{\sP}^{(j)}$ and $\hat{\bx} = \hat{\bx}(\hat{P}^{(j)})$. 
		Thus \eqref{eq: hat bx maps} implies that 
		$\hat{\bx} =  \hat{\bx}(\hat{P}^{(j)}) = \hat{\bx}(J)$. Thus $\hat{\bx} = \hat{\bx}(J)$ for some $J \in \cK_{j+1}(\sP^{(1)})$.
		So suppose form now on that $\hat{\bx}\in \hat{A}(j+1,j,\ba)_{=\es}$.
		Then by \eqref{eq: hat P P relations emptyset} for each $\hat{\by}\leq_{j,j-1}\hat{\bx}$
		there exists $I\in \binom{J}{j}$ with $I\in P^{(j)}(\hat{\by},\hat{\bx}^{(j)}_{\by_*^{(1)}})$.
		Thus $\hat{\by}\in \hat{A}(j+1,j,\ba)_{\neq \es}$ by \eqref{eq: Pj operator define empty}.
		But also $I\in P^{(j)}(\hat{\by},\hat{\bx}^{(j)}_{\by_*^{(1)}})\subseteq \cK_j(\hat{P}^{(j-1)}(\hat{\by}))$.
		So by induction $\hat{\by}=\hat{\by}(I)$.
		In particular, $\by_*^{(1)}=\cl(I)$.
		Let $P^{(j)}\in \sP^{(j)}$ be such that $I\in P^{(j)}$.
		Then $P^{(j)}=P^{(j)}(\hat{\by},\hat{\bx}^{(j)}_{\by_*^{(1)}})$
		and so $\phi(P^{(j)})=\hat{\bx}^{(j)}_{\by_*^{(1)}}=\hat{\bx}^{(j)}_{\cl(I)}$.
		Thus $\hat{\bx}=\hat{\bx}(J)$,
		a contradiction.
	}
	Thus (iii) holds.
	To show (iv), by \eqref{eq: hat P P relations emptyset}, we may assume $\hat{\bx}\in \hat{A}(j+1,j,\ba)_{\neq \emptyset}$.
	Thus we know that $\cK_{j+1}(\hat{P}^{(j)}(\hat{\bx}))$ contains at least one $(j+1)$-set $J$ and $\hat{\bx}=\hat{\bx}(J)$.
	By \eqref{def:polyad}, we have 
	$$\hat{P}^{(j)}(\hat{\bx}) = \hat{P}^{(j)}(J)= \bigcup_{I\in \binom{J}{j}} P^{(j)}(I).$$
	By Proposition~\ref{eq: I subset J then leq}(ii), we know that
	$\hat{\by} \leq_{j,j-1} \hat{\bx}$ if and only if $\hat{\by}= \hat{\bx}(I)$ for some $I\in \binom{J}{j}$.
	Consider any $j$-set $I\subseteq J$. Recall that $\hat{P}^{(j-1)}(\hat{\bx}(I)) = \hat{P}^{(j-1)}(I)$, and thus $I\in \cK_j(\hat{P}^{(j-1)}(\hat{\bx}(I)))$. 
	Together with \eqref{eq: Pj operator define} this implies that $P^{(j)}(I) = P^{(j)}( \hat{\bx}(I), \phi^{(j)}(I))$, where $P^{(j)}(I)$ is the unique part of $\sP^{(j)}$ that contains $I$. Since $\phi^{(j)}(I) = \phi^{(j)}(P^{(j)}(I)) = \bx^{(j)}(J)_{I}$ holds by \eqref{eq: bx J def}, we have
	$$ 
	\hat{P}^{(j)}(\hat{\bx}) = \bigcup_{I\in \binom{J}{j}} P^{(j)}(I)
	=  \bigcup_{I\in \binom{J}{j}} P^{(j)}(\hat{\bx}(I) , \bx^{(j)}(J)_{I})
	= \bigcup_{\hat{\by}\leq_{j,j-1} \hat{\bx}} P^{(j)}(\hat{\by},\bx^{(j)}_{\by^{(1)}_* }).$$
	This shows that (iv) holds.
	It is easy to see that (i), (ii), (iii) and Definition~\ref{def: family of partitions}(ii) together imply (v), (vi) and (vii).
	If $\sP(k-1,\ba)$ is $T$-bounded, (i) implies that
	\begin{align*}
		|\hat{\sP}^{(j)}| \leq |\hat{A}(j+1,j,\ba)| \leq \prod_{i=1}^{j} a_i^{\binom{j+1}{i}} \leq \prod_{i=1}^{j} T^{\binom{j+1}{i} }\leq T^{2^{j+1}-1}.
	\end{align*}
	Thus for $j\in [k-1]\setminus \{1\}$, we have $|\sP^{(j)}| \leq a_{j} |\hat{\sP}^{(j-1)}| \leq T^{2^{j}}$.
	Also $|\sP^{(1)}| = a_1 \leq T$, thus we have (viii). 
	Statement (ix) follows from (i), (ii) and (iii).
	Property (x) is trivial from the definitions.
	
	Finally we show (xi). Suppose $J \in \cK_{j+1}( \hat{P}^{(j)}(\hat{\bx})) \cap \cK_{j+1}(\hat{Q}^{(j)}(\hat{\by}))$ for some $\hat{\bx}\in \hat{A}(j+1,j,\ba)$ and $\hat{\by}\in \hat{A}(j+1,j,\ba^{\sQ})$. 
	Then (iii) implies that $\hat{P}^{(j)}(\hat{\bx}) = \hat{P}^{(j)}(J)$ and $\hat{Q}^{(j)}(\hat{\by})= \hat{Q}^{(j)}(J)$. Since $\sP\prec \sQ$, we have $P^{(j)}(I)\subseteq Q^{(j)}(I)$. Thus 
	$$\hat{P}^{(j)}(\hat{\bx})  \stackrel{\eqref{def:polyad}}{=} \bigcup_{I\in \binom{J}{j}} P^{(j)}(I) \subseteq  \bigcup_{I\in \binom{J}{j}} Q^{(j)}(I) \stackrel{\eqref{def:polyad}}{=}  \hat{Q}^{(j)}(\hat{\by}).$$
	Thus we have $\cK_{j+1}(\hat{P}^{(j)}) \subseteq \cK_{j+1}(\hat{Q}^{(j)}(J))$. This implies (xi).
\end{proof}

\subsection{Constructing families of partitions using the address space}
On several occasions we will construct $P^{(j)}(\hat{\bx},b)$ and $\hat{P}^{(j)}(\hat{\bx})$ first and 
then  show that they actually give rise to a family of partitions
for which we can use the properties listed in Proposition~\ref{prop: hat relation}.
The following lemma,
which can easily be proved by induction, 
provides a criterion to show that this is indeed the case.

\begin{lemma}\label{lem: family of partitions construction}
	Suppose $k\in \N\sm\{1\}$ and $\ba\in \N^{k-1}$. Suppose $\sP^{(1)}= \{V_1,\dots, V_{a_1}\}$ is a partition of a vertex set $V$.
	Suppose that for each $j\in [k-1]\setminus\{1\}$ and each $(\hat{\bx},b)\in \hat{A}(j,j-1,\ba) \times [a_j]$, 
	we are given a $j$-graph $P'^{(j)}(\hat{\bx},b)$, and
	for each $j\in [k]\setminus \{1\}$ and $\hat{\bx}\in \hat{A}(j,j-1,\ba)$,
	we are given a $(j-1)$-graph $\hat{P'}^{(j-1)}(\hat{\bx})$. 
	Let 
	\begin{align*} P'^{(1)}(b,b)&:=V_b \text{ for all } b\in [a_1], \text{ and }\\ 
		\sP^{(j)}&:= \{P'^{(j)}(\hat{\bx},b) :(\hat{\bx},b)\in \hat{A}(j,j-1,\ba) \times [a_j]\} \text{ for all } j \in [k-1]\setminus\{1\}.
	\end{align*}
	Suppose the following conditions hold: 
	\begin{enumerate}[label=\rm (FP\arabic*)]
		\item\label{item:FP1} $P'^{(1)}(b,b) \neq \emptyset$ for each $b\in [a_1]$; moreover for each $j\in [k-1]\setminus \{1\}$ and each $(\hat{\bx},b)\in \hat{A}(j,j-1,\ba) \times [a_j]$, we have $P'^{(j)}(\hat{\bx},b)\neq \emptyset$.
		
		\item\label{item:FP2} For each $j\in [k-1]\setminus\{1\}$ and $\hat{\bx}\in \hat{A}(j,j-1,\ba)$,
		the set $\{P'^{(j)}(\hat{\bx},b):b\in [a_j]\}$ has size $a_j$\COMMENT{
			Need this to get that the $P'^{(j)}(\hat{\bx},b)$ are distinct.} and forms a partition of $\cK_j(\hat{P'}^{(j-1)}(\hat{\bx}))$.
		\item\label{item:FP3} For each $j\in [k-1]$ and $\hat{\bx}\in \hat{A}(j+1,j,\ba)$, we have 
		$$\hat{P'}^{(j)}(\hat{\bx})= 
		\bigcup_{\hat{\by}\leq_{j,j-1} \hat{\bx}} P'^{(j)}(\hat{\by},\bx^{(j)}_{\by^{(1)}_* }).$$
	\end{enumerate}
	Then the $\ba$-labelling $\mathbf{\Phi} = \{\phi^{(i)}\}_{i=2}^{k-1}$ given by $\phi^{(i)}(P'^{(i)}(\hat{\bx},b))=b$ for each $(\hat{\bx},b)\in \hat{A}(i,i-1,\ba)\times[a_i]$ is well-defined and satisfies the following:
	\begin{enumerate}[label={\rm (FQ\arabic*)}]
		\item\label{item:FQ1} $\sP = \{\sP^{(i)}\}_{i=1}^{k-1}$ is a family of partitions on $V$.
		\item\label{item:FQ2} The maps $P^{(j)}(\cdot,\cdot)$ and $\hat{P}^{(j)}(\cdot)$ defined in \eqref{eq: hat P def}--\eqref{eq: hat P P relations emptyset} for $\sP$, $\mathbf{\Phi}$ satisfy that for each $j\in [k-1]\setminus\{1\}$ and $(\hat{\bx},b)\in \hat{A}(j,j-1,\ba) \times [a_j]$, we have 
		$$P^{(j)}(\hat{\bx},b) = P'^{(j)}(\hat{\bx},b),$$ 
		and for each $j \in [k-1]$ and $\hat{\bx} \in \hat{A}(j+1,j,\ba)$ we have
		$$\hat{P}^{(j)}(\hat{\bx}) = \hat{P'}^{(j)}(\hat{\bx}).$$
	\end{enumerate}
\end{lemma}

\section{Hypergraph regularity: counting lemmas and approximation}\label{sec:counting}

In this section we present several results about hypergraph regularity.
The first few results are simple observations
which follow either from the definition of $\epsilon$-regularity or can be easily proved by standard probabilistic arguments.
We omit the proofs.
In Section~\ref{sec:3.2}
we then derive an induced version of the `counting lemma' that is suitable for our needs (see Lemma~\ref{lem: counting hypergraph}).

In Section~\ref{sec:3.4} we make two simple observations on refinements of partitions and in Section~\ref{sec:3.5} we consider small perturbations of a given family of partitions. 

\subsection{Simple hypergraph regularity results}
We will use the following results which follow easily from the definition of hypergraph regularity (see Section~\ref{sec: 2 hypergraph regularity}).

\begin{lemma}\label{lem: simple facts 1}
Suppose $m \in \mathbb{N}$,  $0< \epsilon \leq \alpha^2 <1$ and $d\in [0,1]$.
Suppose $H^{(k)}$ is an $(m,k,k,1/2)$-graph which is $(\epsilon,d)$-regular  with respect to an $(m,k,k-1,1/2)$-graph $H^{(k-1)}$.
Suppose $Q^{(k-1)} \subseteq H^{(k-1)}$ and $H'^{(k)}\subseteq H^{(k)}$ such that $|\cK_k(Q^{(k-1)})| \geq \alpha |\cK_k(H^{(k-1)})|$ and $H'^{(k)}$ is $(\epsilon,d')$-regular with respect to $H^{(k-1)}$ for some $d'\leq d$.  Then 
\begin{enumerate}[label=\rm(\roman*)]
\item $\cK_k(H^{(k-1)}) \setminus H^{(k)}$ is $(\epsilon,1-d)$-regular with respect to $H^{(k-1)}$, 
\item $H^{(k)}$ is $(\epsilon/\alpha,d)$-regular with respect to $Q^{(k-1)}$, and 
\item $H^{(k)}\setminus H'^{(k)}$ is $(2\epsilon,d-d')$-regular with respect to $H^{(k-1)}$.
\end{enumerate}
\end{lemma}

\begin{lemma}[Union lemma]\label{lem: union regularity}
Suppose $0<\epsilon \ll 1/k,1/s.$
Suppose that $H^{(k)}_1,\dots, H^{(k)}_s$ are edge-disjoint $(k,k,*)$-graphs such that each $H^{(k)}_i$ is $\epsilon$-regular with respect to a $(k,k-1,*)$-graph $H^{(k-1)}$. 
Then $\bigcup_{i=1}^{s} H^{(k)}_i$ is $s\epsilon$-regular with respect to $H^{(k-1)}$.
\end{lemma}

We will also use the following observation (see for example \cite{RS07}), which can be easily proved using Chernoff's inequality.

\begin{lemma}[Slicing lemma \cite{RS07}]\label{lem: slicing}
Suppose $0< 1/m \ll d, \epsilon , p_0, 1/s$ and $d\geq 2\epsilon$. Suppose that
\begin{itemize}
\item $H^{(k)}$ is an $(\epsilon,d)$-regular $k$-graph with respect to a $(k-1)$-graph $H^{(k-1)}$,
\item $|\cK_{k}(H^{(k-1)})| \geq  m^{k}/\log m$,
\item $p_1,\dots, p_s \geq p_0$ and $\sum_{i=1}^{s} p_i \leq 1$.
 \end{itemize}
 Then 
there exists a partition $\{H^{(k)}_0,H^{(k)}_1,\dots, H^{(k)}_s\}$ of $H^{(k)}$ such that $H^{(k)}_i$ is $(3\epsilon,p_id)$-regular with respect to $H^{(k-1)}$ for every $i\in [s]$, and $H^{(k)}_0$ is $(3\epsilon,(1-\sum p_i)d)$-regular with respect to $H^{(k-1)}$.
\end{lemma}

\subsection{Counting lemmas}\label{sec:3.2}
Kohayakawa, R\"odl and Skokan proved the following `counting lemma' (Theorem~6.5 in~\cite{KRS02}), which asserts that the number of copies of a given $K^{(k)}_{\ell}$ in an $(\epsilon,\bd)$-regular complex is close to what one could expect in a corresponding random complex. We will deduce several versions of this which suit our needs.

\begin{lemma}[Counting lemma  \cite{KRS02}]\label{lem: counting}
For all $ \gamma, d_0>0$ and  $k, \ell \in \N\sm \{1\}$ with $k\leq \ell$,\COMMENT{We don't use the hierarchy notation here since we refer to $\epsilon_{\ref{lem: counting}}$ in the definition of regularity instance.}
there exist
$\epsilon_0 := \epsilon_{\ref{lem: counting}}(\gamma,d_0,k,\ell)\leq 1$ and $m_0:= n_{\ref{lem: counting}}(\gamma,d_0,k,\ell)$
such that the following holds:
Suppose $0\leq \lambda <1/4$.
Suppose 
$0<\epsilon \leq \epsilon_0$ and $m_0 \leq m$ and
$\bd=(d_2,\dots, d_{k})\in \mathbb{R}^{k-1}$ such that 
$d_j\geq d_0$ for every $j\in [k]\setminus\{1\}$.
Suppose that
$\cH= \{H^{(j)}\}_{j=1}^{k}$ is an $(\epsilon,\bd)$-regular $(m,\ell,k,\lambda)$-complex, 
and $H^{(1)}=\{V_1,\ldots,V_\ell\}$ with $m_i=|V_i|$ for every $i\in [\ell]$.
Then 
$$|\cK_{\ell}(H^{(k)})| = (1\pm \gamma) \prod_{j=2}^{k}d_j^{\binom{\ell}{j}} \cdot \prod_{i=1}^{\ell} m_i.$$
\end{lemma}
\COMMENT{It is possible to delete the condition $1/m \ll \epsilon$ from all counting lemmas.}

Recall that equitable families of partitions were defined in Section~\ref{sec: partitions of hypergraphs and RAL}. 
Based on the counting lemma, 
it is easy to show that for an equitable family of partitions $\sP$ and an $\ba$-labelling $\mathbf{\Phi}$, 
the maps $\hat{P}^{(j-1)}(\cdot):\hat{A}(j,j-1,\ba)\to\hat{\sP}^{(j-1)}$ and $P^{(j)}(\cdot,\cdot):\hat{A}(j,j-1,\ba)\times [a_j]\to \sP^{(j)}$ defined in Section~\ref{sec: address space} are bijections.
We will frequently make use of this fact in subsequent sections,
often without referring to Lemma~\ref{lem: maps bijections} explicitly.

\begin{lemma}\label{lem: maps bijections}
Suppose that $k,t \in \mathbb{N}\setminus \{1\}$, $0\leq \lambda < 1/4$ and $\epsilon/3 \leq \epsilon_{\ref{def: regularity instance}}(t,k)$ and $\ba = (a_1,\dots, a_{k-1}) \in [t]^{k-1}$ and $|V|=n$ with $1/n \ll 1/t, 1/k$.\COMMENT{In other words,  $R=(\epsilon/3, \ba, d_{\ba,k})$ is a regularity instance for any density function $d_{\ba,k}$.}  If $\sP = \sP(k-1,\ba)$ is a $(1/a_1,  \epsilon,\ba,\lambda)$-equitable family of partitions on $V$, and $\sP$ with an $\ba$-labelling $\mathbf{\Phi}$ defines maps $\hat{P}^{(j-1)}(\cdot)$ and $P^{(j-1)}(\cdot,\cdot)$, then the following hold.
\begin{enumerate}[label={\rm(\roman*)}]
\item For each $j\in [k-1]$, $\hat{P}^{(j)}(\cdot):\hat{A}(j+1,j,\ba)\to\hat{\sP}^{(j)}$ is a bijection and if $j>1$, then $P^{(j)}(\cdot,\cdot):\hat{A}(j,j-1,\ba)\times [a_j]\to \sP^{(j)}$ is also a bijection. In particular, 
$\hat{A}(j,j-1,\ba) = \hat{A}(j,j-1,\ba)_{\neq \emptyset}$.
\item For each $j\in [k-1]\setminus \{1\}$ and $\hat{\bx} \in \hat{A}(j+1,j,\ba)$, 
$\hat{\cP}(\hat{\bx})$ is an $(\epsilon, (1/a_2,\dots, 1/a_{j}))$-regular $(j+1,j,\lambda)$-complex.
\end{enumerate}
\end{lemma}

Note that in Lemma~\ref{lem: counting} the graphs $H^\kk[\Lambda]$ for $\Lambda\in \binom{\ell}{k}$ are all $(\epsilon,d_k)$-regular with respect to $H^{(k-1)}[\Lambda]$.
In view of Lemma~\ref{lem: slicing},
 we obtain the following corollary,
which allows for varying densities at the $k$-th `level'.
\begin{corollary}\label{cor: counting}
For all $\gamma, d_0>0$ and  $k, \ell\in \N\sm \{1\}$ with $k\leq \ell$, 
there exist
$\epsilon_0 := \epsilon_{\ref{cor: counting}}(\gamma,d_0,k,\ell)$ and $m_0:= n_{\ref{cor: counting}}(\gamma,d_0,k,\ell)$
such that the following holds:
Suppose $0\leq \lambda <1/4$.
Suppose $d'\geq d_0$,
$0<\epsilon \leq \epsilon_0$ and $m\geq m_0$ for each $i\in [\ell]$, and
$\bd=(d_2,\dots, d_{k-1})\in \mathbb{R}^{k-2}$
such that
$d_j\geq d_0$ for each $j\in [k-1]\setminus \{1\}$.
Suppose
$\cH= \{H^{(j)}\}_{j=1}^{k}$ is an $(m,\ell,k,\lambda)$-complex, $H^{(1)}=\{V_1,\ldots,V_\ell\}$ with $m_i=|V_i|$ for every $i\in [\ell]$, and
for every $\Lambda\in \binom{\ell}{k}$, the complex $\cH[\Lambda]$ is $(\epsilon,(d_2,\ldots,d_{k-1},p_{\Lambda}))$-regular, where $p_\Lambda$ is a multiple of $d'$.
Then 
$$|\cK_{\ell}(H^{(k)})| = (1\pm \gamma)\prod_{\Lambda\in \binom{\ell}{k}}p_\Lambda \cdot \prod_{j=2}^{k-1}d_j^{\binom{\ell}{j}} \cdot \prod_{i=1}^{\ell} m_i.$$
\end{corollary}

Note that in the above lemma, some $p_{\Lambda}$ are allowed to be zero.

Let $F$ be an $\ell$-vertex $k$-graph and $\cH=\{H^{(j)}\}_{j=1}^{k}$ be a complex with $H^{(1)}=\{V_1,\dots, V_\ell\}$. 
For a bijection $\sigma:V(F)\to [\ell]$, 
we say an induced copy $F'$ of $F$ in $H^{(k)}$ is \emph{$\sigma$-induced} 
if for each $v\in V(F)$ the vertex of $F'$ corresponding to $v$ lies in $V_{\sigma(v)}$.
Let $IC_\sigma(F,\cH)$ 
be the number of $\sigma$-induced copies $F'$ of $F$ in $H^{(k)}$ such that $F'$ is contained in an element of $\cK_{\ell}(H^{(k-1)})$.

\begin{lemma}[Induced counting lemma for many clusters]\label{lem: counting complex}
Suppose $0<1/m\ll \epsilon \ll \gamma, d_0, 1/k, 1/\ell$ with $k\in \mathbb{N}\setminus\{1\}$ and suppose that
\begin{itemize}
\item $F$ is an $\ell$-vertex $k$-graph,
\item $d_0\leq d_j \leq 1 - d_0$ for every $j\in [k-1]\sm\{1\}$,
\item $\cH= \{H^{(j)}\}_{j=1}^{k}$ is an $(m,\ell,k)$-complex with $H^{(1)}= \{ V_1,\dots , V_{\ell}\}$,
\item for each $\Lambda\in \binom{[\ell]}{k}$, the complex $\cH[\Lambda]$ is an $(\epsilon,(d_2,\dots, d_{k-1},p_{\Lambda}))$-regular $(m,k,k)$-complex, and
\item $\sigma:V(F)\to[\ell]$ is a bijection. 
\end{itemize}
Then
$$IC_\sigma(F,\cH) = \left(  \prod_{e\in F} p_{\sigma(e)} \prod_{e\notin F, |e|=k}(1-p_{\sigma(e)}) \pm \gamma \right) \prod_{j=2}^{k-1}d_j^{\binom{\ell}{j}} \cdot m^{\ell}.$$
\end{lemma}
\begin{proof}
We select $q\in \mathbb{N}$ such that $1/m\ll \epsilon \ll 1/q \ll \gamma, d_0,1/k, 1/\ell$ and 
define $\overline{H}^{(k)}:= \cK_{k}(H^{(k-1)}) \setminus H^{(k)}$.  
We also define an $(m,\ell,k)$-graph $H'$ on $\{V_1,\dots, V_\ell\}$ so that for each $e\in \binom{V(F)}{k}$, we have
\begin{align*}
	H'[\sigma(e)]:=\left\{\begin{array}{ll}
H^{(k)}[\sigma(e)] & \text{ if } e\in F,\\
\overline{H}^{(k)}[\sigma(e)] & \text{ otherwise,}
\end{array}\right.
\end{align*}
and let $H':=\bigcup_{e\in \binom{V(F)}{k}}H'[\sigma(e)]$.
Note that $H^{(k-1)}$ underlies $H'$.
Observe that there is a bijection between the set of all $\sigma$-induced copies $F'$ of $F^{(k)}$ in $H^{(k)}$ such that $F'$ is contained in an element of $\cK_{\ell}(H^{(k-1)})$ 
and the set of copies of $K_{\ell}^{(k)}$ in $H'$.
For $e\in \binom{V(F)}{k}$, we define
\begin{align*}
	p'_{\sigma(e)}:=\left\{\begin{array}{ll}
p_{\sigma(e)} & \text{ if } e\in F,\\
1-p_{\sigma(e)} & \text{ otherwise.}
\end{array}\right.
\end{align*}
By Lemma~\ref{lem: simple facts 1}(i), for each $\Lambda\in \binom{[\ell]}{k}$, 
the set $\{ H^{(j)}[\Lambda] \}_{j=1}^{k-1}\cup \{H'[\Lambda]\}$ is an $(\epsilon,(d_2,\dots, d_{k-1},p'_{\Lambda}))$-regular $(m,k,k)$-complex.
It suffices to show that 
\begin{align}\label{eq:goal}
|\cK_{\ell}(H')| = \left(\prod_{\Lambda\in \binom{\ell}{k}}p'_{\Lambda}\pm \gamma \right)\cdot \prod_{j=2}^{k-1}d_j^{\binom{\ell}{j}} \cdot  m^\ell.
\end{align}

We apply the slicing lemma (Lemma~\ref{lem: slicing})  to find for each $\Lambda\in \binom{[\ell]}{k}$  a subgraph $H_1'[\Lambda]$ of $H'[\Lambda]$ which is $(3\epsilon,\lfloor q p'_{\Lambda}\rfloor/q)$-regular\COMMENT{
Lemma~\ref{lem: slicing} can be applied if $p'_{\Lambda}\geq 2\epsilon$,
but if $p'_{\Lambda}< 2\epsilon$, then $\lfloor qp'_\Lambda\rfloor/q=0$.
So we can take $H_1'[\Lambda]=\es$.
We also use that if $s=1$ in Lemma~\ref{lem: slicing} then we don't need the condition that $1/m\ll p_1$.
}
 with respect to $H^{(k-1)}[\Lambda]$.
Similarly, for each $\Lambda\in \binom{[\ell]}{k}$, 
we apply Lemma~\ref{lem: slicing} to the graph $\cK_{k}(H^{(k-1)}[\Lambda]) \setminus H'[\Lambda]$.
In combination with the union lemma~(Lemma~\ref{lem: union regularity})
this gives a supergraph $H_2'[\Lambda]$ of $H'[\Lambda]$ which is $(6\epsilon,\lceil q p'_{\Lambda}\rceil/q)$-regular\COMMENT{
Lemma~\ref{lem: slicing} applied to the graph $\cK_{k}(H^{(k-1)}[\Lambda]) \setminus H'[\Lambda]$ gives us $3\epsilon$-regular graph. Taking union of this with $H'[\Lambda]$ gives us super graph which is $6\epsilon$-regular. \newline
If we use complement, then we will get $3\epsilon$-regular. If we use union as now, we will get $6\epsilon$-regular.
} with respect to $H^{(k-1)}[\Lambda]$.

Let $H_i':=\bigcup_{\Lambda\in \binom{[\ell]}{k}}H'_i[\Lambda]$ for each $i\in [2]$.
Observe that 
$|\cK_{\ell}(H'_1)|\leq |\cK_{\ell}(H')| \leq |\cK_{\ell}(H'_2)|.$
By Corollary~\ref{cor: counting} with $\gamma/2, 1/q, 1/q$ playing the roles of $\gamma, d', d_0$, respectively,\COMMENT{We can do this since $\lfloor qp'_{\Lambda}\rfloor/q$ is a multiple of $1/q$.}
\begin{align*}
|\cK_{\ell}(H'_1)| &\geq \left(1- \frac{\gamma}{2}\right)\prod_{\Lambda\in \binom{\ell}{k}}\lfloor q p'_{\Lambda}\rfloor/q \cdot \prod_{j=2}^{k-1}d_j^{\binom{\ell}{j}} \cdot  m^\ell\enspace \text{ and}\\
|\cK_{\ell}(H'_2)|&\leq \left(1+ \frac{\gamma}{2}\right)\prod_{\Lambda\in \binom{\ell}{k}}\lceil q p'_{\Lambda}\rceil/q \cdot \prod_{j=2}^{k-1}d_j^{\binom{\ell}{j}} \cdot  m^\ell.
\end{align*}
Note that for each $\Lambda \in \binom{\ell}{k}$, we have 
$p'_{\Lambda}-1/q \leq \lfloor q p'_{\Lambda}\rfloor/q$ and $\lceil q p'_{\Lambda}\rceil/q \leq p'_{\Lambda}+ 1/q.$
Thus we obtain~\eqref{eq:goal} as required.\COMMENT{
\begin{align*}
|\cK_{\ell}(H')|= \left(1\pm \frac{\gamma}{2}\right)\prod_{\Lambda\in \binom{\ell}{k}}(p'_{\Lambda}\pm 1/q)\cdot \prod_{j=2}^{k-1}d_j^{\binom{\ell}{j}} \cdot  m^\ell
= \left(\prod_{\Lambda\in \binom{\ell}{k}}p'_{\Lambda} \pm \gamma\right)\cdot \prod_{j=2}^{k-1}d_j^{\binom{\ell}{j}} \cdot  m^\ell,
\end{align*}
since $\binom{\ell}{k}/q < \gamma/4$. This completes the proof of \eqref{eq:goal}.}
\end{proof}

The previous lemma counts $\sigma$-induced copies of a $k$-graph $F$.
However, ultimately, we want to count all induced copies of $F$.
Let us introduce the necessary notation for this step.

Suppose $k,\ell\in \mathbb{N}\sm\{1\}$ such that $\ell\geq k$
and suppose $\ba \in \N^{k-1}$.
Suppose that $d_{\ba,k}:\hat{A}(k,k-1,\ba)\to[0,1]$ is a density function.
Suppose $F$ is a $k$-graph on $\ell$ vertices.
Let ${\rm Bij}(A,B)$ be the set of all bijections from $A$ to $B$.
Suppose $\hat{\bx}\in \hat{A}(\ell,k-1,\ba)$
and $\sigma \in {\rm Bij}(V(F),\bx^{(1)}_*)$ is a bijection.
Let $A(F)$ be the size of the automorphism group of $F$.
We now define three functions in terms of the parameters above that will estimate the number of induced copies of $F$ in certain parts of an $\epsilon$-regular $k$-graph.
Let
\begin{align*}
	IC(F,d_{\ba,k},\hat{\bx},\sigma) &:=  \prod_{\substack{\hat{\by}\leq_{k,k-1} \hat{\bx}, \\ \by^{(1)}_*\in \sigma(F)}} d_{\ba,k}(\hat{\by}) \prod_{\substack{\hat{\by}\leq_{k,k-1} \hat{\bx}, \\ \by^{(1)}_*\notin \sigma(F)}}(1-d_{\ba,k}(\hat{\by}))  \prod_{j=2}^{k-1}a_j^{-\binom{\ell}{j}},\\
IC(F,d_{\ba,k},\hat{\bx}) &:=  \frac{1}{A(F)} \sum_{\sigma \in {\rm Bij}(V(F),\bx^{(1)}_*)} IC(F,d_{\ba,k},\hat{\bx},\sigma),\\
IC(F,d_{\ba,k}) &:= \binom{a_1}{\ell}^{-1}\sum_{\hat{\bx} \in \hat{A}(\ell,k-1,\ba)} IC(F,d_{\ba,k},\hat{\bx}).
\end{align*}

We will now show that for a $k$-graph $H$ satisfying a suitable regularity instance $R=(\epsilon,\ba,d_{\ba,k})$,
the value $IC(F,d_{\ba,k})$ is a very accurate estimate for $\mathbf{Pr}(F, H)$ (recall the latter was introduced in Section~\ref{sec: basic notation}).
The same is true if $F$ is replaced by a finite family of $k$-graphs (see Corollary~\ref{cor: counting collection}).

\begin{lemma}[Induced counting lemma for general hypergraphs]\label{lem: counting hypergraph}
Suppose $0<1/n\ll \epsilon \ll 1/t, 1/a_1 \ll \gamma,  1/k, 1/\ell$ with $2\leq k \leq \ell$.
Suppose $F$ is an $\ell$-vertex $k$-graph and $\ba\in [t]^{k-1}$.
Suppose $H$ is an $n$-vertex $k$-graph satisfying a regularity instance $R= (\epsilon,\ba,d_{\ba,k})$.
Then 
$$\mathbf{Pr}(F, H)= IC(F,d_{\ba,k}) \pm \gamma.$$
\end{lemma}
\begin{proof}
Since $H$ satisfies the regularity instance $R$, 
there exists a  $(\epsilon,\ba,d_{\ba,k})$-equitable partition $\sP=\sP(k-1,\ba)$ of $H$ (as defined in Section~\ref{sec: subsub density function}).
Let $\sP^{(1)}= \{V_1, \dots , V_{a_1}\}$ and $m:= \lfloor n/a_1 \rfloor.$
We say an induced copy $F'$ of $F$ in $H$ is \emph{crossing-induced} 
if $V(F')\in \cK_{\ell}(\sP^{(1)})$ and \emph{non-crossing-induced} otherwise.
Then by \eqref{eq: eta a1}, 
\begin{equation}\label{eq: non-crossing-induced}
\begin{minipage}[c]{0.8\textwidth}\em
there are at most $\frac{\gamma}{3}\binom{n}{\ell}$ non-crossing-induced copies of $F$.
\end{minipage}
\end{equation}
The strategy of the proof is as follows.
We only consider crossing-induced copies of $F$, as the number of non-crossing-induced copies is negligible.
For each $\hat{\bx} \in \hat{A}(\ell,k-1,\ba)$, 
we fix some bijection $\sigma$ between $V(F)$ and $\bx_*^{(1)}$.
By Lemma~\ref{lem: counting complex}, we can accurately estimate the number of $\sigma$-induced copies of $F$.
By summing over all choices for $\hat{\bx}$ and $\sigma$ and taking in account which copies we counted multiple times,
we can estimate the number of crossing-induced copies of $F$ in $H$.

For each $\hat{\bx} \in \hat{A}(\ell,k-1,\ba)$, 
we consider the $(k-1)$-polyad $\hat{P}^{(k-1)}(\hat{\bx}) = \bigcup_{\hat{\by}\leq_{k,k-1} \hat{\bx}} \hat{P}^{(k-1)}(\hat{\by})$ as defined in \eqref{eq: larger polyad def}.
By Proposition~\ref{prop:unique}, for every crossing-induced copy $F'$ of $F$ in $H$, 
there is a unique $\hat{\bx}\in \hat{A}(\ell,k-1,\ba)$ such that $F'$ is contained in some element of $\cK_{\ell}(\hat{P}^{(k-1)}(\hat{\bx}))$ .

Consider any $\hat{\bx} \in \hat{A}(\ell,k-1,\ba)$ and a 
bijection $\sigma:V(F)\to\bx^{(1)}_*$. Let
$$\cH'(\hat{\bx}):= \bigg\{ \bigcup_{\hat{\bz}\leq_{k,i} \hat{\bx}} \hat{P}^{(i)}(\hat{\bz})\bigg\}_{i\in [k-1]}
\enspace \text{ and } \enspace \cH(\hat{\bx}):= \cH'(\hat{\bx}) \cup  \{H\cap \cK_{k}(\hat{P}^{(k-1)}(\hat{\bx}))\}.$$
Hence $\cH(\hat{\bx})$ is an $(\ell,k)$-complex and $\cH'(\hat{\bx})$ is an $(\ell,k-1)$-complex. 
Note that $\cH'(\hat{\bx}) = \bigcup_{\hat{\by}\leq_{k,k-1}\hat{\bx}}\hat{\cP}(\hat{\by})$, where $\hat{\cP}(\hat{\by})$ is as defined in \eqref{eq: complex definition by address}.

Lemma~\ref{lem: maps bijections} implies that each $\hat{\cP}(\hat{\by}) = \cH'(\hat{\bx})[\by^{(1)}_*]$ is $(\epsilon,(1/a_2,\dots, 1/a_{k-1}))$-regular (if $k\geq 3$).
Furthermore, since $\sP$ is an $(\epsilon,\ba,d_{\ba,k})$-equitable partition of $H$, for each $e\in \binom{V(F)}{k}$, 
the $k$-graph $H[\sigma(e)]$ is $(\epsilon,d_{\ba,k}(\hat{\by}))$-regular with respect to 
$\hat{P}^{(k-1)}(\hat{\by})= \hat{P}^{(k-1)}(\hat{\bx})[\bigcup_{ i\in \by^{(1)}_* } V_i]$, where $\hat{\by}$ is the unique vector satisfying $\hat{\by}\leq_{k,k-1} \hat{\bx}$ and $\by^{(1)}_* = \sigma(e)$.

Thus, by applying Lemma~\ref{lem: counting complex} with $\cH(\hat{\bx}), a_i^{-1},\gamma/(3\ell!), d_{\ba,k}(\hat{\by})$ playing the roles of $\cH, d_i,\gamma, p_{\by^{(1)}_*}$, we conclude that\COMMENT{$IC_{\sigma}(F,\cH(\hat{\bx}))$ is the number of crossing-induced copies $F'$ of $F$ in $H$ which are also $\sigma$-induced and for which $F'$ lies in some element of $\cK_{\ell}(\hat{P}^{(k-1)}(\hat{\bx}))$.} 
(with $IC_\sigma(F,\cH(\hat{\bx}))$ defined as in Lemma~\ref{lem: counting complex})
\begin{align*}
IC_{\sigma}(F,\cH(\hat{\bx})) &= \left(\prod_{\substack{\hat{\by}\leq_{k,k-1} \hat{\bx}, \\ \by^{(1)}_*\in \sigma( F)}} d_{\ba,k}(\hat{\by}) 
\prod_{\substack{\hat{\by}\leq_{k,k-1}\hat{\bx}, \\\by^{(1)}_*\notin \sigma(F)}}(1-d_{\ba,k}(\hat{\by}))  \pm \frac{\gamma}{3\ell!} \right)
\prod_{j=2}^{k-1}a_j^{-\binom{\ell}{j}} \cdot m^{\ell} \\ 
&= \left(IC(F,d_{\ba,k},\hat{\bx},\sigma)\pm \frac{\gamma}{3\ell!} \prod_{j=2}^{k-1}a_j^{-\binom{\ell}{j}}\right) m^{\ell} .
\end{align*}
Next we want to estimate the number of all crossing-induced copies of $F$ in $H$ which lie in some element of $\cK_{\ell}(\hat{P}^{(k-1)}(\hat{\bx}))$.
Observe that we count every copy of $F$ exactly $A(F)$ times if we sum over all possible bijections $\sigma$.
Therefore, the number of crossing-induced copies of $F$ in $H$ which lie in some element of $\cK_{\ell}(\hat{P}^{(k-1)}(\hat{\bx}))$ is
\begin{align*}
	\frac{1}{A(F)} \sum_{\sigma} IC_{\sigma}(F,\cH(\hat{\bx})) &=
	\frac{1}{A(F)} \sum_{\sigma} \left(IC(F,d_{\ba,k},\hat{\bx},\sigma)\pm \frac{\gamma}{3\ell!} \prod_{j=2}^{k-1}a_j^{-\binom{\ell}{j}}\right) m^{\ell} \\
	&= \left(IC(F,d_{\ba,k},\hat{\bx}) \pm \frac{\gamma}{3} \prod_{j=2}^{k-1}a_j^{-\binom{\ell}{j}}\right) m^{\ell}. 
\end{align*}
Note that $|\hat{A}(\ell,k-1,\ba)| = \binom{a_1}{\ell}\prod_{j=2}^{k-1}a_j^{\binom{\ell}{j}}$
and $\binom{a_1}{\ell} m^{\ell} = (1\pm \gamma/10)\binom{n}{\ell}$, because $1/a_1 \ll \gamma, 1/\ell$. 
Hence the number of crossing-induced copies of $F$ in $H$ is
\begin{align*}
&\sum_{\hat{\bx}\in \hat{A}(\ell,k-1,\ba)} \left(IC(F,d_{\ba,k},\hat{\bx}) \pm \frac{\gamma}{3} \prod_{j=2}^{k-1}a_j^{-\binom{\ell}{j}}\right) m^{\ell} \\
&\qquad\qquad=\left(\sum_{\hat{\bx}\in \hat{A}(\ell,k-1,\ba)} IC(F,d_{\ba,k},\hat{\bx})\right) m^{\ell} \pm \frac{\gamma}{3} \prod_{j=2}^{k-1}a_j^{-\binom{\ell}{j}}|\hat{A}(\ell,k-1,\ba)|m^{\ell}\\
&\qquad\qquad=  (IC(F,d_{\ba,k})\pm \gamma/2) \binom{n}{\ell}.
\end{align*}
This together with \eqref{eq: non-crossing-induced} implies the desired statement.
\end{proof}

In the previous lemma we counted the number of induced copies of a single $k$-graph $F$ in $H$.
It is not difficult to extend this approach to a finite family of $k$-graphs.
For a finite family $\cF$ of $k$-graphs, 
we define 
\begin{align}\label{eq: def IC cF d a k}
IC(\cF, d_{\ba,k}):= \sum_{F\in \cF} IC(F,d_{\ba,k}).
\end{align}

\begin{corollary}\label{cor: counting collection}
Suppose $0<1/n\ll \epsilon \ll 1/t, 1/a_1 \ll \gamma,  1/k, 1/\ell$ with $2\leq k \leq \ell$.
Let $\cF$ be a collection of $k$-graphs on $\ell$ vertices.
Suppose $H$ is an $n$-vertex $k$-graph satisfying a regularity instance $R= (\epsilon,\ba,d_{\ba,k})$ where $\ba\in [ t]^{k-1}$. Then 
$$\mathbf{Pr}(\cF, H)=IC(\cF, d_{\ba,k}) \pm \gamma.$$
\end{corollary}
\begin{proof}
For each $F\in \cF$, we apply Lemma~\ref{lem: counting hypergraph} with $\gamma/2^{\binom{\ell}{k}}$ playing the role of $\gamma$.
As $|\cF|\leq 2^{\binom{\ell}{k}}$, this completes the proof.
\end{proof}

\subsection{Refining a partition}\label{sec:3.4}
In this subsection we make a simple observation regarding refinements of a given partition.
This shows that we can refine a family of partitions 
without significantly affecting the regularity parameters.

\begin{lemma}\label{lem: partition refinement}
Suppose $0< 1/n \ll\epsilon \ll 1/t ,1/k$ with $k, t\in \N\sm \{1\}$, $0<\eta <1$, and $\ba\in \N^{k-1}$.
Suppose $\sP=\sP(k-1,\ba)$ is an $(\eta,\epsilon,\ba)$-equitable family of partitions on $V$ with $|V|=n$. 
Suppose $\bb\in [t]^{k-1}$  and $a_i \mid b_i$ for all $i\in [k-1]$.
Then there exists a family of partitions $\sQ=\sQ(k-1,\bb)$ on $V$ which is $(\eta,\epsilon^{1/3},\bb)$-equitable and $\sQ \prec \sP$.
\end{lemma}
It is easy to prove this by induction on $k$ via an appropriate application of the slicing lemma (Lemma~\ref{lem: slicing}).
We omit the details.

\subsection{Small perturbations of partitions}\label{sec:3.5}
Here we consider the effect of small changes in a partition on the resulting parameters.
In particular, the next lemma implies that for any family of partitions $\sP$,
every family of partitions that is close to $\sP$ in distance is a family of partitions with almost the same parameters. This is proved in \cite{JKKO1}.

\begin{lemma}\label{lem: slightly different partition regularity}
Suppose $k\in \N\sm\{1\}$, $0<1/n \ll \nu \ll \epsilon$, and $0\leq \lambda \leq 1/4$.
Suppose $R=(\epsilon/3,\ba,d_{\ba,k})$ is a regularity instance.
Suppose $V$ is a vertex set of size $n$ and suppose $G^\kk,H^\kk$ are $k$-graphs on $V$ with $|G^{(k)}\triangle H^{(k)}|\leq \nu \binom{n}{k}$.
Suppose $\sP=\sP(k-1,\ba)$ is a $(1/a_1,\epsilon,\ba,\lambda)$-equitable family of partitions on $V$ which is an $(\epsilon,d_{\ba,k})$-partition of $H^{(k)}$.
Suppose $\sQ=\sQ(k-1,\ba)$ is a family of partitions on $V$ such that
for any $j\in [k-1]$, $\hat{\bx}\in \hat{A}(j,j-1,\ba)$, and $b\in [a_j]$, we have
\begin{align}\label{eq: P triangle Q leq nu}
	|P^{(j)}(\hat{\bx},b) \triangle Q^{(j)}(\hat{\bx},b)| \leq \nu \binom{n}{j}.
\end{align}
Then $\sQ$ is a $(1/a_1,\epsilon+ \nu^{1/6},\ba,\lambda+\nu^{1/6})$-equitable family of partitions which is an $(\epsilon+\nu^{1/6},d_{\ba,k})$-partition of $G^{(k)}$.
\end{lemma}

The following lemma 
shows that for every equitable family of partitions $\sP$ whose vertex partition $\sP^{(1)}$ is an almost equipartition,
there is an equitable family of partitions with almost the same parameters whose vertex partition is an equipartition.
\begin{lemma}\label{lem: removing lambda}
Suppose $0< 1/n \ll \lambda   \ll \epsilon \leq 1$, $k\in \N\sm\{1\}$, and $R=(\epsilon/3,\ba,d_{\ba,k})$ is a regularity instance.
Suppose $\sP=\sP(k-1,\ba)$ is a $(1/a_1,\epsilon,\ba,\lambda)$-equitable family of partitions and an $(\epsilon,d_{\ba,k})$-partition of an $n$-vertex $k$-graph $H^{(k)}$.
Then there exists a family of partitions $\sQ=\sQ(k-1,\ba)$ which is an $(\epsilon+\lambda^{1/10},\ba,d_{\ba,k})$-equitable partition of $H^{(k)}$. 
\end{lemma}
\begin{proof}
Let $m:=\lfloor n/a_1 \rfloor$.
We write $\sP^{(1)}=\{V_1,\dots ,V_{a_1}\}$.
Since $\sP$ is a $(1/a_1,\epsilon,\ba,\lambda)$-equitable family of partitions,
we have $|V_i| = (1\pm \lambda)m$ for all $i\in [a_1]$, and Lemma~\ref{lem: maps bijections} implies that for each $j\in [k-1]$, 
the function $\hat{P}^{(j)}(\cdot):\hat{A}(j+1,j,\ba)\to \hat{\sP}^{(j)}$ is a bijection and 
for each $j\in [k-1]\setminus \{1\}$, the function
$P^{(j)}(\cdot,\cdot):\hat{A}(j,j-1,\ba)\times [a_j]\to \sP^{(j)}$ is also a bijection.
Next, we fix the size of the parts in the new equitable partition $\sQ$ of $V:=V_1\cup \dots\cup V_{a_1}$.
For each $i\in [a_1]$, let $m_i:= \lfloor (n+i-1)/a_1\rfloor$.
Thus $m_i\in \{m,m+1\}$.
Choose $U'_i\sub V_i$ of size $\max\{|V_i|, m_i\}$ and let $U'_0:= \bigcup_{i\in [a_1]} V_i\setminus U'_i$.
We partition $U'_0$ into $U''_1,\dots, U''_{a_1}$ in an arbitrary manner such that $|U''_i| = m_i - |U'_i|$.
For each $i\in [a_1]$, let 
\begin{align*}
	U_i := U'_i\cup U''_i \text{ and }\sQ^{(1)}:=\{U_1,\dots, U_{a_1}\}.
\end{align*}
Moreover, let $Q^{(1)}(b,b) := U_b$ for each $b\in [a_1]$,
and $Q^{(1)}(b,b'):=\es$ for all distinct $b,b'\in[a_1]$.
For each $i\in [a_1]$, we have $$|U_i\triangle V_i|\leq | (1\pm \lambda)m-m_i| \leq \lambda m+1.$$
For each $\hat{\bx}=(\alpha_1,\alpha_2) \in \hat{A}(2,1,\ba)$, let 
$\hat{Q}^{(1)}(\hat{\bx}):= U_{\alpha_1} \cup U_{\alpha_2}$.
Note that $\{\cK_{2}(\hat{Q}^{(1)}(\hat{\bx})):\hat{\bx} \in \hat{A}(2,1,\ba)\}$ forms a partition of $\cK_2(\sQ^{(1)})$.

Now, we inductively construct $\sQ^{(2)},\dots, \sQ^{(k-1)}$ in this order.
Assume that for some $j\in [k]\sm\{1\}$, we have already defined  $\{\sQ^{(i)}\}_{i=1}^{j-1}$ with $\sQ^{(i)}=\{Q^{(i)}(\hat{\bx},b): \hat{\bx}\in \hat{A}(i,i-1,\ba), b\in [a_i]\}$ and $\hat{\sQ}^{(i)}=\{\hat{Q}^{(i)}(\hat{\bx}) :\hat{\bx}\in \hat{A}(i+1,i,\ba)\}$ for each $i\in [j-1]$ such that the following hold.

\begin{itemize}
\item[(Q1)$_{j-1}$] For each $i\in [j-1]$, $\hat{\bx}\in \hat{A}(i,i-1,\ba)$ and $b\in [a_i]$, 
we have $|P^{(i)}(\hat{\bx},b)\triangle Q^{(i)}(\hat{\bx},b)| \leq 2^{i}i! \lambda n^{i}$.
\item[(Q2)$_{j-1}$] For each $i\in [j-1]\setminus\{1\}$ and $\hat{\bx}\in \hat{A}(i,i-1,\ba)$, the collection  
$\{Q^{(i)}(\hat{\bx},b) : b\in [a_i]\}$ forms a partition of $\cK_{i}(\hat{Q}^{(i-1)}(\hat{\bx}))$.
\item[(Q3)$_{j-1}$]  For each $i\in [j-1]$ and $\hat{\bx}\in \hat{A}(i+1,i,\ba)$, 
we have $\hat{Q}^{(i)}(\hat{\bx}) =\bigcup_{\hat{\by}\leq_{i,i-1} \hat{\bx}} Q^{(i)}(\hat{\by},\bx^{(i)}_{\by^{(1)}_* })$.
\end{itemize}\vspace{0.2cm}
Note that $\sQ^{(1)}$ satisfies (Q1)$_{1}$--(Q3)$_{1}$. Suppose first that $j\leq k-1$. In this case we will define $\sQ^{(j)}$ satisfying (Q1)$_{j}$--(Q3)$_{j}$.
For each $\hat{\bx}\in \hat{A}(j,j-1,\ba)$ and $b\in [a_j]$, we define
\begin{align*}
	Q^{(j)}(\hat{\bx},b):=\left\{ \begin{array}{ll}
	P^{(j)}(\hat{\bx},b) \cap\cK_{j}(\hat{Q}^{(j-1)}(\hat{\bx})) & \text{ if } b\in [a_j-1], \\
	\cK_{j}(\hat{Q}^{(j-1)}(\hat{\bx}))\setminus \bigcup_{b\in [a_j-1]} P^{(j)}(\hat{\bx},b) & \text{ otherwise.}
	\end{array}\right. 
\end{align*}
Let $$\sQ^{(j)}:= \{	Q^{(j)}(\hat{\bx},b):\hat{\bx}\in \hat{A}(j,j-1,\ba), b\in [a_j]\}.$$
Then for any fixed $\hat{\bx}\in \hat{A}(j,j-1,\ba)$, it is obvious that 
$Q^{(j)}(\hat{\bx},1),\dots, Q^{(j)}(\hat{\bx},a_j)$ forms a partition of $\cK_{j}(\hat{Q}^{(j-1)}(\hat{\bx}))$. Thus (Q2)$_{j}$ holds.

For each $\hat{\bz} \in \hat{A}(j+1,j,\ba)$, let
\begin{align*}
\hat{Q}^{(j)}(\hat{\bz}) :=\bigcup_{\hat{\by}\leq_{j,j-1} \hat{\bz}} Q^{(j)}(\hat{\by},\bz^{(j)}_{\by^{(1)}_* }).
\end{align*}
Then (Q3)$_{j}$ also holds.%
\COMMENT{For each $\hat{\bx}\in \hat{A}(j,j-1,\ba)$ and $b\in [a_j]$, the definition of $Q^{(j)}(\hat{\bx},b)$ implies that 
$$P^{(j)}(\hat{\bx},b)\triangle Q^{(j)}(\hat{\bx},b) \subseteq 
\cK_{j}(\hat{P}^{(j-1)}(\hat{\bx}))\triangle \cK_{j}(\hat{Q}^{(j-1)}(\hat{\bx})).$$
}

Note that for any fixed $(j-1)$-set $J'\in \hat{P}^{(j-1)}(\hat{\bx})\triangle \hat{Q}^{(j-1)}(\hat{\bx})$, there are at most $(1+\lambda)m$ distinct $j$-sets in $\cK_{j}(\hat{P}^{(j-1)}(\hat{\bx}))\triangle \cK_{j}(\hat{Q}^{(j-1)}(\hat{\bx}))$ containing $J'$.
Thus for $\hat{\bx}\in \hat{A}(j,j-1,\ba)$ and $b\in [a_j]$, we obtain
\begin{eqnarray*}
|P^{(j)}(\hat{\bx},b)\triangle Q^{(j)}(\hat{\bx},b)|
&\leq &|\cK_{j}(\hat{P}^{(j-1)}(\hat{\bx}))\triangle \cK_{j}(\hat{Q}^{(j-1)}(\hat{\bx}))|\\ &\leq& (1+\lambda) m |\hat{P}^{(j-1)}(\hat{\bx})\triangle \hat{Q}^{(j-1)}(\hat{\bx})| \\
&\stackrel{\eqref{eq:hatPconsistsofP(x,b)}, {\rm (Q3)}_{j-1} }{\leq}&
 \sum_{\hat{\by}\leq_{j-1,j-2} \hat{\bx}} 2m | P^{(j-1)}(\hat{\by},\bx^{(j)}_{\by^{(1)}_* }) \triangle Q^{(j-1)}(\hat{\by},\bx^{(j)}_{\by^{(1)}_* })|  \\
 & \stackrel{{\rm (Q1)}_{j-1}}{\leq}& 2^{j} j! \lambda n^{j}. 
 \end{eqnarray*}
 \COMMENT{Here, \eqref{eq:hatPconsistsofP(x,b)} indicate the fact that $\sP$ satisfies \eqref{eq:hatPconsistsofP(x,b)}.
 Also \eqref{eq: I subset J then leq} implies that there are at most $j$ vectors $\hat{\by}$ such that $\hat{\by} \leq_{j-1,j-2} \hat{\bx}$. Thus we get the final inequality. }
Thus (Q1)$_{j}$ holds and 
we obtain $\{\sQ^{(i)}\}_{i=1}^{j}$ satisfying (Q1)$_{j}$--(Q3)$_{j}$. 
Inductively we obtain $\sQ= \{\sQ^{(i)}\}_{i=1}^{k-1}$ satisfying (Q1)$_{k-1}$--(Q3)$_{k-1}$.

Note that since $R=(\epsilon/3,\ba,d_{\ba,k})$ is a regularity instance, we have the following inequality.
$$\epsilon\leq \|\ba\|_\infty^{-4^k} \epsilon_{\ref{lem: counting}}(\|\ba\|^{-1}_\infty, \|\ba\|^{-1}_\infty,k-1,k).$$
Thus Lemma~\ref{lem: counting} (together with Lemma~\ref{lem: maps bijections}(ii)) implies for any $j \in [k-1]\sm\{1\}$ and $(\hat{\bx},b)\in \hat{A}(j,j-1,\ba)\times[a_j]$ that
$|P^{(j)}(\hat{\bx},b)|  \geq \epsilon^{1/2} n^{j}.$%
\COMMENT{
$$|P^{(j)}(\hat{\bx},b)| \geq \frac{1}{2a_j} | \cK_j(\hat{P}^{(j-1)}(\hat{\bx}))|  \geq \frac{1}{4} \prod_{i=2}^{j}a_i^{-\binom{j}{i}} ((1-\lambda)m)^{j}  \geq \epsilon^{1/2} n^{j}.$$}
This with (Q1)$_{k-1}$ shows that  for any $j \in [k-1]\sm\{1\}$ and $(\hat{\bx},b)\in \hat{A}(j,j-1,\ba)\times[a_j]$, the $j$-graph $Q^{(j)}(\hat{\bx},b)$ is nonempty.
Together with properties (Q2)$_{k-1}$ and (Q3)$_{k-1}$ this in turn ensures that we can apply Lemma~\ref{lem: family of partitions construction} to show that $\sQ$ is a family of partitions.

By (Q1)$_{k-1}$ and the assumption that $R=(\epsilon/3,\ba,d_{\ba,k})$ is a regularity instance, we can apply Lemma~\ref{lem: slightly different partition regularity} with $\sP, \sQ, H^{(k)}, H^{(k)}, \lambda, \lambda^{9/10}$\COMMENT{ $2^{j}j! \lambda n^{j} \leq \lambda^{9/10} \binom{n}{j}.$} playing the roles of $\sP, \sQ, H^{(k)}, G^{(k)}, \lambda, \nu$ to obtain that  $\sQ=\sQ(k-1,\ba)$ is an $(\epsilon+\lambda^{1/10},\ba,d_{\ba,k})$-equitable partition of $H^{(k)}$.
\end{proof}

\section{Regular approximations of partitions and hypergraphs}
\label{sec: Regular approximations of partitions and hypergraphs}

The main aim of this section is to prove a strengthening of the partition version (Lemma~\ref{RAL(k)}) of the regular approximation lemma. As described in Section~\ref{sec: partitions of hypergraphs and RAL}, Lemma~\ref{RAL(k)} outputs for a given equitable family of partitions $\sQ$ another family of partitions $\sP$ that refines $\sQ$.
In Lemma~\ref{lem:refreg} $\sP$ has the additional feature that it almost refines a further given (arbitrary) family of partitions $\sO$.
Observe that we cannot hope to refine $\sO$ itself, as for example some sets in $\sO^{(1)}$ may be very small.
We also prove two further tools: Lemma~\ref{lem: mimic Kk} allows us to transfer the large scale structure of a hypergraph to another one on a different vertex set and Lemma~\ref{lem: reconstruction} concerns suitable perturbations of a given partition. 
Lemmas~\ref{lem:refreg}--\ref{lem: reconstruction} will all be used in the proof of Lemma~\ref{lem: similar}. 

\begin{lemma}\label{lem:refreg}
For all $k, o \in \N\setminus\{1\}$, $s\in \N$, all $\eta,\nu>0$, 
and every function $\epsilon : \mathbb{N}^{k-1}\rightarrow (0,1]$, 
there are $\mu=\mu_{\ref{lem:refreg}}(k,o,s,\eta,\nu,\epsilon)>0$ 
and $t=t_{\ref{lem:refreg}}(k,o,s,\eta,\nu,\epsilon)\in \N$ and $n_0=n_{\ref{lem:refreg}}(k,o,s,\eta,\nu,\epsilon)\in\N$ such that the following hold. 
Suppose
\begin{itemize}
\item[{\rm (O1)$_{\ref{lem:refreg}}$}] $V$ is a set and $|V|=n \geq n_0$,
\item[{\rm (O2)$_{\ref{lem:refreg}}$}] $\sO(k-1,\ba^{\sO}) = \{\sO^{(j)}\}_{j=1}^{k-1}$ is an $o$-bounded family of partitions on $V$,
\item[{\rm (O3)$_{\ref{lem:refreg}}$}] $\sQ=\sQ(k,\ba^{\sQ})$ is a $(1/a_1^{\sQ},\mu,\ba^{\sQ})$-equitable $o$-bounded family of partitions on $V$, and
\item[{\rm (O4)$_{\ref{lem:refreg}}$}] $\sH^{(k)}= \{H^{(k)}_1,\dots,H^{(k)}_s\}$ is a partition of $\binom{V}{k}$ so that $\sH^{(k)} \prec \sQ^{(k)}$.
\end{itemize}
Then there exist a family of partitions $\sP = \sP(k-1,\ba^{\sP})$ and a partition $\sG^{(k)}= \{G^{(k)}_1,\dots, G^{(k)}_s\}$ of $\binom{V}{k}$ satisfying the following for each $j\in [k-1]$ and $i\in [s]$.
\begin{itemize}
\item[{\rm (P1)$_{\ref{lem:refreg}}$}] $\sP$ is a $t$-bounded $(\eta,\epsilon(\ba^{\sP}), \ba^{\sP})$-equitable family of partitions, and $a^{\sQ}_j$ divides $a^{\sP}_j$,
\item[{\rm (P2)$_{\ref{lem:refreg}}$}] $\sP^{(j)}\prec \sQ^{(j)}$ and $\sP^{(j)} \prec_{\nu} \sO^{(j)}$, 

\item[{\rm (G1)$_{\ref{lem:refreg}}$}] $G^{(k)}_i$ is perfectly $\epsilon(\ba^{\sP})$-regular with respect to $\sP$,
\item[{\rm (G2)$_{\ref{lem:refreg}}$}] $\sum_{i=1}^{s} |G_i^{(k)}\triangle H^{(k)}_i| \leq \nu \binom{n}{k}$, and 
\item[{\rm (G3)$_{\ref{lem:refreg}}$}] $\sG^{(k)} \prec \sQ^{(k)}$ and if $H^{(k)}_i \subseteq \cK_k(\sQ^{(1)})$ then $G_i^{(k)} \subseteq \cK_k(\sQ^{(1)})$.
\end{itemize}
\end{lemma}

We believe that Lemma~\ref{lem:refreg} will have additional applications. As we mentioned it is also used in the proof Corollary~9.3 
in \cite{JKKO1}.

In Lemma~\ref{lem:refreg} we may assume without loss of generality that $1/\mu,t,n_0$ are non-decreasing in $k,o,s$ and non-increasing in $\eta, \nu$.\COMMENT{We cannot include $\epsilon$ here, since $\epsilon$ is a function.}

To prove Lemma~\ref{lem:refreg} we proceed by induction on $k$.
In the induction step,
we first construct an `intermediate' family of partitions $\{\sL^{(i)}\}_{i=1}^{k-1}$ which
satisfies (P1)$_{\ref{lem:refreg}}$ and (P2)$_{\ref{lem:refreg}}$.
The partitions $\sL^{(1)},\ldots,\sL^{(k-1)}$ are constructed via the inductive assumption of Lemma~\ref{lem:refreg} (see Claim~\ref{cl: claim 2 claim 2}).
We then construct a partition $\sL^{(k)}$ via appropriate applications of the slicing lemma (see Claim~\ref{claim:hatsL}).
Finally, we apply Lemma~\ref{RAL(k)} with $\sL_*= \{\sL^{(i)}\}_{i=1}^{k}$ playing the role of $\sQ$ to obtain our desired family of partitions $\sP$ and construct $G_i^\kk$ based on the $k$-graphs guaranteed by Lemma~\ref{RAL(k)}.

\begin{proof}[Proof of Lemma~\ref{lem:refreg}]
First of all, 
by decreasing the value of $\eta$ if necessary, 
we may assume that $\eta < 1/(10k!)$.
We may also assume that $\nu\leq \eta$.\COMMENT{We need this to get (L$^*$1) in the $k=2$ case of the proof of Claim~\ref{cl: claim 2 claim 2}}

We use induction on $k$. For each $k \in \mathbb{N}\setminus\{1\}$, let $L_{\ref{lem:refreg}}(k)$ be the statement of the lemma. Let $L_{\ref{lem:refreg}}(1)$ be the following statement (Claim~\ref{cl: 6.2 k=1 base case}).
\begin{claim}[$L_{\ref{lem:refreg}}(1)$]\label{cl: 6.2 k=1 base case}
For all $o,s\in \N$, all $\eta,\nu>0$, 
there are $t= t_{\ref{lem:refreg}}(1,o,s,\nu):=so\lceil 2\nu^{-2}\rceil$ and $n_0=n_{\ref{lem:refreg}}(1,o,s,\nu)\in\N$ such that the following hold. 
Suppose
\begin{itemize}
\item[{\rm (O1)$^1_{\ref{lem:refreg}}$}] $V$ is a set and $|V|=n \geq n_0$,
\item[{\rm (O2)$^1_{\ref{lem:refreg}}$}] $\sQ^{(1)}$ is an equipartition of $V$ into $a^{\sQ}_1\leq o$ parts,
\item[{\rm (O3)$^1_{\ref{lem:refreg}}$}] $\sH^{(1)}= \{H^{(1)}_1,\dots,H^{(1)}_s\}$ is a partition of $V$ so that $\sH^{(1)} \prec \sQ^{(1)}$.
\end{itemize}
Then there exists a partition $\sP^{(1)}$ of $V$ satisfying the following.
\begin{itemize}
\item[{\rm (P1)$^1_{\ref{lem:refreg}}$}] $\sP^{(1)}$ is an equipartition of $V$ into $a_1^{\sP} \leq t$ parts, and $a^{\sQ}_1$ divides $a^{\sP}_1$,
\item[{\rm (P2)$^1_{\ref{lem:refreg}}$}] $\sP^{(1)}\prec \sQ^{(1)}$ and $\sP^{(1)} \prec_{\nu^2} \sH^{(1)}$.
\end{itemize}
\end{claim}
\begin{proof}
Write $\sQ^{(1)}= \{Q^{(1)}_i : i\in [a_1^{\sQ}]\}.$ Let $a^{\sP}_1:= s a^{\sQ}_1 \lceil 2\nu^{-2} \rceil$, 
let $m:= \min\{|Q^{(1)}_i|: i\in [a_1^{\sQ}]\}$, and let $m':= \lfloor  |V|/a^{\sP}_1\rfloor$. 
Thus $|Q^{(1)}_i|\in \{m,m+1\}$ for each $i\in [a^{\sQ}_1]$. 
The sets in $\sP^{(1)}$ will have size $m'$ or $m'+1$.
Note that
\begin{align*}
m' \cdot\frac{a^{\sP}_1}{ a^{\sQ}_1} 
\leq  \left\lfloor  \frac{|V|}{a^{\sP}_1}\cdot\frac{a^{\sP}_1}{ a^{\sQ}_1} \right\rfloor 
= m 
< (m'+1) \cdot\frac{a^{\sP}_1}{ a^{\sQ}_1}.
\end{align*}
In particular, as $\frac{a^{\sP}_1}{ a^{\sQ}_1}$ is an integer, we have
\begin{align}
\label{eq: m m' size difference}
0\leq m ~({\rm mod}~ m') \leq \frac{a^{\sP}_1}{ a^{\sQ}_1}.
\end{align}
To obtain $\sP^{(1)}$ we further (almost) refine $\sH^{(1)}$.
For each $i\in [s]$, we define $\ell_{i}:= \lfloor|H^{(1)}_{i}|/m' \rfloor$.
We arbitrarily partition $H^{(1)}_{i}$ into $\sL(i,0),\dots, \sL(i,\ell_{i})$ such that $|\sL(i,r)|=m'$ for all $r\in [\ell_{i}]$ and $|\sL(i,0)|< m'$.
For each $j\in [a^{\sQ}_1]$, let $\sL'(j,0):=\bigcup_{ H^{(1)}_i\subseteq Q^{(1)}_j } \sL(i,0)$.
Let $\ell'_j := \lfloor |\sL'(j,0)|/m' \rfloor$.
We arbitrarily partition $\sL'(j,0)$ into $\sL''(j,0),\sL'(j,1),\dots, \sL'(j,\ell'_j)$ such that 
$|\sL'(j,r)| = m'$ for all $r\in [\ell'_j]$ and $|\sL''(j,0)|< m'$.
Note that since $\sH^{(1)}\prec \sQ^{(1)}$, for all $j \in [a_1^{\sQ}]$, we have
$$Q^{(1)}_j = \bigcup_{i\colon H^{(1)}_i\subseteq Q^{(1)}_j } H^{(1)}_{i} 
= \sL''(j,0) \cup \bigcup_{r\in [\ell'_j]}\sL'(j,r)\cup \bigcup_{i\colon H^{(1)}_i\subseteq Q^{(1)}_j, r\in [\ell_{i}]} \sL(i,r).$$
As all sets $\sL'(j,r)$ and $\sL(i,r)$ with $r>0$ have size exactly $m'$, the fact that $|\sL''(j,0)|< m'$ implies that we have $|\sL''(j,0)| = m  ~({\rm mod}~ m')$. 
On the other hand, this together with \eqref{eq: m m' size difference} implies that 
$|\sL''(j,0)| 
\leq a^{\sP}_1/ a^{\sQ}_1 
= \ell'_j + \sum_{i\colon H^{(1)}_i\subseteq Q^{(1)}_j} \ell_{i}.$ 
\COMMENT{
Note that $m = a m' + b$ for some $0\leq b \leq m'-1$. 
Then \eqref{eq: m m' size difference} implies that $a=(a^{\sP}_1/ a^{\sQ}_1)$ and $0\leq b < \frac{a^{\sP}_1}{ a^{\sQ}_1}$.
Now, $$Q^{(1)}_j=\sL''(j,0) \cup \bigcup_{r\in [\ell'_j]}\sL'(j,r)\cup \bigcup_{i\colon H^{(1)}_i\subseteq Q^{(1)}_j, r\in [\ell_{i}]} \sL(i,r)$$ is partitioned into sets of size exactly $m'$   ($\sL'(j,r)$  and $\sL(i,r)$ )  and a set $\sL''(j,0)$ of size less than $m'$.  
(Note that the right hand size of the displayed equation is union of disjoint sets)
Since $|Q^{(1)}_j| \in \{m, m+1\}$ and $\frac{a^{\sP}_1}{ a^{\sQ}_1} < m' -1$, it means that we have $a^{\sP}_1/ a^{\sQ}_1$ sets of size $m'$ and one set  $\sL''(j,0)$ of size exactly $b$ or $b+1$. Thus we have $a^{\sP}_1/ a^{\sQ}_1 
= \ell'_j + \sum_{i\colon H^{(1)}_i\subseteq Q^{(1)}_j} \ell_{i}$ and
$|\sL''(j,0)|\leq b+1 \leq a^{\sP}_1/ a^{\sQ}_1$.
}

Hence, by distributing at most one vertex from $\sL''(j,0)$ into each of the sets in $\{ \sL'(j,r) : r\in [\ell'_j]\} \cup \{\sL(i,r) : H^{(1)}_i\subseteq Q^{(1)}_j, r\in [\ell_{i}]\}$, 
we can obtain the following collection 
$$\{L^{(1)}(j,1),\dots, L^{(1)}(j,a^{\sP}_1/ a^{\sQ}_1)\}$$ 
of sets of size $m'$ or $m'+1$, which forms an equipartition of $Q^{(1)}_j$.\COMMENT{All such sets could have size $m'+1$ and it is OK.}
Let $$\sP^{(1)}:= \{ L^{(1)}(j,r) : j\in [a_1^{\sQ}], r\in [a^{\sP}_1/ a^{\sQ}_1]\}.$$
Then (P1)$^1_{\ref{lem:refreg}}$ holds.
By construction, for each $L^{(1)} \in \sP^{(1)}$, either there exists $(i,r)\in [s]\times [\ell_{i}]$ such that $|L^{(1)}\setminus \sL(i,r)|\leq 1$ or there exists $(j,r)\in [a_1^{\sQ}]\times [\ell'_j]$ such that $|L^{(1)}\setminus \sL'(j,r)|\leq 1$. 
In the former case, let $f(L^{(1)}) := H^{(1)}_i$ and the latter case, let $f(L^{(1)})$ be an arbitrary set in $\sH^{(1)}$. 
Then
 \begin{eqnarray*}
\sum_{L^{(1)}\in \sP^{(1)}} |L^{(1)}\setminus f(L^{(1)})|
&\leq&  |\sP^{(1)}|+\sum_{(j,r)\in [a^{\sQ}_1]\times [\ell_j']} |\sL'(j,r)|
\leq  a^{\sP}_1  + \sum_{ j\in [a^{\sQ}_1], H^{(1)}_i\subseteq Q^{(1)}_j} |\sL(i,0)|\nonumber \\
&\leq & a^{\sP}_1 + s a^{\sQ}_1 m'\leq  \nu^2 |V|.
\end{eqnarray*}
The final inequality follows since $n\geq n_0$.
This and the construction of $\sP^{(1)}$ shows that
$\sP^{(1)} \prec_{\nu^2} \sH^{(1)}$ and $\sP^{(1)} \prec \sQ^{(1)}.$
This shows that (P2)$^1_{\ref{lem:refreg}}$ holds and  thus completes the proof of Claim~\ref{cl: 6.2 k=1 base case}.
\end{proof}

So assume that $k\geq 2$ and $L_{\ref{lem:refreg}}(k-1)$ holds. 
Let $\mu_{\ref{RAL(k)}}, t_{\ref{RAL(k)}}, n_{\ref{RAL(k)}}$ be the functions defined in Lemma~\ref{RAL(k)}.
By decreasing the value of $\epsilon(\ba)$ if necessary, we may assume that for all $\ba\in \N^{k-1}$, we have
\begin{align}\label{eq: epsilon function small}
\epsilon(\ba)\ll 1/s, 1/k, 1/\|\ba\|_{\infty}.
\end{align}

If $k=2$, let $T:=o^{4^k+1}\lceil 2\nu^{-2}\rceil$. 
If $k\geq 3$, for each $\ba \in \mathbb{N}^{k-2}$, let $T=T(\ba,o,\nu) =\max\{ \norm{\ba}, o^{4^k+1}\lceil 2\nu^{-2}\rceil\}.$
If $k\geq 3$, then we also let $\mu':\mathbb{N}^{k-2}\rightarrow (0,1]$ be a function such that for any $\ba \in \mathbb{N}^{k-2}$, we have
\begin{align}\label{eq: mu' def}
\mu'(\ba)\ll \nu, 1/k, 1/o, 1/\norm{\ba} \enspace \text{ and } \enspace \mu'(\ba) <  (\mu_{\ref{RAL(k)}}(k,T ,2s T^{2^{k}}  ,\eta,\nu/3,\epsilon^2))^2.
\end{align}
For all $k\geq 2$, let
\begin{align}\label{eq: def t k-1 def}
t_{k-1}:= \left\{\begin{array}{ll}
  s o^{4^k+1} \lceil 2 \nu^{-2} \rceil & \text{ if }k=2,\\
\max\{ t_{\ref{lem:refreg}}(k-1,o, o^{4^k},\eta,\nu/3,\mu'), o^{4^k+1} \lceil 2 \nu^{-2}  \rceil\} & \text{ if }k\geq 3,
\end{array}\right.
\end{align}
which exists by the induction hypothesis.
Choose an integer $t$ such that
\begin{align}\label{eq: t def}
1/t \ll 1/t_{\ref{RAL(k)}}(k,t_{k-1},2s t_{k-1}^{2^{k}},\eta,\nu/3,\epsilon^2), 1/t_{k-1},
\end{align}
and choose $\mu>0$ such that
\begin{align}\label{eq: mu def}
\mu \ll  \left\{\begin{array}{ll}
 1/t, \mu_{\ref{RAL(k)}}(k,T ,2s T^{2^k}  ,\eta,\nu/3,\epsilon^2) & \text{ if }k=2,\\
1/t, \mu'(\ba), \epsilon(\ba'), \mu_{\ref{lem:refreg}}(k-1,o,o^{4^k},\eta,\nu/3,\mu')  & \text{ if }k\geq 3.\\
\text{ for any } \ba \in [t]^{k-2}, \ba'\in [t]^{k-1}&
\end{array}\right.
\end{align}
\COMMENT{Here, we include $ \mu_{\ref{RAL(k)}}(2,T ,s T^{4}  ,\eta,\nu/3,\epsilon^2)$ for the case of $k=2$ to obtain a partition $\sP$ satisfying (P$'$1)--(P$'$5) later in the proof. (because we only let $\mu\ll \mu'(\ba)$ for $k\geq 3$.) }
Finally, choose an integer $n_0$ such that
\begin{align}\label{eq: n0 def}
1/n_0\ll \left\{\begin{array}{ll}
  1/n_{\ref{RAL(k)}}(k,t_{k-1},2s t_{k-1}^{2^{k}},\eta,\nu/3,\epsilon^2), 1/n_{\ref{lem:refreg}}(k-1,o,o^{4^k},\nu), 1/\mu. & \text{ if }k=2,\\
1/n_{\ref{RAL(k)}}(k,t_{k-1},2s t_{k-1}^{2^{k}},\eta,\nu/3,\epsilon^2), 1/n_{\ref{lem:refreg}}(k-1,o,o^{4^k},\eta,\nu/3,\mu'), 1/\mu & \text{ if }k\geq 3.
\end{array}\right.
\end{align}
Suppose $\sO(k-1,\ba^{\sO})$, $\sQ(k,\ba^{\sQ})$ and $\sH^{(k)}$ are given (families of) partitions satisfying (O1)$_{\ref{lem:refreg}}$--(O4)$_{\ref{lem:refreg}}$ with $\mu, t, n_0$ as defined above.
Write
\begin{align*}
	\sO^{(k-1)}&=\{O^{(k-1)}_1,\dots, O^{(k-1)}_{s_O}\} \enspace \text{ and }\enspace 
	\sQ^{(k-1)}=\{Q^{(k-1)}_1,\dots, Q^{(k-1)}_{s_Q}\}.
\end{align*}
Let
\begin{align}\label{eq:OQR}
	 O^{(k-1)}_{s_O+1}&:= \binom{V}{k-1}\setminus \cK_{k-1}(\sO^{(1)}), \enspace \enspace 
	Q^{(k-1)}_{s_Q+1}:= \binom{V}{k-1}\setminus \cK_{k-1}(\sQ^{(1)}),\text{ and}\\
	\sR^{(k-1)}&:= \{ O^{(k-1)}_i\cap Q^{(k-1)}_j : i\in [s_O+1], j\in [s_Q+1]\} \setminus\{\emptyset\}.\notag
\end{align}
We also write $\sR^{(k-1)}= \{R^{(k-1)}_1,\dots, R^{(k-1)}_{s''}\}.$
See Figure~\ref{fig:regularity} for an illustration of the relationship of the different partitions defined in the proof.
\begin{figure}[t]
\centering
\begin{tikzpicture}


\node (Q1) at (1,5.5) {$Q_1^{(k-1)}$};
\node (Qj) at (5,5.5) {$Q_{g(j)}^{(k-1)}$};

\draw[thick] (3.5,0.6) rectangle (7,1.2);
\draw[thick, fill opacity=0.7,fill=gray!100] (2.3,1.2) rectangle (5.5,2.5);
\draw[thick] (5.5,1.2) rectangle (7.5,2.5);
\draw[thick] (3.3,2.5) rectangle (4.8,3.2);
\draw[thick] (4.8,2.5) rectangle  (6.3,3.2);
\draw[thick] (4.75,1.85) circle (1.1cm); 

\draw[ultra thick] 
(0,0) rectangle (8,5)
(2,0) --(2,5)
(4,0) --(4,5)
(6,0) --(6,5)
;

\node (O) at (3.1,1.7	) {$O_{i'}^{(k-1)}$};

\draw[thick] 
(11,2) +(100:2)--(4.45,2.92)
(11,2) +(-100:2)--(4.45,0.78);

\begin{scope}[shift={(11,2)}]

\draw[thick, fill opacity=0.4,fill=black!50] (-1.13,-1.28) rectangle (1.3,1.28);
\draw[thick, fill opacity=0.4,fill=black!50] (-1.13,-1.28) rectangle (1.1,1.48);
\draw[thick, fill opacity=0.6,fill=black!60] (-1.13,-1.28) rectangle (0.5,0.2);

\draw[thick] (0,0) circle (2cm);
\draw[ultra thick](125:2)--(-125:2);
\draw[thick](140:2)--(40:2)
(-140:2)--(-40:2)
(90:2)-- (90:1.3)
(1.3,1.3)--(1.3,-1.3)
;

\node (M) at (-0.5,2.5	) {$M_j^{(k-1)}$};
\node (R) at (2.9,1.0	) {$R_j^{(k-1)}$};
\node (L) at (2.3,1.6	) {$\sL(i,j)$};
\node (L0) at (-0.7,-2.3	) {$\sL(i,j,0)$};
\node (Lr) at (0.9,-2.8	) {$\sL(i,j,r)$};
\node (P) at (2.2,-1.9	) {$\in \sP^{(k-1)}$};

\draw[thick] 
(M) -- (-0.3,1.48)
(R) -- (1.28,0.7)
(L) -- (0.2,0.2)
(L0) -- (-1.05,-1.15)
(Lr) -- (-0.2,-1.15)
(P) -- (0.25,-0.7)
(0.5,-1.28)--(0.5,1.48)
(-1.13,0.2)--(1.3,0.2)

(0.02,0.2)--(0.02,-1.25)
(-0.46,0.2)--(-0.46,-1.25)
(-0.94,0.2)--(-0.94,-1.25)

(0.02,-0.9)--(0.5,-0.9)
(0.02,-0.55)--(0.5,-0.55)
(0.02,-0.2)--(0.5,-0.2)
;
\end{scope}

\end{tikzpicture}
\caption{An illustration of the cascade of partitions in the proof of Lemma~\ref{lem:refreg}.
}\label{fig:regularity}
\end{figure}
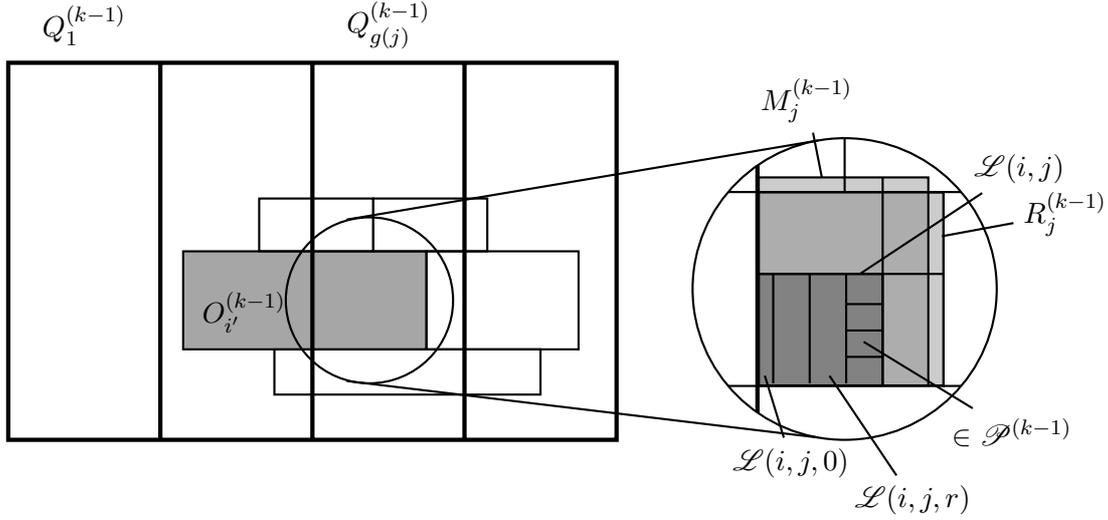

Since $\sO$ and $\sQ$ are both $o$-bounded, Proposition~\ref{prop: hat relation}(viii) implies that 
\begin{align}\label{eq: size s''}
s'' \leq \left(o^{2^k}\right)^{2} \leq o^{4^k}.
\end{align}

Now our aim is to construct a family of partitions $\sL$ as follows.
\begin{claim}\label{cl: claim 2 claim 2}
There exist $\{\sL^{(i)}\}_{i=1}^{k-1}$ and $\ba^{\sL}_{\rm long}=(a_1^{\sL},\ldots,a_{k-1}^{\sL}) \in [t_{k-1}]^{k-1}$ satisfying the following for all $j\in [k-1]$, where $\ba^{\sL}:= (a_1^{\sL},\dots, a_{k-2}^{\sL})$.
\begin{enumerate}[label= \rm (L$^*$\arabic*)]
\item\label{item:L*1} $\{\sL^{(i)}\}_{i=1}^{k-1}$ forms an $(\eta,\mu'(\ba^{\sL})^{1/2},\ba^{\sL}_{\rm long})$-equitable $t_{k-1}$-bounded family of partitions,
and $a_j^{\sQ}$ divides $a_j^{\sL}$,
\item\label{item:L*2} $\sL^{(j)} \prec \sQ^{(j)}$.
\item\label{item:L*3} $\sL^{(j)}\prec_{\nu/2}\sO^{(j)}$.
\end{enumerate}
\end{claim}
Note that if $k=2$, then the function $\mu'$ is not defined, but in this case, 
$\mu'(\ba^{\sL})^{1/2}$ plays no role in the definition of an equitable family of partitions (Definition~\ref{def: equitable family of partitions}) since Definition~\ref{def: equitable family of partitions}(iii) is vacuously true.
\begin{proof}
First we prove the claim for $k=2$. 
We apply $L_{\ref{lem:refreg}}(1)$ with $\sQ^{(1)}, \sR^{(1)}, s'', \nu/2$ playing the roles of $\sQ^{(1)}, \sH^{(1)}, s, \nu$.
(This is possible by~\eqref{eq: n0 def}.)
Then we obtain  
a partition $\sL^{(1)}$ of $V$ which satisfies (P1)$^1_{\ref{lem:refreg}}$ and (P2)$^1_{\ref{lem:refreg}}$. 
Moreover, \eqref{eq: def t k-1 def} implies that $\sL^{(1)}$ is $t_1$-bounded, i.e. $\ba^{\sL}_{\rm long} \in [t_{k-1}]$. 
Since $\sR^{(1)} \prec \sO^{(1)}$, this in turn implies \ref{item:L*1}--\ref{item:L*3}.

Now we assume $k\geq 3$. First, we apply $L_{\ref{lem:refreg}}(k-1)$ with  the following objects and parameters.
(This is possible by~\eqref{eq: def t k-1 def} and~\eqref{eq: mu def}--\eqref{eq: size s''}.)
\newline

{\small
\begin{tabular}{c|c|c|c|c|c|c|c|c|c|c}
 object/parameter & $V$ & $\{\sO^{(j)}\}_{j=1}^{k-2}$ & $\{\sQ^{(i)}\}_{i=1}^{k-1}$ & $\sR^{(k-1)}$ &
$o$ & $s''$ & $\eta$ & $\nu/3$ & $\mu'$ &$t_{k-1}$\\ \hline
 playing the role of & $V$ & $\sO$ & $\sQ$ & $\sH^{(k)}$ & $o$ & $s$ & $\eta$ & $\nu$ & $\epsilon$ &$t$
 \\ 
\end{tabular}
}\newline \vspace{0.2cm}

\noindent
Then we obtain a family of partitions $\sL=\sL(k-2,\ba^{\sL})$ and 
a partition $\sM^{(k-1)}$ of $\binom{V}{k-1}$ with $\sM^{(k-1)}=\{M^{(k-1)}_1,\dots, M^{(k-1)}_{s''}\}$
which satisfy the following for each $i\in [s'']$ and $j\in [k-2]$.

\begin{itemize}
\item[(L$'$1)] $\sL$ is $(\eta,\mu'(\ba^{\sL}), \ba^{\sL})$-equitable and $t_{k-1}$-bounded, and $a^{\sQ}_j$ divides $a^{\sL}_j$,
\item[(L$'$2)] $\sL^{(j)} \prec \sQ^{(j)}$ and $\sL^{(j)} \prec_{\nu/3} \sO^{(j)}$,

\item[(M$'$1)] $M^{(k-1)}_i$ is perfectly $\mu'(\ba^{\sL})$-regular with respect to $\sL$,
\item[(M$'$2)] $\sum_{i=1}^{s''} |M_i^{(k-1)}\triangle R^{(k-1)}_i| \leq (\nu/3)\binom{n}{k-1}$, and 
\item[(M$'$3)] $\sM^{(k-1)} \prec \sQ^{(k-1)}$ and if $R_i^{(k-1)} \subseteq \cK_{k-1}(\sQ^{(1)})$, then $M_i^{(k-1)}\subseteq \cK_{k-1}(\sQ^{(1)})$.
\end{itemize}
\vspace{0.2cm}

Thus $\{\sL^{(i)}\}_{i=1}^{k-2}$ satisfies \ref{item:L*1}--\ref{item:L*3} for $j\in [k-2]$ and it only remains to construct $\sL^{(k-1)}$.
Let 
\begin{align}\label{eq: def t'}
t':= \max \{\norm{\ba^{\sL}}, a^{\sQ}_{k-1} o^{4^{k}} \lceil 2\nu^{-2}  \rceil \}.
\end{align}
Thus $\sL$ is $t'$-bounded and $t'\leq \min\{ t_{k-1},  T(\ba^{\sL},o,\nu)\} $ by~\eqref{eq: def t k-1 def}. 
For convenience, we write $ \hat{\sL}^{(k-2)}=\{ \hat{L}^{(k-2)}_1,\dots, \hat{L}^{(k-2)}_{s'}\}$.
Since $\sM^{(k-1)}\prec \sQ^{(k-1)}$ by (M$'$3), 
for each $j\in [s'']$ there exists a unique $g(j)\in [s_Q+1]$ such that $M^{(k-1)}_j \subseteq Q^{(k-1)}_{g(j)}$. 
For each $i\in [s']$, $j\in [s'']$, we define
\begin{align}\label{eq: def of sL ij}
\sL(i,j) := \cK_{k-1}(\hat{L}^{(k-2)}_i)\cap M^{(k-1)}_j.
\end{align}
For each $i\in [s']$, we define
$J(i):= \{j\in [s'']: \sL(i,j)\neq \emptyset\}$.
Note that $\sL(i,j)\subseteq Q^{(k-1)}_{g(j)}$ for all $i\in [s']$.

\begin{subclaim}
\label{claim:Qregular}
For each $i\in [s']$ and for each $j\in J(i)$, 
the $(k-1)$-graph $Q^{(k-1)}_{g(j)}$ is $(\mu'(\ba^{\sL}),d'_{g(j)})$-regular with respect to $\hat{L}^{(k-2)}_i$, where $d'_{g(j)}\in \{1/a^{\sQ}_{k-1},1\}$. 
\end{subclaim}
\begin{proof}
First, note that since $\sL\prec\{\sQ^{(j)}\}_{j=1}^{k-2}$, one of the following holds.
\begin{enumerate}[label=(LL\arabic*)]
\item\label{item:LL1} There exists 
$\hat{Q}^{(k-2)}\in \hat{\sQ}^{(k-2)}$ such that $\hat{L}^{(k-2)}_i\subseteq \hat{Q}^{(k-2)}$.
\item\label{item:LL2} $\cK_{k-1}(\hat{L_i}^{(k-2)})\subseteq \binom{V}{k-1}\setminus \cK_{k-1}(\sQ^{(1)})$.
\end{enumerate}

If \ref{item:LL1} holds, then by (L$'$1) and the fact that $\mu \ll \mu'(\ba^{\sL}) \ll \norm{\ba^{\sL}}^{-k}$ 
we can apply Lemma~\ref{lem: counting} twice to conclude that 
$$|\cK_{k-1}(\hat{L}^{(k-2)}_i)|\geq t_{k-1}^{-2^k}|\cK_{k-1}(\hat{Q}^{(k-2)})| \stackrel{\eqref{eq: t def}, \eqref{eq: mu def}}{ \geq} \mu^{1/3} |\cK_{k-1}(\hat{Q}^{(k-2)})|.$$
Note that, for each $j\in J(i)$, $Q^{(k-1)}_{g(j)}$ is $(\mu,1/a_{k-1}^{\sQ})$-regular with respect to $\hat{Q}^{(k-2)}$.%
\COMMENT{$\cK_{k-1}(\hat{L}_i^{(k-2)})$ and $Q^{(k-1)}_{g(j)}$ intersect by the definition of $g(j)$ and $J(i)$, and $\hat{L}^{(k-2)}_i\subseteq \hat{Q}^{(k-2)}$. Thus $Q^{(k-1)}_{g(j)} \subseteq \cK_{k-1}(\hat{Q}^{(k-2)})$ by (L$'$2).} 
Together with Lemma~\ref{lem: simple facts 1}(ii) this implies that $Q^{(k-1)}_{g(j)}$ is $(\mu^{2/3},1/a^{\sQ}_{k-1})$-regular with respect to $\hat{L}^{(k-2)}_i$ for each $j\in J(i)$.

If \ref{item:LL2} holds, $j\in J(i)$ implies that $M^{(k-1)}_j\nsubseteq \cK_{k-1}(\sQ^{(1)})$.
Thus~\eqref{eq:OQR} means that $M_j^{(k-1)}\cap Q^{(k-1)}_{s_Q+1}\neq\es$, which implies $g(j)=s_Q+1$.
Also \ref{item:LL2} implies that $\cK_{k-1}(\hat{L}_i^{(k-2)}) \subseteq Q^{(k-1)}_{s_Q+1}.$
Thus $Q^{(k-1)}_{s_Q+1}$ is $(0,1)$-regular with respect to $\hat{L}^{(k-2)}_i$.
This completes the proof of Subclaim~\ref{claim:Qregular}.
\end{proof}

Moreover, (M$'$1) implies that for each $i\in [s']$ and $j\in [s'']$, 
the $(k-1)$-graph $\sL(i,j)$ is $\mu'(\ba^{\sL})$-regular with respect to $\hat{L}^{(k-2)}_i$, and thus it is $(2\mu'(\ba^{\sL}), d(\sL(i,j) \mid \hat{L}^{(k-2)}_i) )$-regular with respect to $\hat{L}^{(k-2)}_i$.\COMMENT{It maybe $(\mu',d)$-regular with density $d'$ and $d\neq d'$. Then it is still $(2\mu',d')$-regular.}
Let 
\begin{align}\label{eq: def aLk-1}
a^{\sL}_{k-1} := a^{\sQ}_{k-1}  o^{4^k} \lceil 2 \nu^{-2} \rceil,
\end{align}
and $\ba^{\sL}_{\rm long} := (a_{1}^\sL,\ldots,a_{k-1}^\sL)$.
By \eqref{eq: def t'}, we conclude
\begin{align}\label{eq: t'-bounded}
\|\ba^{\sL}_{\rm long}\|_{\infty} = t' \leq \min\{ t_{k-1},  T(\ba^{\sL},o,\nu)\}.
\end{align}
Let $\ell_{i,j}:= \lfloor d(\sL(i,j) \mid \hat{L}^{(k-2)}_i)a^{\sL}_{k-1}\rfloor$, so $\ell_{i,j}=0$ if $j\notin J(i)$.
We now apply the slicing lemma (Lemma~\ref{lem: slicing}) to $\sL(i,j)$ for each $i\in [s']$ and $j\in [s'']$.%
\COMMENT{
{\small
\begin{tabular}{c|c|c|c|c|c|c}
object/parameter & $\sL(i,j)$ & $\hat{L}^{(k-2)}_i$ & $d(\sL(i,j) \mid \hat{L}^{(k-2)}_i)$ & $ \left( d(\sL(i,j) \mid \hat{L}^{(k-2)}_i) a^{\sL}_{k-1}\right)^{-1}$ & $\ell_{i,j} $  & $2\mu'(\ba^{\sL})$\\ \hline
playing the role of & $H^{(k)}$& $ H^{(k-1)}$ & $d$ &  $p_i$ & $s$ & $\epsilon$
 \\ 
\end{tabular}
}\newline \vspace{0.2cm}
}
We obtain edge-disjoint $(k-1)$-graphs $\sL(i,j,0),\dots,\sL(i,j,\ell_{i,j})$ such that 
\begin{enumerate}[label=($\sL$\arabic*)]
\item\label{item:sL1} $\sL(i,j) = \sL(i,j,0)\cup \bigcup_{r=1}^{\ell_{i,j}} \sL(i,j,r)$, 
\item\label{item:sL2} $\sL(i,j,r)$ is $(6\mu'(\ba^{\sL}),1/a^{\sL}_{k-1})$-regular with respect to $\hat{L}^{(k-2)}_i$ for each $r\in [\ell_{i,j}]$, and 
\item\label{item:sL3} $\sL(i,j,0)$ is $(6\mu'(\ba^{\sL}),d'_{i,j})$-regular with respect to $\hat{L}^{(k-2)}_i$, where $d'_{i,j}\leq 1/a^{\sL}_{k-1}$.\COMMENT{If $d(\sL(i,j) \mid \hat{L}^{(k-2)}_i)< 1/a^{\sL}_{k-1}$, then we have $\sL(i,j,0)=\sL(i,j) $.}
\end{enumerate}

Observe that $\sL(i,j,0)$ may not have density $1/a^{\sL}_{k-1}$.
Since we would like to achieve this density for all classes,
we now take the union of all these $(k-1)$-graphs and split this union into suitable pieces.
For all $i\in [s']$ and $p\in [s_Q+1]$, let 
$$\sL'(i,p) := \bigcup_{j\colon g(j)=p} \sL(i,j,0) 
= \left(\cK_{k-1}(\hat{L}^{(k-2)}_i)\cap Q^{(k-1)}_p \right)\setminus \left(\bigcup_{j,r\colon g(j)=p,r\in [\ell_{i,j}]} \sL(i,j,r)\right).$$\COMMENT{Here, we can just analyse the second expression, and use the union lemma~(Lemma~\ref{lem: union regularity}) to show $\sL'(i,p)$ is regular, and it is simpler. However, here we want to get the conclusion of $(\mu'(\ba^{\sL})^{2/3}, \ell'_{i,p}/a^{\sL}_{k-1})$-regular with exactly right number $\ell'_{i,p}/a^{\sL}_{k-1}$, so that we can make sure $p_0=0$ in the next application of Lemma~\ref{lem: slicing}.
Thus we instead analyse the third expression.}
Note that if $p\notin g(J(i))$, then $\sL'(i,p)=\emptyset$. So suppose that $p\in g(J(i))$.
Then
 \begin{equation}\label{eq:L'ip}
 \begin{minipage}[c]{0.8\textwidth}\em
$\sL'(i,p)$ is $(\mu'(\ba^{\sL})^{2/3}, \ell'_{i,p}/a^{\sL}_{k-1})$-regular with respect to $\hat{L}^{(k-2)}_i$ 
for some $\ell'_{i,p}\in \N$.
\end{minipage}\ignorespacesafterend 
\end{equation} 
Indeed the union lemma~(Lemma~\ref{lem: union regularity}) (applied with $\sum_{j\colon g(j)=p}\ell_{i,j} \leq s''a^{\sL}_{k-1}  \leq o^{4^k}a_{k-1}^{\sL} $ playing the role of $s$) implies that 
$\bigcup_{j,r\colon g(j)=p,r\in [\ell_{i,j}]} \sL(i,j,r)$ is  $(\mu'(\ba^{\sL})^{3/4}, \sum_{j\colon g(j)=p}\ell_{i,j}/a_{k-1}^{\sL})$-regular with respect to $\hat{L}^{(k-2)}_i$.
In addition, by Subclaim~\ref{claim:Qregular}, 
$\cK_{k-1}(\hat{L}^{(k-2)}_i)\cap Q^{(k-1)}_p$ is $(\mu'(\ba^{\sL}), d'_p)$-regular with respect to $\hat{L}^{(k-2)}_i$ for some $d'_p\in \{1/a^{\sQ}_{k-1},1\}$. 
Note that \eqref{eq: def of sL ij} implies
$$\bigcup_{j,r\colon g(j)=p, r\in [\ell_{i,j}]} \sL(i,j,r)\subseteq \cK_{k-1}(\hat{L}^{(k-2)}_i)\cap Q^{(k-1)}_p.$$ 
So Lemma~\ref{lem: simple facts 1}(iii) implies that \eqref{eq:L'ip} holds
where $\ell'_{i,p}:= a^{\sL}_{k-1}d'_p - \sum_{j: g(j)=p}\ell_{i,j}$. 
(Note $\ell'_{i,p}\in \N$ since $d'_p \in \{1/a^{\sQ}_{k-1},1\}$ and $a^{\sQ}_{k-1} \mid a^{\sL}_{k-1}$.)

In addition, for all $i\in [s']$ and $p\in g(J(i))$, we have
\begin{align}\label{eq: size of sL'(i,p)}
|\sL'(i,p)|
\stackrel{(\sL3)}{\leq} |g^{-1}(p)| \cdot \frac{5}{4a^{\sL}_{k-1}} \cdot |\cK_{k-1}(\hat{L}^{(k-2)}_i)| 
\stackrel{(\ref{eq: size s''}),(\ref{eq: def aLk-1})}{\leq} 
\nu^2 |\cK_{k-1}(\hat{L}^{(k-2)}_i)\cap Q^{(k-1)}_p|.
\end{align}\COMMENT{For the last inequality :
Note that $Q^{(k-1)}_p$ is $(\mu'(\ba^{\sL}),d'_{p})$-regular with respect to $\hat{L}_i^{(k-2)}$ with $d'_{p} \in \{1/a^{\sQ}_{k-1}, 1\}$. Thus
$|\cK_{k-1}(\hat{L}^{(k-2)}_i)\cap Q^{(k-1)}_p| \geq ( 1/a^{\sQ}_{k-1} - \mu'(\ba^{\sL}))  |\cK_{k-1}(\hat{L}^{(k-2)}_i)| \geq 0.8/a^{\sQ}_{k-1}|\cK_{k-1}(\hat{L}^{(k-2)}_i)|$.
This implies that 
$$ |g^{-1}(p)| \cdot \frac{5}{4a^{\sL}_{k-1}} \cdot |\cK_{k-1}(\hat{L}^{(k-2)}_i)| 
\leq  s'' \cdot (5/4) \cdot (a_{k-1}^{\sQ})^{-1}o^{-4^k}\lceil 2\nu^{-2} \rceil^{-1}  (5/4) a^{\sQ}_{k-1} |\cK_{k-1}(\hat{L}^{(k-2)}_i)\cap Q_p^{(k-1)}| 
\leq \nu^2 |\cK_{k-1}(\hat{L}^{(k-2)}_i)\cap Q_p^{(k-1)}| $$
}
Again, we apply the slicing lemma (Lemma~\ref{lem: slicing}), this time to $\sL'(i,p)$.%
\COMMENT{
{\small
\begin{tabular}{c|c|c|c|c|c|c}
object/parameter & $\sL'(i,p)$ & $\hat{L}^{(k-2)}_i$ & $\ell'_{i,p}/a^{\sL}_{k-1}$ & $ \ell'^{-1}_{i,p}$ & $\ell'_{i,p} $ & $\mu'(\ba^{\sL})^{2/3}$\\ \hline
playing the role of & $H^{(k)}$& $ H^{(k-1)}$ & $d$ &  $p_i$ & $s$ & $\epsilon$
 \\ 
\end{tabular}
}\newline \vspace{0.2cm}
}
By \eqref{eq:L'ip},
we obtain edge-disjoint $(k-1)$-graphs $\sL'(i,p,1),\dots, \sL'(i,p,\ell'_{i,p})$ such that 
\begin{itemize}
\item[($\sL'$1)] $\sL'(i,p) = \bigcup_{\ell=1}^{\ell'_{i,p}} \sL'(i,p,\ell)$, 
\item[($\sL'$2)] $\sL'(i,p,\ell)$ is $(\mu'(\ba^{\sL})^{1/2},1/a^{\sL}_{k-1})$-regular with respect to $\hat{L}^{(k-2)}_i$ for each $\ell\in [\ell'_{i,p}]$.
\end{itemize}
Thus for each $i\in [s']$, \eqref{eq: def of sL ij}, ($\sL1$) and ($\sL'1$) imply that
\begin{align*}
	\bigcup_{p\in g(J(i))}\left(\{\sL(i,j,r) : j, r \text{ with } g(j)=p, r\in [\ell_{i,j}]\}\cup \{\sL'(i,p,1),\dots, \sL'(i,p,\ell_{i,p}') \}\right)
\end{align*}
forms a partition of $\cK_{k-1}(\hat{L}^{(k-2)}_i)$ into $a^{\sL}_{k-1}$ edge-disjoint $(k-1)$-graphs,\COMMENT{That we get exactly $a_{k-1}^{\sL}$ parts follows automatically from the regularity.} each of which is $(\mu'(\ba^{\sL})^{1/2},1/a_{k-1}^{\sL})$-regular with respect to $\hat{L}^{(k-2)}_i$,
where the latter follows from \ref{item:sL2} and ($\sL'2$).
We define
$$\sL^{(k-1)}
:=\bigcup_{ i\in [s'], j\in J(i) } \{\sL(i,j,r) :r\in [\ell_{i,j}] \}\cup  \bigcup_{ i\in [s'],  p\in g(J(i)) } \{\sL'(i,p,\ell):\ell\in [\ell'_{i,p}] \}.$$ 
Then \ref{item:L*1} follows from (L$'$1) and the construction of $\sL^{(k-1)}$ ($t_{k-1}$-boundedness follows by~\eqref{eq: t'-bounded}).
Note that for all $i\in [s']$, $j\in [s'']$, $r\in[\ell_{i,j}]$, $p\in [s_Q+1]$, $\ell\in [\ell_{i,p}']$,
we have $\sL(i,j,r)\subseteq Q^{(k-1)}_{g(j)}$ and $\sL'(i,p,\ell)\subseteq Q^{(k-1)}_{p}$, and so \ref{item:L*2} holds.

\begin{subclaim}\label{claim:Lprec}
$\sL^{(k-1)}\prec_{\nu/2}\sO^{(k-1)}$.
\end{subclaim}
\begin{proof}
To prove the subclaim, we define a suitable function $f_{k-1}: \sL^{(k-1)}\to \sO^{(k-1)}$.
For each $j\in [s'']$,
let $h(j) \in [s_O+1]$ be the index such that $R^{(k-1)}_j =  O^{(k-1)}_{h(j)}\cap Q^{(k-1)}_{p}$ for some $p\in [s_Q+1]$. 
For each 
$i\in [s']$, $j\in J(i)$,
$r\in [\ell_{i,j}]$, $\ell\in [\ell'_{i,g(j)}]$, let
$$f_{k-1}(\sL(i,j,r)) := O^{(k-1)}_{h(j)} \text{ and }f_{k-1}(\sL'(i,g(j),\ell)) := O^{(k-1)}_{h(j)}.$$ 
For fixed $j\in [s'']$, \eqref{eq: def of sL ij} and \ref{item:sL1} imply that
\begin{align}\label{eq:Lijr}
	\bigcup_{i\in [s'],r\in [\ell_{i,j}]} \sL(i,j,r) \subseteq M^{(k-1)}_j.
\end{align}
Hence
\begin{eqnarray*}
& & \hspace{-2.4cm} \sum_{L^{(k-1)}\in \sL^{(k-1)}} |L^{(k-1)}\setminus f_{k-1}(L^{(k-1)})|
 \leq \sum_{i,j,r} |\sL(i,j,r)\setminus f_{k-1}(\sL(i,j,r)) | + \sum_{i,p,\ell} |\sL'(i,p,\ell)| \nonumber \\
&\stackrel{\eqref{eq:Lijr},(\sL'1)}{\leq}& \sum_{j\in [s'']} |M^{(k-1)}_j \setminus O^{(k-1)}_{h(j)} | + \sum_{i,p} |\sL'(i,p)|\nonumber \\
&\stackrel{\eqref{eq: size of sL'(i,p)}}{\leq}&  \sum_{j\in [s'']} |M^{(k-1)}_j \setminus R^{(k-1)}_{j} | + \nu^2\sum_{i,p}  |\cK_{k-1}(\hat{L}^{(k-2)}_i)\cap Q^{(k-1)}_p|\nonumber\\
&\stackrel{{\rm(M}'{\rm2)}}{\leq}& \frac{\nu}{3}\binom{n}{k-1} + \nu^2 \binom{n}{k-1} 
\leq \frac{2\nu}{5}\binom{n}{k-1}.  
\end{eqnarray*}\COMMENT{Here, the second inequality holds since $\sL(i,j,r)$ are all pairwise disjoint.}
The fact that $a_1^{\sL} \geq \eta^{-1}$ and \eqref{eq: eta a1} together imply that $|\cK_{k-1}(\sL^{(1)})|\geq \frac{4}{5}\binom{n}{k-1}$, so the subclaim follows.
\end{proof}
This shows that \ref{item:L*3} holds and completes the proof of Claim~\ref{cl: claim 2 claim 2}.
\end{proof}
Note that $\{\sL^{(i)}\}_{i=1}^{k-1}$ obtained in Claim~\ref{cl: claim 2 claim 2}
naturally defines $\hat{\sL}^{(k-1)}$.
Write $\hat{\sL}^{(k-1)}=\{ \hat{L}^{(k-1)}_1, \dots, \hat{L}^{(k-1)}_{\hat{s}_{\sL}}\}.$
We now construct $\sL^{(k)}$ by refining $\cK_k(\hat{L}_i^{(k-1)})$ for all $i\in [\hat{s}_{\sL}]$.

\begin{claim}\label{claim:hatsL}
For each $i\in [\hat{s}_{\sL}]$,
there is a  partition $\{L^{(k)}(i,1),\dots, L^{(k)}(i,a^{\sQ}_k)\}$ of $\cK_k(\hat{L}^{(k-1)}_i)$ such that 
$L^{(k)}(i,r)$ is $(\mu^{1/2},1/a^{\sQ}_k)$-regular with respect to $\hat{L}^{(k-1)}_i$ for each $r\in [a^{\sQ}_k]$.
Moreover, we can ensure that $\{ L(i,r): i\in [\hat{s}_{\sL}], r\in [a^{\sQ}_k] \} \prec \sQ^{(k)}$.
\end{claim}
\begin{proof}
Since $\sL^{(1)} \prec \sQ^{(1)}$, for each $\hat{L}^{(k-1)}_i \in 	\hat{\sL}^{(k-1)}$, either $\cK_k(\hat{L}^{(k-1)}_i) \subseteq \cK_k(\sQ^{(1)})$ or $\cK_k(\hat{L}^{(k-1)}_i) \subseteq \binom{V}{k}\setminus\cK_k(\sQ^{(1)}).$

Suppose first that $\cK_k(\hat{L}^{(k-1)}_i) \subseteq \cK_k(\sQ^{(1)})$. 
As $\sL^{(k-1)} \prec \sQ^{(k-1)}$, 
there exists a (unique) $\hat{Q}^{(k-1)}_j \in \hat{\sQ}^{(k-1)}$ such that $\hat{L}^{(k-1)}_i\subseteq \hat{Q}^{(k-1)}_j$.
In addition, there are exactly $a^{\sQ}_{k}$ many $k$-graphs $Q^{(k)}(j,1),\dots, Q^{(k)}(j,a^{\sQ}_k)$ in $\sQ^{(k)}$ that partition $\cK_k(\hat{Q}^{(k-1)}_j)$.
For each $r\in [a^{\sQ}_k]$, let $$L^{(k)}(i,r):= Q^{(k)}(j,r)\cap \cK_k(\hat{L}^{(k-1)}_i).$$ 
Hence $\{L^{(k)}(i,1),\dots, L^{(k)}(i,a^{\sQ}_k)\}$ forms a partition of $\cK_k(\hat{L}^{(k-1)}_i)$.
We can now apply Lemma~\ref{lem: counting} twice and use \ref{item:L*1} as well as (O3)$_{\ref{lem:refreg}}$ to obtain that 
$$|\cK_k(\hat{L}^{(k-1)}_i)|\geq t_{k-1}^{-2^{k}} |\cK_k(\hat{Q}^{(k-1)}_j)|.$$
Thus Lemma~\ref{lem: simple facts 1}(ii), (O3)$_{\ref{lem:refreg}}$ and \eqref{eq: mu def} imply that 
$L^{(k)}(i,r)$ is $(\mu^{1/2},1/a^{\sQ}_k)$-regular with respect to $\hat{L}^{(k-1)}_i$ for each $r\in [a^{\sQ}_k]$.

Suppose next that we have $\cK_k(\hat{L}^{(k-1)}_i) \subseteq \binom{V}{k}\setminus\cK_k(\sQ^{(1)})$. 
We apply the slicing lemma (Lemma~\ref{lem: slicing}) with $\cK_k(\hat{L}^{(k-1)}_i), \hat{L}^{(k-1)}_i, 1, 1/a^{\sQ}_k$ playing the roles of $H^{(k)}, H^{(k-1)}, d, p_i$ respectively.
We obtain a partition $\{L^{(k)}(i,1),\dots, L^{(k)}(i,a^{\sQ}_k)\}$ of $\cK_k(\hat{L}^{(k-1)}_i)$ such that 
$L^{(k)}(i,r)$ is $(\mu^{1/2},1/a^{\sQ}_k)$-regular with respect to $\hat{L}^{(k-1)}_i$ for each $r\in [a^{\sQ}_k]$. 

The moreover part of Claim~\ref{claim:hatsL} is immediate from the construction in both cases.
\end{proof}

Let
$$\sL^{(k)}:= \{ L^{(k)}(i,r): i\in [\hat{s}_{\sL}], r\in [a^{\sQ}_k] \}, \enspace \ba^{\sL_*}:= (a^{\sL}_1,\dots, a^{\sL}_{k-1},a^{\sQ}_k) \enspace \text{ and } \enspace \sL_*:= \{\sL^{(i)}\}_{i=1}^{k},$$
$$\sJ^{(k)}:= \left( \{ H^{(k)}_i \cap L_*^{(k)} : i\in [s], L_*^{(k)}\in \sL^{(k)}\} \cup \{H^{(k)}_i \setminus \cK_{k}(\sL^{(1)}): i\in [s]\}\right)\setminus \{\emptyset\}.$$ 
\begin{align}\label{eq: sJ k j' for each i}
 \{J^{(k)}_1, \dots, J^{(k)}_{s_J}\}:= \sJ^{(k)}, \text{ and } J'(i):= \{ j' \in [s_J]: J^{(k)}_{j'} \subseteq H^{(k)}_i\} \text{ for each } i\in[s].
 \end{align}
Then $\sL^{(k)} \prec \sQ^{(k)}$. 
Let $\mu_* := \mu^{1/2}$ if $k=2$ and $\mu_*:= \mu'(\ba^{\sL})^{1/2} > \mu^{1/2}$ if $k\geq 3$. 
Then by \ref{item:L*1}, Claim~\ref{claim:hatsL}, and~\eqref{eq: t'-bounded}, we have
\begin{equation}\label{eq:sLkk}
\begin{minipage}[c]{0.8\textwidth}\em
 $\sL_*$ is a $(1/a_1^{\sL}, \mu_*,\ba^{\sL_*})$-equitable $t'$-bounded family of partitions.
\end{minipage}
\end{equation}
Moreover, $s_J\leq 2 s t'^{2^k}$ by \eqref{eq: t'-bounded} and Proposition~\ref{prop: hat relation}(viii).
Also we have $\sJ^{(k)} \prec \sH^{(k)}$ and $\{J'(1), \dots, J'(s)\}$ forms a partition of $[s_J]$.

Our next aim is to apply Lemma~\ref{RAL(k)} with the following objects and parameters.\COMMENT{
Cannot use $t_{k-1}$ instead of $t'$ since in~\eqref{eq: mu' def} we only have $\mu'(\ba) <  (\mu_{\ref{RAL(k)}}(k,T ,2s T^{2^{k}}  ,\eta,\nu/3,\epsilon^2))^2$.
}
\newline

{\small
\begin{tabular}{c|c|c|c|c|c|c|c|c|c}
object/parameter & $\sL_*$ & $\sJ^{(k)}$ & $t$ & $t'$ & $s_J $ & $\eta$ & $\nu/3$ & $\epsilon^2$& $\mu^*$ \\ \hline
playing the role of & $\sQ$ & $\sH^{(k)}$ & $t$ & $o$ & $s$ & $\eta$ & $\nu$ & $\epsilon$& $\mu$
 \\ 
\end{tabular}
}\newline \vspace{0.2cm}

Indeed, we can apply Lemma~\ref{RAL(k)}:
\eqref{eq: n0 def} ensures that (O1)$_{\ref{RAL(k)}}$ holds;
by \eqref{eq: mu' def}, \eqref{eq: mu def}, \eqref{eq: t'-bounded} and \eqref{eq:sLkk},
$\sL_*$ satisfies (O2)$_{\ref{RAL(k)}}$. 
By construction, $\sJ^{(k)} \prec \sL_*^{(k)}$, thus (O3)$_{\ref{RAL(k)}}$ also holds.
We obtain $\sP= \sP(k-1, \ba^{\sP})$ and $\sG'^{(k)}=\{ G'^{(k)}_1,\dots, G'^{(k)}_{s_J}\}$ satisfying the following.
\begin{enumerate}[label=(P$'$\arabic*)]
\item\label{eq:P'1} $\sP$ is $(\eta,\epsilon(\ba^{\sP})^2,\ba^{\sP})$-equitable and $t$-bounded, and $a^{\sL_*}_j$ divides $a^{\sP}_j$ for all $j\in [k-1]$,
\item\label{eq:P'2} for each $j\in [k-1]$, $\sP^{(j)} \prec  \sL^{(j)}$,
\item\label{eq:P'3} $G'^{(k)}_i$ is perfectly $\epsilon(\ba^{\sP})^2$-regular with respect to $\sP$ for all $i\in [s_J]$,
\item\label{eq:P'4} $\sum_{i\in[s_J]} |G'^{(k)}_i\triangle J^{(k)}_i| \leq (\nu/3)\binom{n}{k}$, and
\item\label{eq:P'5} $\sG'^{(k)} \prec \sL^{(k)}$ and if $J^{(k)}_i \subseteq \cK_k(\sL^{(1)})$, then $G'^{(k)}_i \subseteq \cK_k(\sL^{(1)})$.
\end{enumerate}
Here we obtain \ref{eq:P'1} from \eqref{eq: t def}.
In addition, we also have the following.
\begin{equation}\label{eq: Q perfectly regular wrt P}
\begin{minipage}[c]{0.8\textwidth}\em
$Q^{(k)}_{i'}$ is perfectly $\epsilon(\ba^{\sP})^2$-regular with respect to $\sP$ for all $i'\in [s_Q+1]$.
\end{minipage}
\end{equation}
Indeed, $\sP^{(k-1)}\prec \sL^{(k-1)} \prec \sQ^{(k-1)}$. Thus for each $\hat{P}^{(k-1)} \in \hat{\sP}^{(k-1)}$, either there exists unique $\hat{Q}^{(k-1)} \in \hat{\sQ}^{(k-1)}$ such that $\cK_k(\hat{P}^{(k-1)}) \subseteq \cK_{k}(\hat{Q}^{(k-1)})$, or $\cK_k(\hat{P}^{(k-1)}) \subseteq \binom{V}{k} \setminus \cK_k(\sQ^{(1)})$. 

In the former case, by two applications of Lemma~\ref{lem: counting} and \ref{eq:P'1}, it is easy to see that 
$$|\cK_k(\hat{P}^{(k-1)})|\geq t^{-2^{k}} |\cK_k(\hat{Q}^{(k-1)})|.$$
Thus (O3)$_{\ref{lem:refreg}}$ with Lemma~\ref{lem: simple facts 1}(ii) and \eqref{eq: mu def} implies that $Q^{(k)}_{i'}$ is $\epsilon(\ba^{\sP})^2$-regular with respect to $\hat{P}^{(k-1)}$.

Now suppose that $\cK_k(\hat{P}^{(k-1)}) \subseteq \binom{V}{k} \setminus \cK_k(\sQ^{(1)})$.
If $i'\in [s_Q]$, then we have $Q^{(k)}_{i'} \subseteq \cK_{k}(\sQ^{(1)})$. Thus $Q^{(k)}_{i'} \cap \cK_{k}(\hat{P}^{(k-1)}) =\emptyset,$ and $Q^{(k)}_{i'}$ is $(\epsilon(\ba^{\sP})^2,0)$-regular with respect to $\hat{P}^{(k-1)}$.
If $i'=s_Q+1$, then $Q^{(k)}_{i'} = \binom{V}{k}\setminus \cK_{k}(\sQ^{(1)})$, thus $Q^{(k)}_{i'}$ is $(\epsilon(\ba^{\sP})^2,1)$-regular with respect to $\hat{P}^{(k-1)}$.
Thus we have \eqref{eq: Q perfectly regular wrt P}.

It is easy to see that \ref{eq:P'1} and \ref{item:L*1}  imply (P1)$_{\ref{lem:refreg}}$. 
The statements \ref{eq:P'2}, \eqref{eq: prec triangle} together with \ref{item:L*2} and \ref{item:L*3} imply (P2)$_{\ref{lem:refreg}}$.

As $\sL^\kk \prec\sQ^\kk$ and \ref{eq:P'5} holds, we obtain $\sG'^\kk \prec\sQ^\kk$.
For each $i'\in [s_Q+1]$, let
$$G'(i'):= \{ j' \in[s_J] : G'^{(k)}_{j'} \subseteq Q^{(k)}_{i'}\},\enspace \text{and } \enspace H'(i'):= \{ i\in [s]: H^{(k)}_i \subseteq Q^{(k)}_{i'}\}.$$
Note that $\{G'(1), \dots G'(s_Q+1)\}$ forms a partition of $[s_J]$. Also by (O4)$_{\ref{lem:refreg}}$, 
$\{H'(1), \dots H'(s_Q+1)\}$ forms a partition of $[s]$. Moreover, both $G'(i')$ and $H'(i')$ are non-empty sets.
For each $i'\in[s_Q+1]$, 
we arbitrarily choose a representative $h'_{i'} \in H'(i')$.

Recall that $J'(i)$ was defined in \eqref{eq: sJ k j' for each i}.
For each $i'\in [s_Q+1]$ and $i\in H'(i')\setminus \{h'_{i'}\}$, we define
$$G^{(k)}_i := \bigcup_{j'\in J'(i)\cap G'(i')} G'^{(k)}_{j'} \enspace \text{ and } \enspace G^{(k)}_{h'_{i'}} :=  Q_{i'}^{(k)}\setminus \bigcup_{\ell \in H'(i')\setminus h'_{i'} } G^{(k)}_{\ell}.$$
Let
$$\sG^\kk:= \{ G^{(k)}_i :i\in [s]\}.$$
By the construction, $\sG^\kk$ forms a partition of $\binom{V}{k}$.
Moreover, we have the following:
\begin{equation}\label{eq: i,j',i' condi}
\begin{minipage}[c]{0.8 \textwidth} \em
Suppose that $i'\in [s_Q+1]$, $i \in H'(i')$ and $j'\in J'(i) \cap G'(i')$. 
Then $G'^{(k)}_{j'}\subseteq G^{(k)}_i$.
\end{minipage}
\end{equation}\COMMENT{If $i\neq h'_{i'}$, then it is trivial by the definition. If $i=h'_{i'}$, then still $G'^{(k)}_{j'}\subseteq Q_{i'}^{(k)}$ holds. Also 
$j'\notin J'(i_*)$ for any $i_*\neq h'_{i'}$, thus we get that
 $G'^{(k)}_{j'}\subseteq G^{(k)}_{h'_{i'}}$.
}
Note that the construction of $\sG^{(k)}$, \ref{eq:P'3}, the union lemma~(Lemma~\ref{lem: union regularity}),  \eqref{eq: epsilon function small} and \eqref{eq: t'-bounded} together imply that 
for each $i'\in [s_Q+1]$ and $i\in H'(i')\setminus\{h'_{i'}\}$, $G^{(k)}_i$ is perfectly $\epsilon(\ba^{\sP})^{3/2}$-regular with respect to $\sP$.\COMMENT{We need that $\epsilon(\ba^{\sP}) \ll 1/s_J$ here. But by \eqref{eq: epsilon function small} this holds since $s_J \leq 2st'^{2^k}$ and by \eqref{eq: t'-bounded} we have $t'= \norm{\ba^{\sL}_{\rm long}} \leq \norm{\ba^{\sP}}$. }
In particular, together with \eqref{eq: Q perfectly regular wrt P}, the union lemma~(Lemma~\ref{lem: union regularity}) and Lemma~\ref{lem: simple facts 1}(iii), this implies that
for each $i'\in [s_Q+1]$, $G^{(k)}_{h'_{i'}}$ is also perfectly $\epsilon(\ba^{\sP})$-regular with respect to $\sP$. Thus we obtain (G1)$_{\ref{lem:refreg}}$.

By the definition of $G^{(k)}_i$, 
we conclude that for every $i\in [s]$, 
there exists $i'\in [s_Q+1]$ such that $G^{(k)}_{i} \subseteq Q^{(k)}_{i'}$. 
Thus $\sG^{(k)} \prec \sQ^{(k)}$.
Moreover, 
if $H^{(k)}_{i} \subseteq \cK_{k}(\sQ^{(1)})$, then $i\in H'(i')$ for some $i' \neq s_Q+1$. 
Hence in this case $G^{(k)}_{i} \subseteq Q^{(k)}_{i'}$ with $i'\neq s_{Q}+1$, and so $G^{(k)}_{i}\subseteq \cK_{k}(\sQ^{(1)})$. 
Thus (G3)$_{\ref{lem:refreg}}$ holds.

We now verify (G2)$_{\ref{lem:refreg}}$.
Consider any edge $e \in G^{(k)}_{i} \setminus H^{(k)}_{i}$ for some $i\in  [s]$. 
We claim that 
\begin{align}\label{eq:einJG}
	e\in \bigcup_{j' \in [s_J] } J^{(k)}_{j'} \setminus G'^{(k)}_{j'}.
\end{align}
To prove \eqref{eq:einJG} note that
since $\sJ^{(k)}$ is a partition of $\binom{V}{k}$, there exists $j'\in [s_J]$ such that $e\in J^{(k)}_{j'}$. So \eqref{eq:einJG} holds if $e\notin G'^{(k)}_{j'}$. 
Thus assume for a contradiction that $e\in G'^{(k)}_{j'}$. 
Let $i_*\in [s]$ be the index such that $j'\in J'(i_*)$. 
Then $J^{(k)}_{j'} \subseteq H^{(k)}_{i_*}$.

Since $\{H'(1),\dots, H'(s_Q+1)\}$ forms a partition of $[s]$, 
there exists $i'\in [s_Q+1]$ such that $i_*\in H'(i')$.
Thus $e\in J^{(k)}_{j'}\subseteq  H^{(k)}_{i_*} \subseteq Q^{(k)}_{i'}$. 
Hence $Q^{(k)}_{i'}\cap G'^{(k)}_{j'}\neq \emptyset$.
Since $\sG'^{(k)}\prec \sQ^{(k)}$, 
this implies that $G'^{(k)}_{j'} \subseteq Q^{(k)}_{i'}$ and so $j'\in G'(i')$.
Consequently, we have $i_* \in H'(i')$ and $j'\in J'(i_*) \cap G'(i')$. This together with \eqref{eq: i,j',i' condi} implies that $e\in G'^{(k)}_{j'} \subseteq G^{(k)}_{i_*}$. Since $\sG^{(k)}$ is a partition of $\binom{V}{k}$, this implies that $i=i_*$.
But then $e\in H^{(k)}_{i_*} = H^{(k)}_{i}$, a contradiction.
This proves \eqref{eq:einJG}.

Then
\begin{align}\label{eq:GHJ}
	\sum_{i\in [s]} |G^{(k)}_i\setminus H^{(k)}_i| \stackrel{(\ref{eq:einJG})}{\leq}
	\sum_{j'\in [s_J] } |J'^{(k)}_{j'} \setminus G'^{(k)}_{j'} |.
\end{align}
Since all of $\sH^\kk, \sG^\kk, \sJ^\kk$ and $\sG'^\kk$ are partitions of $\binom{V}{k}$, we obtain 
$$\sum_{i\in[s]} |G^{(k)}_i\triangle H^{(k)}_i| = 2\sum_{i\in [s]} |G^{(k)}_i\setminus H^{(k)}_i|, \text{ and }
\sum_{i\in[s_J]} |G'^{(k)}_i\triangle J^{(k)}_i|= 2\sum_{i\in[s_J]} |J^{(k)}_i\setminus G'^{(k)}_i|.
$$\COMMENT{Every edge counted in the right side is counted twice for the left side. Note that if $v \in G^{(k)}_i\setminus J^{(k)}_i$ implies that $v\in J^{(k)}_{i'} \setminus G^{(k)}_{i'}$ for some $i'\neq i$. }
Thus we conclude
\begin{align*}
\sum_{i\in [s]} |G^{(k)}_i\triangle H^{(k)}_i|
=  2\sum_{i\in [s]} |G^{(k)}_i\setminus H^{(k)}_i| 
\stackrel{(\ref{eq:GHJ})}{\leq}
2\sum_{j'\in [s_J] } |J'^{(k)}_{j'} \setminus G'^{(k)}_{j'} |
\stackrel{\text{\ref{eq:P'4}}}{\leq}  \nu\binom{n}{k}.
\end{align*}
Thus (G2)$_{\ref{lem:refreg}}$ holds.
\end{proof}

Suppose we are given a $(k-1)$-graph $H$ on a vertex set $V$.
In the next lemma we apply Lemma~\ref{RAL(k)} to show that, given a different vertex set $V'$,
there exists another $(k-1)$-graph $F$ on $V'$ whose large scale structure is very close to that of~$H$.

\begin{lemma}\label{lem: mimic Kk}
Suppose $0< 1/m, 1/n \ll \epsilon \ll \nu, 1/o,  1/k \leq 1$ and $k,o \in \N\sm \{1\}$.
Suppose $\sP=\sP(k-1,\ba)$ and $\sQ=\sQ(k-1,\ba)$ are both $o$-bounded $(1/a_1,\epsilon,\ba)$-equitable families of partitions of $V$ and $V'$ respectively with $|V|=n$ and $|V'|=m$. 
Suppose that $H^{(k-1)} \subseteq \cK_{k-1}(\sP^{(1)})$. 
Then there exists a $(k-1)$-graph $G^{(k-1)}\subseteq \cK_{k-1}(\sP^{(1)})$ on $V$ and a $(k-1)$-graph $F^{(k-1)}\subseteq \cK_{k-1}(\sQ^{(1)})$ on $V'$ such that
\begin{itemize}
\item[{\rm (F1)$_{\ref{lem: mimic Kk}}$}]  $|H^{(k-1)} \triangle G^{(k-1)}| \leq \nu \binom{n}{k-1}$,
\item[{\rm (F2)$_{\ref{lem: mimic Kk}}$}] $d( \cK_{k}(G^{(k-1)}) \mid \hat{P}^{(k-1)}(\hat{\bz})) = d( \cK_{k}(F^{(k-1)}) \mid \hat{Q}^{(k-1)}(\hat{\bz})) \pm \nu$
for each $\hat{\bz}\in \hat{A}(k,k-1,\ba)$.
\end{itemize}
\COMMENT{Here, when $k=2$, $\cK_{k}(F^{(k-1)})$ only counts the crossing $2$-sets in $F^{(k-1)}$, i.e. pairs of vertices in $F^{(k-1)}$ such that two vertices lie in different set in $\sP^{(1)}$.}
\end{lemma}
To prove Lemma~\ref{lem: mimic Kk},
we first apply Lemma~\ref{RAL(k)}
to obtain a family of partitions $\sR=\sR(k-2,\ba^{\sR})$ and a $k$-graph $G^{(k-1)}$ as in {\rm (F1)$_{\ref{lem: mimic Kk}}$}.
We then `project' $\sR$ onto $V'$ (so that it refines $\sQ$).
This results in a partition $\sL$.
We then apply the slicing lemma to construct $F^{(k-1)}$ which respects $\sL$ (and in particular has the appropriate densities).
\begin{proof}[Proof of Lemma~\ref{lem: mimic Kk}]
First suppose $k=2$. 
Then $H^{(1)} \subseteq V$.
Let $G^{(1)}:= H^{(1)}$. 
Thus (F1)$_{\ref{lem: mimic Kk}}$ holds. 
Recall that for each $b\in [a_1]$, the vertex sets $P^{(1)}(b,b)$ and $Q^{(1)}(b,b)$ denote the $b$-th parts in $\sP^{(1)}$ and $\sQ^{(1)}$, respectively.
For each $b\in [a_1]$, 
let $F^{(1)}(b,b)$ be a subset of $Q^{(1)}(b,b)$ with 
$$|F^{(1)}(b,b)|= \left\lfloor \frac{m}{n}|H^{(1)}\cap P^{(1)}(b,b)| \right\rfloor$$
and let $F^{(1)}:=\bigcup_{b\in [a_1]}F^{(1)}(b,b)$.
For each $\bz=( \alpha_1, \alpha_2) \in \hat{A}(2,1,\ba)$, we have
\begin{align*}
d( \cK_{2}(F^{(1)}) \mid \hat{Q}^{(1)}(\hat{\bz})) 
& = \frac{|F^{(1)}(\alpha_1,\alpha_1)| |F^{(1)}(\alpha_2,\alpha_2)|}{ (m/a_1\pm 1)^2} \\
& =\frac{ ( |H^{(1)}\cap P^{(1)}(\alpha_1,\alpha_1)| \pm n/m) ( |H^{(1)}\cap P^{(1)}(\alpha_2,\alpha_2)| \pm n/m) }{ (n/a_1\pm n/m)^2} \\
&= d(\cK_{2}(G^{(1)}) \mid \hat{P}^{(1)}(\hat{\bz})) \pm \nu.
\end{align*}
\COMMENT{Penultimate equality: 
Let $A:=  |H^{(1)}\cap P^{(1)}(\alpha_1,\alpha_1)|, B:=  |H^{(1)}\cap P^{(1)}(\alpha_2,\alpha_2)|$ then $A,B\leq n/a_1+1$.\newline
\begin{align*}
&\frac{ ( |H^{(1)}\cap P^{(1)}(\alpha_1,\alpha_1)| \pm n/m) ( |H^{(1)}\cap P^{(1)}(\alpha_2,\alpha_2)| \pm n/m) }{ (n/a_1\pm n/m)^2} = \frac{AB\pm (A+B)n/m \pm n^2/m^2}{ (n/a_1\pm n/m)^2} \\
&= \frac{AB \pm 2n^2/(a_1m) \pm 2n/m \pm n^2/m^2}{ (n/a_1\pm n/m)^2 }
= \frac{AB}{ (n/a_1\pm n/m)^2 } \pm  ( 2.5 a_1/m + 2.5 a_1^2/mn + 2a_1^2/m^2 )
\\
&= \frac{AB}{ ( 1\pm a_1/m)^2 (n/a_1)^2 } \pm 3a_1/m
= (1\pm 3a_1/m)(1\pm 3a_1/n) \frac{AB}{ |P^{(1)}(\alpha_1)| |P^{(1)}(\alpha_2)|} \pm 3a_1/m \\
&=  \frac{AB}{ |P^{(1)}(\alpha_1)| |P^{(1)}(\alpha_2)|}  \pm 6a_1(1/m+1/n)\\
&=d(\cK_{2}(G^{(1)}) \mid \hat{P}^{(1)}(\hat{\bz})) \pm \nu.
\end{align*}
We get the last inequality since $ \frac{AB}{ |P^{(1)}(\alpha_1)| |P^{(1)}(\alpha_2)|} \leq 1$.
}
Thus (F2)$_{\ref{lem: mimic Kk}}$ holds.

Now we show the lemma for $k\geq 3$. 
Let $\eta'$ be a constant such that $\epsilon \ll \eta' \ll \nu,1/o, 1/k$.
Let $\epsilon':\mathbb{N}^{k-2}\rightarrow (0,1]$ be a function such that 
\begin{align}\label{eq: epsilon' def mimic}
\epsilon'(\bb)\ll  \nu, 1/o, 1/k, 1/\|\bb\|_{\infty}
\text{ for all }\bb\in\mathbb{N}^{k-2}.
\end{align}
Let $t:=t_{\ref{RAL(k)}}(k-1,o, o^{4^{k}} ,\eta',\nu,\epsilon').$
Since $\epsilon \ll \nu,1/o,1/k, \eta'$, we may assume that
\begin{align} \label{eq: mimic epsilon def}
0<\epsilon \ll \mu_{\ref{RAL(k)}}(k-1,o,o^{4^{k}},\eta',\nu,\epsilon'), 1/t, \min\{ \epsilon'(\bb) : \bb\in [t]^{k-2}\},
\end{align}
and we may assume that $n,m> n_0:= n_{\ref{RAL(k)}}(k-1,o,o^{4^{k}},\eta',\nu,\epsilon').$
Let 
\begin{align*}
\{H^{(k-1)}_1,\dots, H^{(k-1)}_{s}\}&:=\{ P^{(k-1)}\cap H^{(k-1)} : P^{(k-1)}\in \sP^{(k-1)}\}\setminus \{\emptyset\} ,\\
\{H^{(k-1)}_{s+1},\dots, H^{(k-1)}_{s+s'}\} 
&:=  \left( \{ P^{(k-1)}\setminus H^{(k-1)} : P^{(k-1)}\in \sP^{(k-1)}\} \right. \\
 & \hspace{0.8cm} \left. \cup \left\{ \binom{V}{k-1}\setminus \cK_{k-1}(\sP^{(1)})\right\} \right)\setminus \{\emptyset\},\\
\sH^{(k-1)}&:=\{H^{(k-1)}_1,\dots, H^{(k-1)}_{s+s'}\} .
\end{align*}
Hence $\sH^{(k-1)}$ is a partition of $\binom{V}{k-1}$ such that $\sH^{(k-1)}\prec \sP^{(k-1)}$ and 
$s +s' \leq 2o^{2^k}+1 \leq o^{4^{k}}$ by Proposition~\ref{prop: hat relation}(viii).
We first construct $G^{(k-1)}$.
By \eqref{eq: mimic epsilon def}, 
we may apply Lemma~\ref{RAL(k)} with the following objects and parameters.\newline

{\small
\begin{tabular}{c|c|c|c|c|c|c|c|c|c}
object/parameter & $\sP$ & $\sH^{(k-1)}$ & $o$ & $s+s'$ & $\eta'$ & $\nu$ & $\epsilon'$ & $k-1$&$t$ \\ \hline
playing the role of & $\sQ$ & $\sH^{(k)}$ & $o$ & $s$ & $\eta$ & $\nu$ & $\epsilon$ & $k$&$t$
 \\ 
\end{tabular}
}\newline \vspace{0.2cm}

\noindent
We obtain $\sR= \sR(k-2,\ba^{\sR})$ and $\sG^{(k-1)}=\{G^{(k-1)}_1,\dots, G^{(k-1)}_{s+s'}\}$ satisfying the following.
\begin{itemize}
\item[(R1)$_{\ref{lem: mimic Kk}}$] $\sR$ is $(\eta',\epsilon'(\ba^{\sR}),\ba^{\sR})$-equitable and $t$-bounded and for each $j\in [k-2]$, $a_j$ divides $a^{\sR}_j$, 
\item[(R2)$_{\ref{lem: mimic Kk}}$] $ \{\sR^{(j)}\}_{j=1}^{k-2}=\sR \prec \{\sP^{(j)}\}_{j=1}^{k-2}$,
\item[(R3)$_{\ref{lem: mimic Kk}}$] for each $i\in [s+s']$, $G^{(k-1)}_i$ is perfectly $\epsilon'(\ba^{\sR})$-regular with respect to $\sR$,
\item[(R4)$_{\ref{lem: mimic Kk}}$] $\sum_{i=1}^{s+s'} |G^{(k-1)}_i\triangle H^{(k-1)}_i| \leq \nu \binom{n}{k-1}$, and
\item[(R5)$_{\ref{lem: mimic Kk}}$] $\sG^{(k-1)} \prec \sP^{(k-1)}$, and 
if $H^{(k-1)}_i \subseteq \cK_{k-1}(\sP^{(1)})$, then $G_i^{(k-1)} \sub\cK_{k-1}(\sP^{(1)})$.
\end{itemize}
Observe that $a^{\sR}_1 > \eta'^{-1}$ by (R1)$_{\ref{lem: mimic Kk}}$. 
Thus 
\begin{align}\label{eq: a1 sr} 
1/a^{\sR}_1 \ll \nu, 1/o, 1/k.
\end{align}
Let $G^{(k-1)}:= \bigcup_{i=1}^{s} G_i^{(k-1)}.$
Then (F1)$_{\ref{lem: mimic Kk}}$ holds and $G^{(k-1)}\subseteq \cK_{k-1}(\sP^{(1)})$.

Next we show how to construct $F^{(k-1)}$.
To this end we define a family of partitions $\sL$ on $V'$ which has the same number of parts as $\sR$.
We apply Lemma~\ref{lem: partition refinement} with $\{\sQ^{(j)}\}_{j=1}^{k-2}, \epsilon, \ba^{\sR}$ playing the roles of $\sP,\epsilon,\bb$ to obtain $\sL$ so that\COMMENT{
 Lemma~\ref{lem: partition refinement} shows that $\sL$ is $(1/a_1^{\sR}, \epsilon, \ba^{\sR})$-equitable but since $a_1^{\sR} \geq \eta'^{-1}$, we get that $\sL$ is actually $(\eta',\epsilon,\ba^{\sR})$-equitable.}
\begin{equation*}
\begin{minipage}[c]{0.9\textwidth}\em
$\sL=\sL(k-2,\ba^{\sR})$ is an $(\eta',\epsilon^{1/3}, \ba^{\sR})$-equitable family of partitions such that $\sL\prec \{\sQ\}_{j=1}^{k-2}$.
\end{minipage}
\end{equation*}
Let $\ba' := (a_1,\dots, a_{k-2})$, where $\ba := (a_1,\dots, a_{k-1})$.
By taking an appropriate $\ba^{\sR}$-labelling for $\sL$, we may also assume that for each $\hat{\bx}\in \hat{A}(k-1,k-2,\ba')$,
\begin{align}\label{eq: A hat bx def}
A(\hat{\bx})&:= \{\hat{\by}\in \hat{A}(k-1,k-2,\ba^{\sR}) : \hat{R}^{(k-2)}(\hat{\by})\subseteq \hat{P}^{(k-2)}(\hat{\bx})\} \nonumber \\ &= \{\hat{\by}\in \hat{A}(k-1,k-2,\ba^{\sR})  : \hat{L}^{(k-2)}(\hat{\by})\subseteq \hat{Q}^{(k-2)}(\hat{\bx})\}.
\end{align}
For each  $\hat{\bx}\in \hat{A}(k-1,k-2,\ba')$ and $\hat{\by}\in A(\hat{\bx})$, 
Lemma~\ref{lem: counting} implies that
\begin{align}\label{eq: cK hat L big mimic}
|\cK_{k-1}(\hat{L}^{(k-2)}(\hat{\by}))| &\geq (1-\nu)\prod_{j=1}^{k-1} (a^{\sR}_j)^{-\binom{k-1}{j}} m^{k-1} 
\stackrel{(\ref{eq: mimic epsilon def})}{\geq} \epsilon^{1/3} (1+\nu)\prod_{j=1}^{k-1} (a^{\sQ}_j)^{-\binom{k-1}{j}} m^{k-1} \nonumber \\
&\geq  \epsilon^{1/3}|\cK_{k-1}(\hat{Q}^{(k-2)}(\hat{\bx}))|.
\end{align}
We would like the relative densities of $F^{(k-1)}$ (with respect to the polyads of $\sL$) to reflect the relative densities of $G^{(k-1)}$ (with respect to the polyads of $\sR$).
For this, we first determine the relative densities of $G^{(k-1)}$ (see~\eqref{eq: Gyb reg}).
For each $\hat{\by}\in \hat{A}(k-1,k-2,\ba^{\sR})$, $b\in [a_{k-1}]$, and
the unique vector $\hat{\bx}$ with $\hat{\by} \in A(\hat{\bx})$, we define
\begin{align}\label{eq: Q* P* def}
&Q^{(k-1)}_*(\hat{\by},b) := Q^{(k-1)}(\hat{\bx},b)\cap \cK_{k-1}(\hat{L}^{(k-2)}(\hat{\by})), \nonumber \\ 
&P^{(k-1)}_*(\hat{\by},b):= P^{(k-1)}(\hat{\bx},b)\cap \cK_{k-1}(\hat{R}^{(k-2)}(\hat{\by})).
\end{align}
Since each $Q^{(k-1)}(\hat{\bx},b) \in \sQ^{(k-1)}$ is $(\epsilon,1/a_{k-1})$-regular with respect to $\hat{Q}^{(k-2)}(\hat{\bx})$ for each $b\in [a_{k-1}]$,  
Lemma~\ref{lem: simple facts 1}(ii) and \eqref{eq: cK hat L big mimic} with the definition of $A(\hat{\bx})$ imply that 
\begin{equation}\label{eq: Q/L regular}
\begin{minipage}[c]{0.8\textwidth}\em
$Q^{(k-1)}_*(\hat{\by},b)$ is $(\epsilon^{2/3},1/a_{k-1})$-regular with respect to $\hat{L}^{(k-2)}(\hat{\by})$.
\end{minipage}
\end{equation}
Similarly,\COMMENT{ \eqref{eq: cK hat L big mimic} for $\hat{R}^{(k-2)}(\hat{\by})$ can be also proved using Lemma~\ref{lem: counting} and \eqref{eq: mimic epsilon def}.}
\begin{equation}\label{eq: P/R regular}
\begin{minipage}[c]{0.8\textwidth}\em
$P^{(k-1)}_*(\hat{\by},b)$ is $(\epsilon^{2/3},1/a_{k-1})$-regular with respect to $\hat{R}^{(k-2)}(\hat{\by})$.
\end{minipage}
\end{equation}
For each $\hat{\by}\in \hat{A}(k-1,k-2,\ba^{\sR})$ and $b\in [a_{k-1}]$,
let 
\begin{align}\label{eq: G hat by b def mimic}
G(\hat{\by},b) := G^{(k-1)} \cap P^{(k-1)}_*(\hat{\by},b).
\end{align}
Thus $G(\hat{\by},b) \subseteq \cK_{k-1}(\hat{R}^{(k-2)}(\hat{\by}))$ by \eqref{eq: Q* P* def}. 
Since $\sG^{(k-1)}\prec \sP^{(k-1)}$ by (R5)$_{\ref{lem: mimic Kk}}$, 
we know that $G(\hat{\by},b)$ is the union of some $(k-1)$-graphs in $\{G^{(k-1)}_i\cap \cK_{k-1}(\hat{R}^{(k-2)}(\hat{\by})) : i\in [s]\}$. 
Thus (R3)$_{\ref{lem: mimic Kk}}$ and the union lemma~(Lemma~\ref{lem: union regularity}) with the fact\COMMENT{$\epsilon'(\ba)\ll 1/o$ and $s\leq o^{4^k}$.} that $\epsilon'(\ba^{\sR})\ll 1/s$
imply that $G(\hat{\by},b)$ is $\epsilon'(\ba^{\sR})^{2/3}$-regular with respect to $\hat{R}^{(k-2)}(\hat{\by})$. 
As $G(\hat{\by},b)\subseteq P^{(k-1)}_*(\hat{\by},b)$ and \eqref{eq: P/R regular} holds, 
there exists a number $d(\hat{\by},b)\in [0,1/a_{k-1}]$ such that\COMMENT{need that $\epsilon^{2/3} \leq \epsilon'(\ba^{\sR})^{1/2}$ here.}
\begin{equation}\label{eq: Gyb reg}
\begin{minipage}[c]{0.8\textwidth}\em
$G(\hat{\by},b)$ is $(\epsilon'(\ba^{\sR})^{1/2},d(\hat{\by},b))$-regular with respect to $\hat{R}^{(k-2)}(\hat{\by})$.
\end{minipage}
\end{equation}
Now we use the values $d(\hat{\by},b)$ to construct $F^{(k-1)}$.
We apply the slicing lemma (Lemma~\ref{lem: slicing}) with the following objects and parameters.\newline

{\small
\begin{tabular}{c|c|c|c|c|c}
object/parameter & $Q^{(k-1)}_*(\hat{\by},b)$ & $\hat{L}^{(k-2)}(\hat{\by})$ & $1/a_{k-1}$ & $ \max\{ d(\hat{\by},b)a_{k-1}, 1-d(\hat{\by},b)a_{k-1} \}$ & $1 $ \\ \hline
playing the role of & $H^{(k)}$& $ H^{(k-1)}$ & $d$ &  $p_1$ & $s$
 \\ 
\end{tabular}
}\newline \vspace{0.2cm}

\noindent
By~\eqref{eq: Q/L regular} we obtain a partition of $Q^{(k-1)}_*(\hat{\by},b)$ into two $(k-1)$-graphs such that
for one of these, say $F(\hat{\by},b)$, we have that
\begin{equation}\label{eq: Fyb reg}
\begin{minipage}[c]{0.8\textwidth}\em
$F(\hat{\by},b)$ is $(\epsilon'(\ba^{\sR})^{1/2},d(\hat{\by},b))$-regular with respect to $\hat{L}^{(k-2)}(\hat{\by})$ and $F(\hat{\by},b)\subseteq Q_*^{(k-1)}(\hat{\by},b)$.
\end{minipage}
\end{equation}
Let 
$$F^{(k-1)}:= \bigcup_{\hat{\by}\in \hat{A}(k-1,k-2,\ba^{\sR}), b\in [a_{k-1}]} F(\hat{\by},b).$$
Thus $F^{(k-1)}\sub \cK_{k-1}(\sQ^{(1)})$.

Now we have defined $F^{(k-1)}$ and $G^{(k-1)}$.
It only remains to show that these two $(k-1)$-graphs satisfy  (F2)$_{\ref{lem: mimic Kk}}$.
Fix any vector $\hat{\bz}\in \hat{A}(k,k-1,\ba)$.
Consider $\hat{\by}$ and $\hat{\bx}$ such that $\hat{\by}\in A(\hat{\bx}) \text{ and } \hat{\bx} \leq_{k-1,k-2} \hat{\bz}$. 
By \eqref{eq:hatPconsistsofP(x,b)} we have
\begin{align}\label{eq:unionP}
	\hat{P}^{(k-1)}(\hat{\bz}) = \bigcup_{\hat{\bw}\leq_{k-1,k-2} \hat{\bz}} P^{(k-1)}(\hat{\bw},\bz^{(k-1)}_{\bw^{(1)}_*}).
\end{align}
By \eqref{eq: A hat bx def}, $\hat{\by}\in A(\hat{\bx})$ implies that 
\begin{align*}
	&\cK_{k-1}(\hat{R}^{(k-2)}(\hat{\by})) \subseteq \cK_{k-1}(\hat{P}^{(k-2)}(\hat{\bx})) \enspace \text{ and }\\
	&\cK_{k-1}(\hat{R}^{(k-2)}(\hat{\by})) \cap \cK_{k-1}(\hat{P}^{(k-2)}(\hat{\bw}))=\emptyset \text{ for }\hat{\bw}\neq \hat{\bx}.
\end{align*}\COMMENT{This is used for showing the first equality below.}
Also $P^{(k-1)}(\hat{\bw},\bz^{(k-1)}_{\bw^{(1)}_*})\subseteq \cK_{k-1}(\hat{P}^{(k-2)}(\hat{\bw}))$ whenever $\hat{\bw} \leq_{k-1,k-2} \hat{\bz}$.
Together this implies
$$\hat{P}^{(k-1)}(\hat{\bz}) \cap \cK_{k-1}(\hat{R}^{(k-2)}(\hat{\by})) 
\stackrel{(\ref{eq:unionP})}{=} P^{(k-1)}(\hat{\bx},\bz^{(k-1)}_{\bx^{(1)}_*}) \cap \cK_{k-1}(\hat{R}^{(k-2)}(\hat{\by})) 
\stackrel{(\ref{eq: Q* P* def})}{=}  P_*^{(k-1)}(\hat{\by},\bz^{(k-1)}_{\bx^{(1)}_*}).$$
Thus
\begin{align}\label{eq: G hat y z sub x*}
G^{(k-1)}\cap \hat{P}^{(k-1)}(\hat{\bz}) \cap \cK_{k-1}(\hat{R}^{(k-2)}(\hat{\by}))=
G^{(k-1)} \cap  P_*^{(k-1)}(\hat{\by},\bz^{(k-1)}_{\bx^{(1)}_*})
 \stackrel{\eqref{eq: G hat by b def mimic}}{=} G(\hat{\by},\bz^{(k-1)}_{\bx^{(1)}_*}).
 \end{align}
Together with (R2)$_{\ref{lem: mimic Kk}}$ this implies that 
\begin{eqnarray*}
 G^{(k-1)} \cap \hat{P}^{(k-1)}(\hat{\bz}) &=&  \bigcup_{\hat{\bx}\leq_{k-1,k-2} \hat{\bz}\; } 
\bigcup_{ \hat{\by}\in A(\hat{\bx})} G(\hat{\by},\bz^{(k-1)}_{\bx^{(1)}_*}) .
\end{eqnarray*}
\COMMENT{Here, note that in the first union $\hat{\bx}\in \hat{A}(k-1,k-2,\ba)$, while $A(\hat{\bx})$ is defined for $\hat{\bx}\in \hat{A}(k-1,k-2,\ba')$. However, $\hat{A}(k-1,k-2,\ba') =\hat{A}(k-1,k-2,\ba)$. So the definition of $\leq_{k-1,k-2}$ works fine.}
Similarly 
\begin{eqnarray*}
 F^{(k-1)} \cap \hat{Q}^{(k-1)}(\hat{\bz}) &=&  \bigcup_{\hat{\bx}\leq_{k-1,k-2} \hat{\bz}\;} 
\bigcup_{ \hat{\by}\in A(\hat{\bx})} F(\hat{\by},\bz^{(k-1)}_{\bx^{(1)}_*}) .
\end{eqnarray*}
For each $\hat{\by}\in \hat{A}(k-1,k-2,\ba^{\sR})$
let 
$$d_{\ba^{\sR},\hat{\bz},k-1}(\hat{\by}):=\left\{\begin{array}{ll} d(\hat{\by},b) 
&\text{ if } \hat{\by}\in A(\hat{\bx}) \text{ for some } \hat{\bx} \leq_{k-1,k-2} \hat{\bz} \text{ and } b=\bz^{(k-1)}_{\bx^{(1)}_*}, \\
0 &\text{ otherwise.}
\end{array}\right.$$
The properties \eqref{eq: Gyb reg} and \eqref{eq: G hat y z sub x*} together imply that  for each $\hat{R}^{(k-2)}(\hat{\by})\in \hat{\sR}^{(k-2)}$, 
\begin{equation*}
\begin{minipage}[c]{0.8\textwidth}\em
$G^{(k-1)} \cap \hat{P}^{(k-1)}(\hat{\bz})$ is $(\epsilon'(\ba^{\sR})^{1/2}, d_{\ba^{\sR},\hat{\bz},k-1}(\hat{\by}))$-regular with respect to $\hat{R}^{(k-2)}(\hat{\by})$.
\end{minipage}
\end{equation*}
Analogously using \eqref{eq: Fyb reg},
we obtain that
for each $\hat{L}^{(k-2)}(\hat{\by})\in \hat{\sL}^{(k-2)}$,
\begin{equation*}
\begin{minipage}[c]{0.8\textwidth}\em
$F^{(k-1)} \cap \hat{Q}^{(k-1)}(\hat{\bz})$ is $(\epsilon'(\ba^{\sR})^{1/2}, d_{\ba^{\sR},\hat{\bz},k-1}(\hat{\by}))$-regular with respect to $\hat{L}^{(k-2)}(\hat{\by})$.
\end{minipage}
\end{equation*}
In other words, 
$\sR$ is an $(\epsilon'(\ba^{\sR})^{1/2},d_{\ba^{\sR},\hat{\bz},k-1})$-partition of $G^{(k-1)}\cap \hat{P}^{(k-1)}(\hat{\bz})$. 
From \eqref{eq: epsilon' def mimic} and \eqref{eq: a1 sr}, we know 
\begin{align*}
	\epsilon'(\ba^{\sR}) \ll 1/\|\ba^{\sR}\|_{\infty}\leq  1/a_1^{\sR} \ll \nu, 1/o, 1/k.
\end{align*}
In particular, this means that $G^{(k-1)}\cap \hat{P}^{(k-1)}(\hat{\bz})$ satisfies the following regularity instance.
 $$R:= (\epsilon'(\ba^{\sR})^{1/2}, \ba^{\sR}, d_{\ba^{\sR}, \hat{\bz},k-1}).$$
Similarly, $F^{(k-1)}\cap \hat{Q}^{(k-1)}(\hat{\bz})$ also satisfies the regularity instance $R$.
Thus we can apply Lemma~\ref{lem: counting hypergraph} twice with the following objects and parameters, once with $G^{(k-1)} \cap \hat{P}^{(k-1)}(\hat{\bz})$ playing the role of $H$ and once more with $F^{(k-1)} \cap \hat{Q}^{(k-1)}(\hat{\bz})$ playing the role of $H$.\newline

{\small
\begin{tabular}{c|c|c|c|c|c|c|c}
object/parameter & $\epsilon'(\ba^{\sR})^{1/2}$ & $\|a^{\sR}\|_{\infty}$ & $ a^{\sR}_1$ & $\nu^2 o^{-4^k}$ & $ \ba^{\sR}$ & $k-1$ & $K^{(k-1)}_k$  \\ \hline
playing the role of & $\epsilon$ & $t$ & $ a_1$ & $\gamma$ & $\ba$ & $k$ & $ F$ 
\end{tabular}
}\newline \vspace{0.2cm}

\noindent
Thus we obtain
\begin{align}\label{eq: GP IC sim}
\frac{|\cK_{k}( G^{(k-1)} \cap \hat{P}^{(k-1)}(\hat{\bz}))|}{\binom{n}{k}} 
&= IC(K^{(k-1)}_k,d_{\ba^{\sR}, \hat{\bz},k-1}) \pm \nu^2 o^{-4^{k}}\enspace\text{ and}\\
\label{eq: FQ IC sim}
\frac{|\cK_{k}( F^{(k-1)} \cap \hat{Q}^{(k-1)}(\hat{\bz}))|}{\binom{m}{k}}  
&= IC(K^{(k-1)}_k,d_{\ba^{\sR}, \hat{\bz},k-1}) \pm \nu^2 o^{-4^{k}}.
\end{align}
On the other hand, we can apply Lemma~\ref{lem: counting} to show that for every $\hat{\bz}\in \hat{A}(k,k-1,\ba)$
\begin{align*}
	|\cK_{k}(\hat{P}^{(k-1)}(\hat{\bz}))| 
	&= (1\pm \nu^2)\prod_{j=1}^{k-1}a_j^{-\binom{k}{j}}n^k
	\geq o^{-4^{k}} \binom{n}{k} \enspace\text{ and }\\
	|\cK_{k}(\hat{Q}^{(k-1)}(\hat{\bz}))| 
	&= (1\pm \nu^2 )\prod_{j=1}^{k-1}a_j^{-\binom{k}{j}}m^k
	\geq o^{-4^{k}}\binom{m}{k}.
\end{align*}
This together with \eqref{eq: GP IC sim} and \eqref{eq: FQ IC sim} implies that
$$ d( \cK_{k}(G^{(k-1)}) \mid \hat{P}^{(k-1)}(\hat{\bz})) = d( \cK_{k}(F^{(k-1)}) \mid \hat{Q}^{(k-1)}(\hat{\bz})) \pm \nu.$$
\end{proof}

Suppose we are given two families of partitions $\sP,\sO$ such that $\sP$ almost refines $\sO$ and such that $\sO$ is an equitable partition of some $k$-graph $H$.
Roughly speaking, the next lemma shows that there is a family of partitions $\sO'$ 
such that $\sP \prec \sO'$  and such that $\sO'$ is still an equitable partition of $H$ (with a somewhat larger regularity constant).

\begin{lemma}\label{lem: reconstruction}
Suppose $0< 1/m, 1/n \ll \epsilon \ll \nu \ \ll \epsilon_0\leq 1$ and $k\in \N\sm\{1\}$.
Suppose $V$ is a vertex set of size $n$.
Suppose $R=(\epsilon_0/2, \ba^{\sO}, d_{\ba^{\sO},k})$ is a regularity instance and
$\sO=\sO(k-1,\ba^{\sO})$ is an 
$(\epsilon_0,\ba^{\sO},d_{\ba^{\sO},k})$-equitable partition of a $k$-graph $H^{(k)}$ on $V$.
Suppose there exists an 
$(\eta,\epsilon,\ba^{\sP})$-equitable family of partitions  $\sP=\sP(k-1,\ba^{\sP})$ on $V$ such that $\sP \prec_{\nu} \sO$.
Then there exists a family of partitions $\sO'$ on $V$ such that 
\begin{itemize}
\item[{\rm (O$'$1)$_{\ref{lem: reconstruction}}$}] $\sP\prec \sO'$,
\item[{\rm (O$'$2)$_{\ref{lem: reconstruction}}$}] $\sO'$ is a $(1/a_1^{\sO}, \epsilon_0+\nu^{1/20},\ba^{\sO}, \nu^{1/20})$-equitable family of partitions and it is an
$(\epsilon_0+\nu^{1/20},d_{\ba^{\sO},k})$-partition of $H^{(k)}$,
\item[{\rm (O$'$3)$_{\ref{lem: reconstruction}}$}] for each $j\in [k-1]$ and $ (\hat{\bx},b) \in \hat{A}(j,j-1,\ba^{\sO})\times [a_j^{\sO}]$, we have $ |O'^{(j)}(\hat{\bx},b)\triangle O^{(j)}(\hat{\bx},b)| \leq \nu^{1/2} \binom{n}{j}.$
\end{itemize}
\end{lemma}
We construct $\sO'$ by induction on $j\in [k-1]$.
When constructing ${\sO'}^{(j-1)}$ a natural approach is as follows.
For a given class $O^{(j)}(\hat{\bx},b)$ of $\sO^{(j)}$ we can let $O'^{(j)}(\hat{\bx},b)$ consist e.g.~of all classes of $\sP^{(j)}$ which lie (mostly) in $O^{(j)}(\hat{\bx},b)$.
This is formalized by the function $f_j$ in~\eqref{eq: f similar}.
However, this construction may not fit with the existing polyads of $\hat{P}^{(j-1)}$ (i.e.~it may violate Definition~\ref{def: family of partitions}(ii)).
This issue requires some adjustments, whose overall effect can be shown to be negligible.

\begin{proof}[Proof of Lemma~\ref{lem: reconstruction}]
For any function $f:\hat{A}(j,j-1,\ba^{\sP})\times [a^{\sP}_j] \rightarrow \hat{A}(j,j-1,\ba^{\sO})\times [a^{\sO}_j]$, let
$$d(f):= \sum_{ (\hat{\bx},b) \in \hat{A}(j,j-1,\ba^{\sP})\times [a^{\sP}_j] } |P^{(j)}(\hat{\bx},b) \setminus O^{(j)}( f(\hat{\bx},b))|.$$
Note that $\sP \prec_{\nu} \sO$ implies that for each $j\in [k-1]$, 
there exists a function $f_{j}: \hat{A}(j,j-1,\ba^{\sP})\times [a^{\sP}_j] \rightarrow \hat{A}(j,j-1,\ba^{\sO})\times [a^{\sO}_j]$ such that 
\begin{align}\label{eq: f similar}
d(f_{j}) \leq \nu \binom{n}{j}.
\end{align}
Moreover, note that since $R$ is a regularity instance (see Definition~\ref{def: regularity instance}), we have $\epsilon_0\leq \norm{\ba^{\sO}}^{-4^k}\epsilon_{\ref{lem: counting}}(\|\ba^{\sO}\|^{-1}_\infty, \|\ba^{\sO}\|^{-1}_\infty,k-1,k)$. 
Thus Lemma~\ref{lem: counting} and the definition of an equitable family of partitions (see Definition~\ref{def: equitable family of partitions}) imply that for any $j \in [k-1]\sm\{1\}$ and $(\hat{\bx},b)\in \hat{A}(j,j-1,\ba^{\sO})\times[a_j^{\sO}]$, we have
\begin{align}\label{eq: O hypergraph not too small}
|O^{(j)}(\hat{\bx},b)| 
\geq \frac{1}{2\norm{\ba^{\sO}}^{2^k}}n^j
\geq \epsilon_0^{1/2} n^{j}.
\end{align}
For each $i\in [a_1^{\sO}]$,
let
$$O'^{(1)}(i,i) :=  \bigcup_{s\in [a^{\sP}_1], f_1( s )= i} P^{(1)}(s,s) \enspace \text{ and let } \enspace \sO'^{(1)}:= \{ O'^{(1)}(i,i)  : i\in [a^{\sO}_1]\}.$$

Note that \eqref{eq: f similar} implies that 
$|O'^{(1)}(i,i)| = (1\pm a^{\sO}_1\nu)n/a^{\sO}_1 = (1\pm \nu^{1/2}) n/a^{\sO}_1$\COMMENT{
Here, $\nu\ll \epsilon_0$. Since $R$ is a regularity instance, $\epsilon_0 \ll \norm{\ba^{\sO}}^{-1}$. Thus $\nu a_1^{\sO} \leq \nu^{1/2}$.}.
For all distinct $i,i'\in [a_1^\sO]$,
let ${O'}^{(1)}(i,i'):=\es$.
Hence $\sO'^{(1)}$ satisfies properties (O$'$1)$_1$--(O$'$4)$_1$ below. (Here, (O$'$2)$_1$ and (O$'$4)$_1$ are vacuous.)
Assume for some $j\in [k-1]\setminus\{1\}$ we have defined $\sO'^{(1)},\dots, \sO'^{(j-1)}$ satisfying the following for each $i\in[j-1]$:
\begin{itemize}
\item[(O$'$1)$_i$] $\sO'^{(i)}$ forms a partition of $\cK_{i}(\sO'^{(1)})$ and $\sP^{(i)}\prec \sO'^{(i)}$,
\item[(O$'$2)$_i$] if $i>1$, then 
$\sO'^{(i)}=\{O'^{(i)}(\hat{\bx},b) : (\hat{\bx},b) \in \hat{A}(i,i-1,\ba^{\sO})\times [a_i^{\sO}]\},$
\item[(O$'$3)$_i$] 
$$\sum_{ (\hat{\bx},b) \in \hat{A}(i,i-1,\ba^{\sO})\times [a_i^{\sO}] } |O'^{(i)}(\hat{\bx},b)\setminus O^{(i)}(\hat{\bx},b)| \leq i \nu n^{i},$$
\item[(O$'$4)$_i$] if $i>1$, then for each $\hat{\bx}\in \hat{A}(i,i-1,\ba^{\sO})$, the collection $\{O'^{(i)}(\hat{\bx},1),\dots, O'^{(i)}(\hat{\bx},a^{\sO}_{i})\}$ forms a partition of $\cK_{i}(\hat{O'}^{(i-1)}(\hat{\bx}))$, where 
$$\hat{O'}^{(i-1)}(\hat{\bx}):= \bigcup_{\hat{\by}\leq_{i-1,i-2} \hat{\bx}} O'^{(i-1)}(\hat{\by}, \bx^{(i-1)}_{\by^{(1)}_*}).$$
\end{itemize}\vspace{0.3cm}

We will now construct ${\sO'}^{(j)}$ satisfying (O$'$1)$_j$--(O$'$4)$_j$.
So assume that $k\geq 3$.
Note that (O$'$3)$_1$--(O$'$3)$_{j-1}$ with \eqref{eq: O hypergraph not too small} shows that  for any $i \in [j-1]$ and $(\hat{\bx},b)\in \hat{A}(i,i-1,\ba^{\sO})\times[a^{\sO}_i]$, the $i$-graph $O'^{(i)}(\hat{\bx},b)$ is nonempty.
Together with (O$'$1)$_1$--(O$'$1)$_{j-1}$, (O$'$2)$_1$--(O$'$2)$_{j-1}$, (O$'$3)$_1$--(O$'$3)$_{j-1}$, (O$'$4)$_1$--(O$'$4)$_{j-1}$,
\eqref{eq: O hypergraph not too small}, and Lemma~\ref{lem: family of partitions construction} this implies that $\{\sO'^{(i)}\}_{i=1}^{j-1}$ forms a family of partitions.\COMMENT{
Actually we are using Lemma~\ref{lem: family of partitions construction} not only to show it is a family of partitions but that the classes and polyads of $\sO'$ behave in the same way as the maps in \eqref{eq: hat P def}--\eqref{eq: hat P P relations emptyset}. This allows us to apply Proposition~\ref{prop: hat relation}.
Actually, in order to be able to apply Lemma~\ref{lem: family of partitions construction},
we not only need (O$'$2)$_i$,
but we also need that $|{O'}^{(i)}|=a_i^{\sO}$ (i.e.~that the ${O'}^{(i)}(\hat{\ba},b)$ are all distinct).
This follows from (O$'$3)$_i$ and \eqref{eq: O hypergraph not too small} since $\nu \ll \epsilon_0$.
}
Let $\hat{\sO'}^{(j-1)}$ be the collection of all the $\hat{O'}^{(j-1)}(\hat{\bx})$ with $\hat{\bx}\in \hat{A}(j,j-1,\ba^{\sO})$.
Note that Proposition~\ref{prop: hat relation}(iv)~and~(vi) implies that
\begin{equation}\label{eq: cK hat P j-1 forms a partition of}
\begin{minipage}[c]{0.8\textwidth}\em
$\{\cK_{j}(\hat{O}'^{(j-1)}(\hat{\bx})): \hat{\bx}\in \hat{A}(j,j-1,\ba^{\sO})\}$ forms a partition of $\cK_{j}(\sO'^{(1)})$.
\end{minipage}
\end{equation}
By Proposition~\ref{prop: hat relation}(xi)\COMMENT{
Can apply this with $\{{\sO'}^{(i)}\}^{j-1}_i$ playing the role of $\sP$.
}
\begin{align}\label{eq: hat P j-1 prec hat S' j-1}
	\{\cK_{j}(\hat{P}^{(j-1)}(\hat{\bx})): \hat{\bx}\in \hat{A}(j,j-1,\ba^{\sP})\}
	\prec \{\cK_{j}(\hat{O}'^{(j-1)}(\hat{\bx})): \hat{\bx}\in \hat{A}(j,j-1,\ba^{\sO})\}.
\end{align}
Let 
$$A:= \{ \hat{\by}\in \hat{A}(j,j-1,\ba^{\sP}) : \cK_{j}(\hat{P}^{(j-1)}(\hat{\by}))\subseteq \cK_{j}(\sO'^{(1)})\}.$$
Then~\eqref{eq: hat P j-1 prec hat S' j-1} implies that
 \begin{align}\label{eq: A partitions cKj sO'}
\bigcup_{\hat{\by}\in A} \cK_{j}(\hat{P}^{(j-1)}(\hat{\by})) = \cK_{j}(\sO'^{(1)}).
\end{align}
By~\eqref{eq: cK hat P j-1 forms a partition of} and~\eqref{eq: hat P j-1 prec hat S' j-1}, 
for each $\hat{\by} \in A$,
 there exists $g(\hat{\by})\in \hat{A}(j,j-1,\ba^{\sO})$ such that $\cK_{j}(\hat{P}^{(j-1)}(\hat{\by}))$ is a subset of $\cK_{j}(\hat{O'}^{(j-1)}(g(\hat{\by}))).$ 

\setcounter{claim}{0}
\begin{claim}\label{cl:smalldiff}
\begin{align*}
\sum_{\hat{\by}\in A} |\cK_{j}(\hat{P}^{(j-1)}(\hat{\by}) )\setminus \cK_{j}(\hat{O}^{(j-1)}(g(\hat{\by})))|\leq (j-1)\nu n^{j}.
\end{align*}
\end{claim}
\begin{proof}
Observe that by \eqref{eq: hat P j-1 prec hat S' j-1} for each $\hat{\bx}\in  \hat{A}(j,j-1,\ba^{\sO})$,
we have
$$\bigcup_{\hat{\by}\colon g(\hat{\by})=\hat{\bx}} \cK_{j}(\hat{P}^{(j-1)}(\hat{\by}))
=\cK_{j}(\hat{O'}^{(j-1)}(\hat{\bx})).$$ 
This implies that
$$ \sum_{\hat{\by}\in A} |\cK_{j}(\hat{P}^{(j-1)}(\hat{\by}) )\setminus \cK_{j}(\hat{O}^{(j-1)}(g(\hat{\by})))| =
\sum_{\hat{\bx}\in \hat{A}(j,j-1,\ba^{\sO})} |\cK_{j}(\hat{O'}^{(j-1)}(\hat{\bx}))\setminus \cK_{j}(\hat{O}^{(j-1)}(\hat{\bx}))|.$$
Any $j$-set counted on the right hand side lies in $\cK_{j}(\sO'^{(1)})$ and 
contains
a $(j-1)$-set $J\in  O'^{(j-1)}(\hat{\bz},b) \setminus O^{(j-1)}(\hat{\bz},b)$ for some $\hat{\bz}\leq_{j-1,j-2} \hat{\bx}$ and $b=\bx^{(j-1)}_{\bz^{(1)}_*}$. 
Note that (O$'$3)$_{j-1}$ implies that there are at most $(j-1)\nu n^{j-1}$ such sets $J$. 
For such a fixed $(j-1)$-set $J$, 
there are at most $n$ $j$-sets in $\cK_{j}(\sO'^{(1)})$ containing $J$. 
Thus at most $(j-1)\nu n^{j-1}\cdot n = (j-1)\nu n^{j}$ $j$-sets are counted in the above expression. 
This proves the claim.
\end{proof}

Ideally,
for every $\bx\in \hat{A}(j,j-1,\ba^{\sO})$ and $b\in [a_j^{\sO}]$,
we would like to define $O'^{(j)}(\bx,b)$
as the union of all $P^{(j)}(\hat{\by},b')$ for which  $f_{j}(\hat{\by},b') = (\hat{\bx},b)$ holds.
However, we may have $f_j(\hat{\by},b')\neq (g(\hat{\by}),b) $ for all $b\in [a_j^{\sO}]$.
This leads to difficulties when attempting to prove (O$'$4)$_j$.
We resolve this problem by defining a function $f_j'$, which is a slight modification of $f_j$.
To this end, let $$W:= \{ (\hat{\by},b') \in A\times [a^{\sP}_j] :f_j(\hat{\by},b')\neq (g(\hat{\by}),b) \text{ for all }b\in [a_j^{\sO}]\}.$$
Thus
if $(\hat{\by},b')\in W$, then $O^{(j)}(f_j(\hat{\by},b')) \cap \cK_{j}(\hat{O}^{(j-1)}(g(\hat{\by}))) =\emptyset$.%
\COMMENT{Because it is a partition refining another. Thus either it belongs, or doesn't intersect. there's no case of `intersecting some but not contained'} 
This and the fact that $P^{(j)}(\hat{\by},b') \subseteq  \cK_{j}(\hat{P}^{(j-1)}(\hat{\by}))$ imply that for $(\hat{\by},b')\in W$
\begin{align}\label{eq: cap belong}
P^{(j)}(\hat{\by},b') \cap O^{(j)}(f_{j}(\hat{\by},b')) \subseteq \cK_{j}(\hat{P}^{(j-1)}(\hat{\by}))\setminus \cK_{j}(\hat{O}^{(j-1)}(g(\hat{\by}))).
\end{align}

We define a function $f'_{j}:A\times [a^{\sP}_j]\to \hat{A}(j,j-1,\ba^{\sO})\times [a^{\sO}_j]$ by 
$$f'_{j}(\hat{\by},b') := \left\{\begin{array}{ll}
(g(\hat{\by}),b) \text{ for an arbitrary } b\in [a^{\sO}_j] &\text{ if } (\hat{\by},b')\in W,\\
f_{j}(\hat{\by},b') &\text{ otherwise.}
\end{array} \right.$$
For each $\hat{\bx}\in \hat{A}(j,j-1,\ba^{\sO})$ and $b\in[a^{\sO}_j]$, let
\begin{align}\label{eq: O' x b def}
O'^{(j)}(\hat{\bx},b) 
:= \bigcup_{ \substack{(\hat{\by},b')\in A\times [a_j^\sP]\colon\\ f'_{j}(\hat{\by},b') = (\hat{\bx},b)}} P^{(j)}(\hat{\by},b').
\end{align}

Let $\sO'^{(j)}$ be as described in (O$'$2)$_j$.
By \eqref{eq: cK hat P j-1 forms a partition of}, \eqref{eq: hat P j-1 prec hat S' j-1}, 
and the fact that $f'_{j}$ is defined for all $A\times [a^{\sP}_j]$, 
we obtain (O$'$1)$_j$.

We now verify (O$'$3)$_j$.
For this, we estimate $d(f_j')$, namely
\begin{eqnarray}\label{eq: efj ef'j}
d(f'_j) &=& \sum_{ (\hat{\by},b') \in \hat{A}(j,j-1,\ba^{\sP})\times [a^{\sP}_j] } \hspace{-0.5cm} |P^{(j)}(\hat{\by},b') \setminus O^{(j)}( f'_j(\hat{\by},b'))| \nonumber \\
&\leq&  \sum_{ (\hat{\by},b') \in \hat{A}(j,j-1,\ba^{\sP})\times [a^{\sP}_j]}   \hspace{-0.5cm}  |P^{(j)}(\hat{\by},b') \setminus O^{(j)}( f_j(\hat{\by},b'))|  \notag\\
&& \qquad+ \sum_{ (\hat{\by},b')\in W} |P^{(j)}(\hat{\by},b') \cap O^{(j)}(f_{j}(\hat{\by},b'))| \nonumber \\
&\stackrel{\eqref{eq: cap belong}}{\leq }& d(f_j) + \sum_{\hat{\by}\in A} |\cK_{j}(\hat{P}^{(j-1)}(\hat{\by}))\setminus \cK_{j}(\hat{O}^{(j-1)}(g(\hat{\by})))|\nonumber \\
&\stackrel{\rm Claim~\ref{cl:smalldiff}}{\leq}& d(f_j)+ (j-1)\nu n^{j} 
\stackrel{\eqref{eq: f similar}}{\leq} j \nu n^{j}.
\end{eqnarray}
This in turn implies that
\begin{eqnarray*}
& & \hspace{-2.4cm} \sum_{  (\hat{\bx},b) \in \hat{A}(j,j-1,\ba^{\sO}) \times [a_j^{\sO}]}  |O'^{(j)}(\hat{\bx},b)\setminus O^{(j)}(\hat{\bx},b)| \\ 
&\stackrel{(\ref{eq: O' x b def})}{=}& \sum_{ (\hat{\bx},b) \in \hat{A}(j,j-1,\ba^{\sO}) \times [a_j^{\sO}]}  \sum_{ \substack{(\hat{\by},b') \in A\times [a^{\sP}_j]\colon \\f'_{j}(\hat{\by},b')= (\hat{\bx},b)}}   |P^{(j)}(\hat{\by},b')\setminus O^{(j)}( \hat{\bx},b )|
\\
&=&  \sum_{(\hat{\by},b')\in A \times [a^{\sP}_j] }   |P^{(j)}(\hat{\by},b')\setminus O^{(j)}(f'_j(\hat{\by},b'))| \leq d(f'_{j}) \stackrel{\eqref{eq: efj ef'j}}{\leq} j\nu n^{j}. 
\end{eqnarray*}
Thus (O$'$3)$_j$ holds. 

Suppose $\hat{\bx}\in \hat{A}(j,j-1,\ba^{\sO})$ and $b\in [a_j^{\sO}]$.
Note that $P^{(j)}(\hat{\by},b') \subseteq \cK_{j}(\hat{P}^{(j-1)}(\hat{\by})) \subseteq \cK_j(\hat{O'}^{(j-1)}(g(\hat{\by})))$ for each $\hat{\by}\in A$ and $b'\in [a_j^{\sP}]$.
Together with~\eqref{eq: O' x b def} and the definition of $f_j'$,
we obtain that
$O'^{(j)}(\hat{\bx},b) \subseteq \cK_j(\hat{O'}^{(j-1)}(\hat{\bx})).$ 
By this and \eqref{eq: cK hat P j-1 forms a partition of}--\eqref{eq: A partitions cKj sO'}, the collection $\{O'^{(j)}(\hat{\bx},1),\dots, O'^{(j)}(\hat{\bx},a^{\sO}_{i})\}$ forms a partition of $\cK_{j}(\hat{O'}^{(j-1)}(\hat{\bx}))$. 
Thus (O$'$4)$_j$ holds.

\medskip

By repeating this procedure, 
we obtain $\sO'^{(1)},\dots ,\sO'^{(k-1)}$.
Let $\sO':=\{\sO'^{(j)}\}_{j=1}^{k-1}$.
As observed before~\eqref{eq: cK hat P j-1 forms a partition of}, $\sO'$ is a family of partitions.
Properties (O$'$1)$_1$--(O$'$1)$_{k-1}$ and imply (O$'$1)$_{\ref{lem: reconstruction}}$.

Note that (O$'$3)$_1$ implies that for each $j\in [k-1]$ we have
$|\cK_{j}(\sO^{(1)}) \triangle \cK_{j}({\sO'}^{(1)})| \leq 2 \nu n^{j}$.
Thus for each $j\in [k-1]$ and $ (\hat{\bx},b) \in \hat{A}(j,j-1,\ba^{\sO})\times [a_j^{\sO}]$, 
 this implies that
\begin{eqnarray*}
	|O'^{(j)}(\hat{\bx},b)\triangle O^{(j)}(\hat{\bx},b)|  \hspace{-0.2cm}
	&\leq&  \hspace{-0.2cm} |\cK_{j}(\sO^{(1)}) \triangle \cK_{j}({\sO'}^{(1)})|+  \hspace{-0.4cm}\sum_{\hat{\bx} \in \hat{A}(j,j-1,\ba) , b\in [a_j^{\sO}] } \hspace{-0.4cm} |O'^{(j)}(\hat{\bx},b)\setminus O^{(j)}(\hat{\bx},b)|  \\
	 \hspace{-0.2cm}&\stackrel{   ({\rm O}'3)_j }{\leq }  \hspace{-0.2cm}&(j+2)\nu n^{j} \leq \nu^{1/2} \binom{n}{j}.
\end{eqnarray*}
Thus we have (O$'$3)$_{\ref{lem: reconstruction}}$. Finally, since $R$ is a regularity instance, (O$'$3)$_{\ref{lem: reconstruction}}$ enables us to apply Lemma~\ref{lem: slightly different partition regularity} with the following objects and parameters.\newline

{\small
\begin{tabular}{c|c|c|c|c|c|c|c|c}
object/parameter & $\sO$ & $\sO'$ & $\nu^{1/2}$& $ 0 $ & $\epsilon_0$ & $d_{\ba^{\sO},k}$ & $H^{(k)}$ & $ H^{(k)}$  \\ \hline
playing the role of & $\sP$ & $\sQ$ & $\nu$& $\lambda$ & $\epsilon$ & $d_{\ba,k}$ & $H^{(k)}$ & $G^{(k)}$ 
 \\ 
\end{tabular}
}\newline \vspace{0.2cm}

\noindent
\COMMENT{Use (O3)$_{1}$ for $\lambda$.}
This implies (O$'$2)$_{\ref{lem: reconstruction}}$.
\end{proof}

\section{Sampling a regular partition}
\label{sec: Sampling a regular partition}

In this section we prove Theorem~\ref{lem: random choice2}. In Section~\ref{Building a family of partitions from three others} we provide the main tool (Lemma~\ref{lem: similar}) for this result
and in Section~\ref{Random samples} we deduce Theorem~\ref{lem: random choice2}.

\subsection{Building a family of partitions from three others}\label{Building a family of partitions from three others}
In this subsection we prove our key tool (Lemma~\ref{lem: similar}) for the proof of Theorem~\ref{lem: random choice2}.
Roughly speaking Lemma~\ref{lem: similar} states the following.
Suppose there are two $k$-graphs $H_1,H_2$ with vertex set $V_1,V_2$, respectively, and there are two  $\epsilon$-equitable families of partitions of these $k$-graphs which 
have the same parameters.
Suppose further that there is another  $\epsilon_0$-equitable family of partitions $\sO_1$ for $H_1$.
Then there is an equitable family of partitions $\sO_2$ of $H_2$ which has the (roughly) same parameters as $\sO_1$  provided $\epsilon\ll \epsilon_0$.
Even more loosely,
the result says that if two hypergraphs share a single `high-quality' regularity partition, then they share any `low-quality' regularity partition.

\begin{lemma}\label{lem: similar}
Suppose $0< 1/n, 1/m \ll \epsilon \ll  1/T, 1/a^{\sQ}_1 \ll \delta \ll \epsilon_0\leq 1$ and $k\in \N\sm\{1\}$.
Suppose $\ba^{\sQ}\in [T]^{k-1}$. 
Suppose that $R=(\epsilon_0/2, \ba^{\sO}, d_{\ba^{\sO},k})$ is a regularity instance.
Suppose $V_1,V_2$ are sets of size $n,m$, and $H_1^\kk, H_2^\kk$ are $k$-graphs on $V_1,V_2$, respectively.
Suppose
\begin{itemize}
\item[{\rm (P1)$_{\ref{lem: similar}}$}] $\sQ_1=\sQ_1(k-1,\ba^{\sQ})$ is an $(\epsilon,\ba^{\sQ},d_{\ba^{\sQ},k})$-equitable partition of $H_1^{(k)}$, 
\item[{\rm (P2)$_{\ref{lem: similar}}$}] $\sQ_2=\sQ_2(k-1,\ba^{\sQ})$ is an $(\epsilon,\ba^{\sQ},d_{\ba^{\sQ},k})$-equitable partition of $H_2^{(k)}$, and
\item[{\rm (P3)$_{\ref{lem: similar}}$}] $\sO_1=\sO_1(k-1,\ba^{\sO})$ is an $(\epsilon_0,\ba^{\sO},d_{\ba^{\sO},k})$-equitable partition of $H_1^{(k)}$.
\end{itemize}
Then there exists an $(\epsilon_0+\delta,\ba^{\sO},d_{\ba^{\sO},k})$-equitable partition $\sO_2$ of $H_2^{(k)}$. 
\end{lemma}

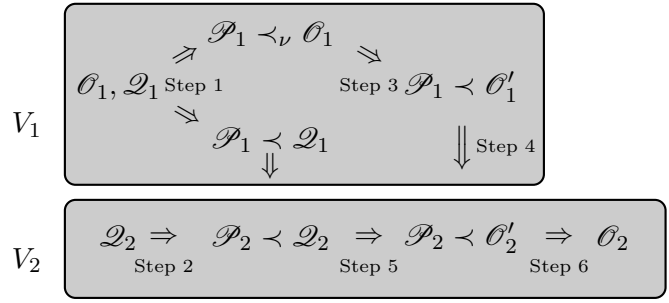
\begin{figure}[t]
\centering
\begin{tikzpicture}

\draw[rounded corners,thick,fill=gray!40] (1.3,-0.8) rectangle (9.2,0.5);
\draw[rounded corners,thick,fill=gray!40] (1.3,0.7) rectangle (7.6,3.1);

\node (v1) at (0.8,1.5) {$V_1$};
\node (v2) at (0.8,-0.3) {$V_2$};

\node (o1q1) at (2,2) {$\sO_1,\sQ_1$};
\node (q2) at (2,0) {$\sQ_2$};

\node (p2q2) at (4,0) {$\sP_2\prec\sQ_2$};

\node (p2o2') at (6.5,0) {$\sP_2\prec\sO_2'$};

\node (o2) at (8.5,0) {$\sO_2$};

\node (p1q1) at (4,1.3) {$\sP_1\prec\sQ_1$};
\node (p1o1) at (4,2.7) {$\sP_1\prec_\nu\sO_1$};

\node (p1o1') at (6.5,2) {$\sP_1\prec\sO_1'$};

\node (a1) at (2.6,0) {$\Rightarrow$};
\node (a2) at (5.3,0) {$\Rightarrow$};
\node (a3) at (7.8,0) {$\Rightarrow$};
\node[rotate=-30] (a4) at (5.3,2.4) {$\Rightarrow$};
\node (a5) at (4,0.95) {$\Downarrow$};
\node[rotate=-90] (a6) at (6.5,1.2) {$\Longrightarrow$};
\node[rotate=-30] (a7) at (2.9,1.6) {$\Rightarrow$};
\node[rotate=30] (a8) at (2.9,2.4) {$\Rightarrow$};

\node (s1) at (3,2) {{\tiny Step 1}};
\node (s2) at (2.6,-0.4) {{\tiny Step 2}};
\node (s3) at (5.3,2) {{\tiny Step 3}};
\node (s4) at (7.1,1.2) {{\tiny Step 4}};
\node (s5) at (5.3,-0.4) {{\tiny Step 5}};
\node (s6) at (7.8,-0.4) {{\tiny Step 6}};

\end{tikzpicture}
\caption{The proof strategy for Lemma~\ref{lem: similar}.}\label{fig:steps}
\end{figure}

A crucial point here is that the construction of $\sO_2$ incurs only an additive increase (by~$\delta$) of the regularity parameter of $\sO_1$.

For an illustration of the proof strategy of Lemma~\ref{lem: similar} see Figure~\ref{fig:steps}.
Our strategy is first to apply Lemma~\ref{lem:refreg} to $\sQ_1,\sO_1$
to obtain a family of partitions $\sP_1$ that refines $\sQ_1$ and almost refines $\sO_1$ (see Step~\ref{step1}).
Moreover, we refine $\sQ_2$ and obtain $\sP_2$ in such a way that $\sP_2$ has the same number of partition classes as $\sP_1$ (see Step~\ref{step2}).
We then apply Lemma~\ref{lem: reconstruction} to construct a family $\sO'_1$ of partitions that is very similar to $\sO_1$ and satisfies $\sP_1\prec \sO_1'$ (see Step~\ref{step3}).
Then we analyse how $\sP_1$ refines $\sO_1'$ (see Step~\ref{step4}). 
We then use Lemma~\ref{lem: mimic Kk} to mimic this structure in order to build $\sO_2'$ from $\sP_2$ (see Step~\ref{step5}). 
Finally in Step~\ref{step6} we apply Lemma~\ref{lem: removing lambda} to show that $\sO_2'$ can be slightly modified to obtain the desired $\sO_2$.

\begin{proof}[Proof of Lemma~\ref{lem: similar}]

We start with several definitions. Choose a new constant $\nu$ such that $1/T, 1/a_1^{\sQ} \ll \nu \ll \delta$.
Let $\overline{\epsilon}:\mathbb{N}^{k-1}\rightarrow (0,1]$ be a function such that  for any $\ba=(a_1,\dots, a_{k-1})\in \N^{k-1}$, we have
$0<\overline{\epsilon}(\ba) \ll \nu, \norm{\ba}^{-k}.$
Now given $\overline{\epsilon}$, we define
$$t':=t_{\ref{lem:refreg} }(k,T,T^{4^k},1/a^{\sQ}_1,\nu,\overline{\epsilon}).$$
\COMMENT{The definition of regularity instance implies that $\epsilon_0 \ll \norm{\ba^{\sO}}^{-4^k} \leq 2^{-4^k}$ and thus $\epsilon \ll 1/k$, i.e. we don't have to include $1/k$ in the hierarchy.}
Observe that $0<\epsilon \ll 1/k,1/T,1/a^{\sQ}_1,\nu, 1/t'$.
Thus we may assume that for any $\ba\in [t']^{k-1}$, we have
$$0<\epsilon \ll \overline{\epsilon}(\ba),\mu_{\ref{lem:refreg}}(k,T,T^{4^k},1/a^{\sQ}_1,\nu,\overline{\epsilon}).$$
\begin{step}\label{step1}
Constructing $\sP_1$ as a refinement of $\sQ_1$.
\end{step}
Let 
\begin{align*}
\sQ_1^{(k)}&:= \{ \cK_{k}(\hat{Q}_1^{(k-1)}): \hat{Q}_1^{(k-1)} \in \hat{\sQ}_1^{(k-1)} \}\enspace\text{ and}\\
	{\sQ_1'}^{(k)}&:=\left(\sQ_1^{(k)} \cup \left\{ \binom{V_1}{k}\sm \cK_{k}(\sQ_1^{(1)})\right\}\right) \setminus\{\emptyset\}.
\end{align*}

Since $\sQ_1$ is $T$-bounded, 
$|\sQ_1^{(k)}|\leq T^{2^{k}} $ by Proposition~\ref{prop: hat relation}(viii). 
Thus $|{\sQ_1'}^{(k)}|\leq T^{4^k}$. Moreover, the fact that $R$ is a regularity instance (and the definition of $ \epsilon_{\ref{def: regularity instance}}$) implies that $\ba^{\sO} \in [T]^{k-1}$.
We can apply Lemma~\ref{lem:refreg} with the following objects and parameters.\newline

{\small
\begin{tabular}{c|c|c|c|c|c|c|c|c|c}
object/parameter & $V_1$ & $\sO_1$ & ${\sQ_1'}^{(k)} $  & $\{\sQ_1^{(i)}\}_{i=1}^{k}$ &
$T$ & $T^{4^k}$ & $1/a^{\sQ}_1$ & $\nu$ & $\overline{\epsilon}$ \\ \hline
playing the role of & $V$ & $\sO$ & $\sH^{(k)}$ & $\sQ$ & $o$ & $s$ & $\eta$ & $\nu$ & $\epsilon$
 \\ 
\end{tabular}
}\newline \vspace{0.2cm}

\noindent
Observe that $\sO_1$, $\{\sQ_1^{(i)}\}_{i=1}^{k}$ and ${\sQ_1'}^{(k)}$ playing the roles of $\sO$, $\sQ$ and $\sH^{(k)}$, respectively, 
satisfy (O1)$_{\ref{lem:refreg}}$--(O4)$_{\ref{lem:refreg}}$ in Lemma~\ref{lem:refreg}. 
We obtain a family of partitions $\sP_1=\sP_1(k-1,\ba^{\sP})$ such that the following hold.

\begin{enumerate}[label=(P1\arabic*)]
\item\label{item:P11} $\sP_1$ is $(1/a^{\sQ}_1,\overline{\epsilon}(\ba^{\sP}), \ba^{\sP})$-equitable and $t'$-bounded, and $a^{\sQ}_j$ divides $a^{\sP}_j$ for each $j\in [k-1]$.
\item\label{item:P12} $\sP_1^{(j)}\prec \sQ_1^{(j)}$ and $\sP_1^{(j)} \prec_{\nu} \sO_1^{(j)}$ for each $j\in [k-1]$.
\end{enumerate}
Let 
$$\epsilon' := \overline{\epsilon}(\ba^{\sP}), 
\enspace \text{ and } \enspace a_{k}^{\sP}=a_{k}^{\sQ}=a_{k}^{\sO}:=1.$$%
\COMMENT{We do not update $\ba^{\sP}$.}
Hence $\epsilon' \ll \nu, \norm{\ba^{\sP}}^{-k}$ by the definition of $\overline{\epsilon}$.

By \ref{item:P12}, for each $j\in [k-1]$, $\hat{\by}\in \hat{A}(j,j-1,\ba^{\sP})$, and $b'\in [a^{\sP}_j]$, 
either there exists $\hat{\bx}\in \hat{A}(j,j-1,\ba^{\sQ})$ and $b\in [a^{\sQ}_j]$ such that $P_1^{(j)}(\hat{\by},b') \subseteq Q_1^{(j)}(\hat{\bx},b)$
or $P_1^{(j)}(\hat{\by},b') \cap Q_1^{(j)}(\hat{\bx},b)=\es$ for all $\hat{\bx}\in \hat{A}(j,j-1,\ba^{\sQ})$ and $b\in [a^{\sQ}_j]$.
This allows us to describe $\sP_1$ in terms of $\sQ_1$ in the following way.
For each $j\in [k-1]$, $\hat{\bx}\in \hat{A}(j,j-1,\ba^{\sQ})$, and $b\in [a^{\sQ}_j]$, we define
\begin{align}\label{eq: def Aj hat bx b and Aj}
A_j(\hat{\bx},b)&:= \{ (\hat{\by},b')\in \hat{A}(j,j-1,\ba^{\sP}) \times [a^{\sP}_j]: P_1^{(j)}(\hat{\by},b') \subseteq Q_1^{(j)}(\hat{\bx},b)\}\text{ and} \nonumber \\ 
A_j&:=\bigcup_{\hat{\bx}\in \hat{A}(j,j-1,\ba^{\sQ}), b\in [a_j^{\sQ}]} A_j(\hat{\bx},b).
\end{align}
For each $\hat{\bx}\in \hat{A}(k,k-1,\ba^{\sQ})$, let
\begin{align*}
	\hat{A}_k(\hat{\bx}) &:= \{ \hat{\by} \in \hat{A}(k,k-1,\ba^{\sP}) : \hat{P}^{(k-1)}_1(\hat{\by})\subseteq \hat{Q}^{(k-1)}_1(\hat{\bx})\}, \text{ and }\\ 
	\hat{A}_k&:= \bigcup_{\hat{\bx}\in \hat{A}(k,k-1,\ba^{\sQ})} \hat{A}_k(\hat{\bx}).
\end{align*}

The density function $d_{\ba^{\sQ},k}$ for $\sQ_1$ naturally gives rise to a density function for $\sP_1$.
Indeed, for each $\hat{\by}\in A(k,k-1,\ba^{\sP})$, we define
$$d_{\ba^{\sP},k}(\hat{\by})
:=\left\{\begin{array}{ll} 
d_{\ba^{\sQ},k}(\hat{\bx}) &\text{ if } \hat{\by}\in \hat{A}_k(\hat{\bx}) \text{ for some $\hat{\bx}\in \hat{A}(k,k-1,\ba^{\sQ})$ and}\\
0 &\text{ if } \hat{\by}\notin \hat{A}_k. \end{array}\right.$$

Recall that $\sP_1$ is a $(1/a^{\sQ}_1,\epsilon',\ba^{\sP})$-equitable family of partitions, 
$\sQ_1$ is a $(1/a^{\sQ}_1,\epsilon,\ba^{\sQ})$-equitable family of partitions,
and $\epsilon \ll \epsilon' \ll \nu, \norm{\ba^{\sP}}^{-k}$.
Thus Lemma~\ref{lem: counting}
implies that for each $j\in [k-1]$, $\hat{\by}\in \hat{A}(j+1,j,\ba^{\sP})$, and $\hat{\bx}\in \hat{A}(j+1,j,\ba^{\sQ})$,
we have
\begin{align}\label{eq: 1 polyad size}
|\cK_{j+1}(\hat{P}^{(j)}_1(\hat{\by}))|  = (1\pm \nu) \prod_{i=1}^{j} (a^{\sP}_i)^{-\binom{j+1}{i}} n^{j+1} \text{ and } |\cK_{j+1}(\hat{Q}^{(j)}_1(\hat{\bx}))| = (1\pm \nu) \prod_{i=1}^{j} (a^{\sQ}_i)^{-\binom{j+1}{i}} n^{j+1}.
\end{align}
By \ref{item:P11}, for each $j\in [k-2]$, $\hat{\by}\in \hat{A}(j+1,j,\ba^{\sP})$ and $b\in [a_{j+1}^{\sP}]$, we have
\begin{align}\label{eq: P'1 hypergraph not too small}
|P^{(j+1)}_1(\hat{\by},b)| = (1/a^{\sP}_{j+1} \pm \overline{\epsilon}(\ba^{\sP}))|\cK_{j+1}(\hat{P}^{(j)}_1(\hat{\by}))| = 
(1\pm 2\nu) \prod_{i=1}^{j+1} (a^{\sP}_i)^{-\binom{j+1}{i}} n^{j+1} .
\end{align}

It will be convenient to restrict our attention to the $k$-graph $G_1^\kk$ which consists of the crossing $k$-sets of $H_1^\kk$ 
with respect to $\sQ_1^{(1)}$ (rather than $H_1^\kk$ itself).

\setcounter{claim}{0}
\begin{claim}\label{cl: we define G1}
Let $G^{(k)}_1:= H^{(k)}_1 \cap\bigcup_{ \hat{\by}\in \hat{A}_k }  \cK_k(\hat{P}_1^{(k-1)}(\hat{\by}))$. 
Then
\begin{enumerate}[label=\rm (G1\arabic*)]
\item\label{item:G11} $\sP_1$ is an $(\epsilon', \ba^{\sP}, d_{\ba^{\sP},k})$-equitable partition of $G_1^{(k)}$.
\item\label{item:G12} $|G^{(k)}_1 \triangle H^{(k)}_1 |\leq \nu \binom{n}{k}$.
\end{enumerate}
\end{claim}

\begin{proof}
Consider $\hat{\bx} \in \hat{A}(k,k-1,\ba^{\sQ})$ and $\hat{\by} \in \hat{A}_k(\hat{\bx})$.
Note that \eqref{eq: 1 polyad size} implies that $$|\cK_{k}(\hat{P}^{(k-1)}_1(\hat{\by}))| \geq \epsilon' |\cK_{k}(\hat{Q}^{(k-1)}_1(\hat{\bx}))|.$$ Also by (P1)$_{\ref{lem: similar}}$, $H_1^{(k)}$ is $(\epsilon,d_{\ba^{\sQ},k}(\hat{\bx}))$-regular with respect to $\hat{Q}^{(k-1)}_1(\hat{\bx})$.
Thus by Lemma~\ref{lem: simple facts 1}(ii) $H^{(k)}_1$ is $(\epsilon', d_{\ba^{\sQ},k}(\hat{\bx}))$-regular with respect to $\hat{P}_1^{(k-1)}(\hat{\by})$. Together with the definition of $d_{\ba^{\sQ},k}(\hat{\by})$ this in turn shows that for all $\hat{\by} \in \hat{A}(k,k-1,\ba^{\sP})$ we have that 
$G_1^{(k)}$ is $(\epsilon',d_{\ba^{\sP},k}(\hat{\by}))$-regular with respect to $\hat{P}^{(k-1)}_1(\hat{\by})$.
Thus \ref{item:G11} holds.

Note that \ref{item:P12} and the definition of $\hat{A}_k$ imply that
$$G^{(k)}_1 \triangle H^{(k)}_1  \subseteq 
\binom{V}{k}\setminus \cK_{k}(\sQ_1^{(1)}).$$
Since $\sQ_1$ is $(1/a_1^{\sQ},\epsilon,\ba^{\sQ})$-equitable and $1/a_1^{\sQ}\ll \nu, 1/k$, we obtain
$$
\left|\binom{V}{k}\setminus \cK_{k}(\sQ_1^{(1)})\right| \stackrel{\eqref{eq: eta a1}}{\leq} \frac{k^2}{a_1^{\sQ}} \binom{n}{k} \leq \nu \binom{n}{k}.$$
This proves \ref{item:G12}.
\end{proof}

\begin{step}\label{step2}
Refining $\sQ_2$ into $\sP_2$ which mirrors $\sP_1$.
\end{step}

We have now set up the required definitions for the objects on $V_1$ and will now proceed with the objects on $V_2$.
By using Lemma~\ref{lem: partition refinement} with $\sQ_2$, $\ba^{\sP}$, $t'$ playing the roles of $\sP, \bb, t$, respectively, 
we can obtain 
a $(1/a_1^{\sQ},\epsilon^{1/3},\ba^{\sP})$-equitable family of partitions $\sP_2=\sP_2(k-1,\ba^{\sP})$ such that $\sP_2\prec \sQ_2$. 
By considering an appropriate $\ba^{\sP}$-labelling, 
we may assume that 
 for each $j\in [k-1]$, $(\hat{\by},b')\in \hat{A}(j,j-1,\ba^{\sP})\times [a_j^{\sP}]$ and $(\hat{\bx},b)\in \hat{A}(j,j-1,\ba^{\sQ})\times [a_j^{\sQ}]$, we have
$$P^{(j)}_2(\hat{\by},b') \subseteq Q^{(j)}_2(\hat{\bx},b) \text{ if and only if }(\hat{\by},b') \in A_j(\hat{\bx},b).$$%
\COMMENT{We know that $\hat{P}^{(j)}_1(\hat{\by},b') \subseteq \hat{Q}^{(j)}_1(\hat{\bx},b) \text{ if and only if }(\hat{\by},b') \in A_j(\hat{\bx},b)$ from the definition of $A_j$. So we just see the way $\sP_1$ is labelled, and mimic it into $\sP_2$, then we get this. Not sure whether we need to prove this or just say that we can do it.}

Again Lemma~\ref{lem: counting} and the fact that $\epsilon'\ll \nu, \norm{\ba^{\sP}}^{-k}$
imply that for each $j\in [k-1]$, $\hat{\by}\in \hat{A}(j+1,j,\ba^{\sP})$ and $\hat{\bx}\in \hat{A}(j+1,j,\ba^{\sQ})$,
we have
\begin{align}\label{eq: 2 polyad size}
\begin{split}
|\cK_{j+1}(\hat{P}^{(j)}_2(\hat{\by}))|  &= (1\pm \nu) \prod_{i=1}^{j} (a^{\sP}_i)^{-\binom{j+1}{i}} m^{j+1} \enspace \text{ and } \\
|\cK_{j+1}(\hat{Q}^{(j)}_2(\hat{\bx}))| &= (1\pm \nu) \prod_{i=1}^{j} (a^{\sQ}_i)^{-\binom{j+1}{i}} m^{j+1}.
\end{split}
\end{align}

Let $$G^{(k)}_2:= H^{(k)}_2 \cap\bigcup_{ \hat{\by}\in \hat{A}_k }  \cK_k(\hat{P}_2^{(k-1)}(\hat{\by})).$$
Similarly as in Claim~\ref{cl: we define G1},\COMMENT{Also similarly as in the proof of Claim~\ref{cl: we define G1},
for each $\hat{\bx}\in \hat{A}(k,k-1,\ba^{\sQ})$ and $\hat{\by} \in \hat{A}_k(\hat{\bx})$, 
the $k$-graph $H_2^{(k)}$ is $(\epsilon', d_{\ba^{\sQ},k}(\hat{\by}))$-regular with respect to $\hat{P}_2^{(k-1)}(\by)$.}
we conclude the following.
\begin{enumerate}[label=(G2\arabic*)]
\item\label{item:G21} $\sP_2$ is an $(\epsilon', \ba^{\sP}, d_{\ba^{\sP},k})$-equitable partition of $G_2^{(k)}$.
\item\label{item:G22} $|G^{(k)}_2 \triangle H^{(k)}_2 |\leq \nu \binom{n}{k}$.
\end{enumerate}

\begin{step}\label{step3}
Modifying $\sO_1$ into $\sO_1'$ with $\sP\prec \sO_1'$.
\end{step}

Recall that $\sP_1\prec_{\nu} \sO_1$ by \ref{item:P12}.
We next replace $\sO_1$ by a very similar family of partitions $\sO'_1$ such that $\sP_1\prec \sO'_1$.
To this end we apply Lemma~\ref{lem: reconstruction} with $\sO_1, \sP_1$ playing the roles of $\sO , \sP$, respectively, and obtain $\sO'_1=\sO'_1(k-1,\ba^{\sO})$ such that
\begin{enumerate}[label=(O$'$1\arabic*)]
\item\label{item:O'11} $\sP_1\prec \sO'_1$.
\item\label{item:O'12} $\sO'_1$ is a $(1/a_1^{\sO},\epsilon_0+ \nu^{1/20},\ba^{\sO},\nu^{1/20})$-equitable family of partitions 
which is an $(\epsilon_0+\nu^{1/20},d_{\ba^{\sO},k})$-partition of $H^{(k)}_1$.
\item\label{item:O'13} for each $j\in [k-1]$ and $ (\hat{\bx},b) \in \hat{A}(j,j-1,\ba^{\sO})\times [a_j^{\sO}]$, we have $ |O'^{(j)}(\hat{\bx},b)\triangle O^{(j)}(\hat{\bx},b)| \leq \nu^{1/2} \binom{n}{j}.$
\end{enumerate}

Note that since $(\epsilon_0/2,\ba^{\sO},d_{\ba^{\sO},k})$ is a regularity instance and $\nu\ll \epsilon_0$, we have 
$$\epsilon_0+\nu^{1/20} \leq \norm{\ba^{\sO}}^{-4^k} \cdot \epsilon_{\ref{lem: counting}}(\|\ba^{\sO}\|_\infty^{-1}, \|\ba^{\sO}\|_\infty^{-1},k-1,k).$$ 
Thus Lemma~\ref{lem: counting} implies for any $ j \in [k-1]$ and $\hat{\bw} \in \hat{A}(j+1,j,\ba^{\sO})$, we have
\begin{align}\label{eq: O'1 hat hypergraph not too small}
|\cK_{j+1}(\hat{O}'^{(j)}_1(\hat{\bw}))| \geq \epsilon_0^{1/2} n^{j+1}.
\end{align}
Also, \ref{item:O'12} implies that
for all $j\in[k-2]$, $\hat{\bw} \in \hat{A}(j+1,j,\ba^{\sO})$ and $b\in [a_{j+1}^{\sO}]$, we have
\begin{align}\label{eq: O'1 hypergraph not too small}
|O'^{(j+1)}_1(\hat{\bw},b)| \geq (1/a_{j+1}^{\sO} -2\epsilon_0)|\cK_{j+1}(\hat{O}'^{(j)}_1(\hat{\bw}))|  \geq \epsilon_0^{2/3} n^{j+1}.
\end{align}

Moreover, by (O$'$12), (O$'$13) and \ref{item:G12}, we can apply Lemma~\ref{lem: slightly different partition regularity} with $\sO'_1, \sO'_1, H^{(k)}_1$ and $G_1^{(k)}$ playing the roles of $\sP, \sQ, H^{(k)}$ and $G^{(k)}$ to obtain that
\begin{equation}\label{eq: O'1 is also good partition too}
\begin{minipage}[c]{0.8\textwidth} \em
$\sO'_1$ is an $(\epsilon_0 + 2\nu^{1/20}, d_{\ba^{\sO},k})$-partition of $G_1^{(k)}$.
\end{minipage}
\end{equation}

\begin{step}\label{step4}
Describing $\sO_1'$ in terms of its refinement $\sP_1$.
\end{step}
We now describe how the partition classes and polyads of $\sO_1'$ can be expressed in terms of $\sP_1$.
This description will be used to construct $\sO_2'$ from $\sP_2$ in Step~\ref{step5}.

For each $j\in [k-2]$,
our next aim is to define $B_{j+1}(\hat{\bw},b)$ for $\hat{\bw}\in \hat{A}(j+1,j,\ba^{\sO})$ and $b\in [a^{\sO}_{j+1}]$ 
in a similar way as we defined $A_{j+1}(\bx,b)$ for $\hat{\bx}\in \hat{A}(j+1,j,\ba^{\sQ})$ and $b\in [a^{\sQ}_{j+1}]$ in \eqref{eq: def Aj hat bx b and Aj}.
To this end, for each $b\in [a_1^{\sO}]$, let
$$B_1(b,b) := \{ (b',b') \in \hat{A}(1,0,\ba^{\sP})\times [a^{\sP}_1] : P^{(1)}_1(b',b') \subseteq O'^{(1)}_1(b,b) \}.$$
For each $j\in [k-1]$, let
$$\hat{B}_{j+1} := \left\{\hat{\bu} \in \hat{A}(j+1,j,\ba^{\sP}) : \left|\bu^{(1)}_* \cap \{b' : (b',b')\in B_1(b,b)\}\right|\leq 1 \text{ for each } b\in [a_1^{\sO}] \right\}.$$ 
Note that this easily implies that\COMMENT{Note that by \eqref{eq: 1 polyad size}, $\cK_{j+1}(\hat{P}^{(j)}_1(\hat{\bu}))$ is never an emptyset. }
\begin{align}\label{eq: hat Bj if crossing in O}
\hat{\bu} \in \hat{B}_{j+1} \enspace \text{if and only if} \enspace \cK_{j+1}(\hat{P}^{(j)}_1(\hat{\bu})) \subseteq \cK_{j+1}(\sO_1'^{(1)}).
\end{align}
Consider any $j\in [k-1]$ and $\hat{\bw}\in \hat{A}(j+1,j,\ba^{\sO})$.
Let 
\begin{align}\label{eq: B hat bw b  def2}
&\hat{B}_{j+1}(\hat{\bw}):= \left\{ \hat{\bu} \in \hat{A}(j+1,j,\ba^{\sP}): \cK_{j+1}(\hat{P}^{(j)}_1(\hat{\bu}))\subseteq \cK_{j+1}(\hat{O'}^{(j)}_1(\hat{\bw}))\right\}.
\end{align}
Together with (O$'$11) and Proposition~\ref{prop: hat relation}(xi) this implies that 
\begin{align}\label{eq: cK hat O' bw is union of hat Bj+1}
\cK_{j+1}(\hat{O'_1}^{(j)}(\hat{\bw})) = \bigcup_{\hat{\bu}\in \hat{B}_{j+1}(\hat{\bw})} \cK_{j+1}(\hat{P_1}^{(j)}(\hat{\bu})).
\end{align}
Moreover, if $ j \in [k-2]$ and $b\in [a_{j+1}^{\sO}]$, let
\begin{align}\label{eq: B hat bw b  def}
&B_{j+1}(\hat{\bw},b):=
 \left\{ (\hat{\bu},b') \in \hat{A}(j+1,j,\ba^{\sP})\times [a^{\sP}_{j+1}] : P^{(j+1)}_1(\hat{\bu},b') \subseteq O'^{(j+1)}_1(\hat{\bw},b)\right\}.
\end{align}
Thus \ref{item:O'11}, \eqref{eq: 1 polyad size} and \eqref{eq: O'1 hat hypergraph not too small} imply that for all $j\in [k-1]$ and $\hat{\bw}\in \hat{A}(j+1,j,\ba^{\sO})$
\begin{align}\label{eq: Bj hat bw size}
|\hat{B}_{j+1}(\hat{\bw})| \geq \frac{| \cK_{j+1}( \hat{O'}^{(j)}_1(\hat{\bw}))|}{(1 + \nu) \prod_{i=1}^{j} (a^{\sP}_i)^{-\binom{j+1}{i}} n^{j+1}} \geq \frac{1}{2}\epsilon_0^{1/2}\prod_{i=1}^{j}(a^{\sP}_i)^{\binom{j+1}{i}}.
\end{align}
Similarly,
\ref{item:O'11}, \eqref{eq: P'1 hypergraph not too small}  and \eqref{eq: O'1 hypergraph not too small}
imply that for all $j\in [k-1]\setminus\{1\}$, $\hat{\bw}\in \hat{A}(j,j-1,\ba^{\sO})$ and $b\in [a_j^\sO]$,
\begin{align}\label{eq: Bj wb size}
|B_j(\hat{\bw},b)| \geq \frac{| O'^{(j-1)}_1(\hat{\bw},b)|}{(1+ 2\nu) \prod_{i=1}^{j} (a^{\sP}_i)^{-\binom{j}{i}} n^{j} } \geq \frac{1}{2}\epsilon_0^{2/3}\prod_{i=1}^{j}(a^{\sP}_i)^{\binom{j}{i}} > 0.
\end{align}
Note that by Proposition~\ref{prop: hat relation}(xi) and \ref{item:O'11},  for each $j\in [k-1]$, we have
\begin{align}\label{eq: cKjP precs cKjO}
\{\cK_{j+1}( \hat{P}^{(j)}_1(\hat{\bu})) :\hat{\bu}\in \hat{A}(j+1,j,\ba^{\sP})\}
\prec \{\cK_{j+1}(\hat{O'}_1^{(j)}(\hat{\bw})): \hat{\bw}\in \hat{A}(j+1,j,\ba^{\sO})\}.
\end{align}
Together with \eqref{eq: hat Bj if crossing in O} and Proposition~\ref{prop: hat relation}(vi) applied to $\sO'_1$, this implies that $\hat{\bu} \in \hat{B}_{j+1}$ if and only if $\hat{\bu} \in \hat{B}_{j+1}(\hat{\bw})$ for some $\hat{\bw}\in  \hat{A}(j+1,j,\ba^{\sO})$.\COMMENT{
Together with \eqref{eq: hat Bj if crossing in O} and Proposition~\ref{prop: hat relation}(vi) applied to $\sO'_1$, this implies that $\cK_{j+1}( \hat{P}^{(j)}_1(\hat{\bu})) \subseteq \cK_{j+1}( \hat{\sO'_1}^{(1)})$ if and only if there exists $\hat{\bw} \in \hat{A}(j+1,j,\ba^{\sO})$ such that $\cK_{j+1}(\hat{P}^{(j)}_1(\hat{\bu})) \subseteq \cK_{j+1}(\hat{O'_1}^{(j)}(\hat{\bw}))$. In other words, by using \eqref{eq: hat Bj if crossing in O}, we can see that $\hat{\bu} \in \hat{B}_{j+1}$ if and only if $\hat{\bu} \in \hat{B}_{j+1}(\hat{\bw})$ for some $\hat{\bw}\in  \hat{A}(j+1,j,\ba^{\sO})$.} 
Thus for each $j\in [k-1]$,
\begin{equation}\label{eq: union of hat Bj is hat Bj}
\begin{minipage}[c]{0.8\textwidth}\em
$\{  \hat{B}_{j+1}(\hat{\bw})  : \hat{\bw} \in \hat{A}(j+1,j,\ba^{\sO}) \} $ forms a partition of $ \hat{B}_{j+1}$.
\end{minipage}
\end{equation}

The following observations relate polyads and partition classes of $\sO_1'$ and $\sP_1$.
They will be used in the proof of Claim~\ref{eq: sO'2 family of partitions} to relate $\sO_2'$ (which is constructed in Step~\ref{step5}) and $\sP_2$.

\begin{claim}\label{eq: two sets are same}
\begin{enumerate}[label=\rm (\roman*)]
	\item For all $j\in[k-1]\sm\{1\}$ and $\hat{\bw} \in \hat{A}(j,j-1,\ba^{\sO})$, we have
\begin{align*}
\bigcup_{b\in [a_j^{\sO}]} B_j(\hat{\bw},b) = \hat{B}_j(\hat{\bw}) \times [a_j^{\sP}].
\end{align*}
	\item For all $j\in [k-1]$ and $\hat{\bw}\in \hat{A}(j+1,j,\ba^{\sO})$, we have
$$\left\{ (\hat{\bv},\bu^{(j)}_{\bv^{(1)}_*}) : \hat{\bu}\in \hat{B}_{j+1}(\hat{\bw}), \hat{\bv}\leq_{j,j-1}\hat{\bu}\right\}\subseteq \bigcup_{\hat{\bz}\leq_{j,j-1}\hat{\bw}}  B_{j}(\hat{\bz},\bw^{(j)}_{\bz^{(1)}_*}).$$
\end{enumerate}

\end{claim}
\begin{proof}
We first prove (i).
Note that for all $j\in [k-1]\setminus\{1\}$,  $(\hat{\bu},b')\in  \hat{A}(j,j-1,\ba^{\sP})\times [a_j^{\sP}]$ and 
$(\hat{\bw},b)\in  \hat{A}(j,j-1,\ba^{\sO})\times [a_j^{\sO}]$ with $(\hat{\bu},b')\in B_j(\hat{\bw},b)$, we have
$$P_1^{(j)}(\hat{\bu},b') \subseteq O'^{(j)}_1(\hat{\bw},b) \subseteq \cK_{j}(\hat{O'_1}^{(j-1)}(\hat{\bw})).$$
Since $P_1^{(j)}(\hat{\bu},b') \subseteq \cK_{j}( \hat{P}^{(j-1)}_1(\hat{\bu}))$, this means
$\cK_{j}( \hat{P}^{(j-1)}_1(\hat{\bu}))\cap \cK_{j}(\hat{O'}^{(j-1)}(\hat{\bw})) \neq \emptyset.$ 
By \eqref{eq: cKjP precs cKjO} this in turn implies that 
$\cK_j(\hat{P}^{(j-1)}_1(\hat{\bu}))\subseteq \cK_j(\hat{O'_1}^{(j-1)}(\hat{\bw}))$, and thus $\hat{\bu} \in \hat{B}_j(\hat{\bw})$. 
On the other hand, if $\hat{\bu}\in \hat{B}_j(\hat{\bw})$, then \ref{item:O'11} implies that for each $b'\in [a_j^{\sP}]$ there exists $b\in [a_j^{\sO}]$ such that $P_1^{(j)}(\hat{\bu},b')  \subseteq O'^{(j)}_1(\hat{\bw},b)$, and thus $(\hat{\bu},b') \in B_j(\hat{\bw},b)$.

We now prove (ii).
Recall that for each $\hat{\bw}\in \hat{A}(j+1,j,\ba^{\sO})$, $\hat{O'_1}^{(j)}(\hat{\bw})$ 
satisfies \eqref{eq:hatPconsistsofP(x,b)}. 
Together with (O$'$11) this implies that 
\begin{align}\label{eq:Phat1}
\hat{O'_1}^{(j)}(\hat{\bw})  = \bigcup_{\hat{\bz}\leq_{j,j-1}\hat{\bw}}  {O'_1}^{(j)}(\hat{\bz},\bw^{(j)}_{\bz^{(1)}_*}) \stackrel{\eqref{eq: B hat bw b  def}}{=} 
\bigcup_{\hat{\bz}\leq_{j,j-1}\hat{\bw}}
\bigcup_{(\hat{\bv},b')\in B_{j}(\hat{\bz}, \bw^{(j)}_{\bz^{(1)}_*})} P^{(j)}_1(\hat{\bv},b').
\end{align}
Then 
\begin{eqnarray*}
& &(\hat{\bv},b') \in \{ (\hat{\bv},\bu^{(j)}_{\bv^{(1)}_*}) : \hat{\bu}\in \hat{B}_{j+1}(\hat{\bw}), \hat{\bv}\leq_{j,j-1}\hat{\bu}\}\\
&\stackrel{\eqref{eq:hatPconsistsofP(x,b)}, \eqref{eq: B hat bw b  def2}}{\Longrightarrow }&   \exists \hat{\bu} \in \hat{A}(j+1,j,\ba^{\sP}) :  \cK_{j+1}(\hat{P}^{(j)}_1(\hat{\bu}))\subseteq \cK_{j+1}(\hat{O'}_1^{(j)}(\hat{\bw})), P_1^{(j)}(\hat{\bv},b') \subseteq \hat{P}_1^{(j)}(\hat{\bu})  \\
& \stackrel{\eqref{def:polyad} }{\Longrightarrow } & \exists (I,J) \in P_1^{(j)}(\hat{\bv},b') \times \cK_{j+1}(\hat{O'}_1^{(j)}(\hat{\bw})) : I \subseteq J\\
& \stackrel{\eqref{def:polyad},({\rm O}'11)}{\Longrightarrow} & P^{(j)}_1(\hat{\bv},b') \subseteq \hat{O'_1}^{(j)}(\hat{\bw}) \\
& \stackrel{ \eqref{eq:Phat1}}{\Longrightarrow} & 
(\hat{\bv},b') \in \bigcup_{\hat{\bz}\leq_{j,j-1}\hat{\bw}}  B_{j}(\hat{\bz},\bw^{(j)}_{\bz^{(1)}_*}).
\end{eqnarray*}
This proves the claim.
\end{proof}

\begin{step}\label{step5}
Constructing $\sO_2'$ from $\sP_2$.
\end{step}

Together $B_j(\bw,b)$ and $\hat{B}_j(\bw)$  encode how $\sO_1'$ can be refined into $\sP_1$.
We now use this information to construct $\sO_2'$ from $\sP_2$. 
Claim~\ref{eq: sO'2 family of partitions} then shows that this construction indeed yields a family of partitions (whose polyads can be expressed in terms of those of $\sP_2$).
Finally, Claim~\ref{cl: O'2 bw b is epsilon0 regular wrt} shows that the partition classes are appropriately regular.

For each $b\in [a^{\sO}_1]$, we let 
\begin{align}\label{eq: O'(1)(b,b) def}
O'^{(1)}_2(b,b):= \bigcup_{(b',b') \in B_1(b,b)} P^{(1)}_2(b',b').
\end{align}
We also let $\sO_2'^{(1)}:= \{ O_2'^{(1)}(b,b): b\in [a_1^{\sO}]\}$.
Again, as in \eqref{eq: hat Bj if crossing in O}, this easily implies that for each $j\in [k-1]$
\begin{align}\label{eq: hat Bj if crossing in O2}
\hat{\bu} \in \hat{B}_{j+1} \enspace \text{if and only if} \enspace \cK_{j+1}(\hat{P}^{(j)}_2(\hat{\bu})) \subseteq \cK_{j+1}(\sO_2'^{(1)}).
\end{align}
Note that for each $b\in [a^{\sO}_1]$, by (O$'$12), we have that 
\begin{eqnarray}\label{eq: how equitable O'2(1) is}
 |O'^{(1)}_2(b,b)| &=& \sum_{(b',b')\in B_1(b,b)} | P^{(1)}_2(b',b') | =  (1\pm 2\nu^{1/20})m/a_1^{\sO}.
\end{eqnarray}\COMMENT{
\begin{eqnarray*}
\sum_{(b',b')\in B_1(b,b)} | P^{(1)}_2(b',b') | &=& \sum_{(b',b')\in B_1(b,b)} (\frac{m}{a_1^{\sP}} \pm 1)   \nonumber \\
 & =&  \frac{m}{n} \left(\sum_{(b',b')\in B_1(b,b)}  \frac{n}{a_1^{\sP}} \right) \pm  |B_1(b,b)| \nonumber \\
 & =& \frac{m}{n} \left( |O'^{(1)}_1(b,b)| \pm |B_1(b,b)| \right) \pm |B_1(b,b)| \nonumber \\
 &\stackrel{ {\rm(O}'{\rm12)}}{=}&   \frac{m}{n} \left( (1\pm \nu^{1/20})n/a_1^{\sO} \pm |B_1(b,b)| \right) \pm |B_1(b,b)| \nonumber \\
 & = & (1\pm 2\nu^{1/20})m/a_1^{\sO}.
\end{eqnarray*}
We get the last equality since $1/m \ll \nu, 1/T$ and $|B_1(b,b)| \leq T$.}
In analogy to \eqref{eq: B hat bw b  def}, 
for each $j\in [k-1]\setminus\{1\}$, $\hat{\bw}\in \hat{A}(j,j-1,\ba^{\sO})$, and $b\in [a^{\sO}_j]$,
we define
\begin{align}\label{eq: Q'2 bwb def}
O'^{(j)}_2(\hat{\bw},b)
:= \bigcup_{(\hat{\bu},b')\in B_j(\hat{\bw},b)} P^{(j)}_2(\hat{\bu},b'),
\end{align}
and for each $j\in [k-1]\setminus \{1\}$, we let
$$\sO'^{(j)}_2:= \{ O'^{(j)}_2(\hat{\bw},b) :\hat{\bw}\in \hat{A}(j,j-1,\ba^{\sO}), b\in [a^{\sO}_j]\}.$$
Moreover, let ${\sO'_2}:= \{{\sO_2'}^{(j)}\}_{j=1}^{k-1}$.
Note that since $\sP_2$ is a family of partitions,\COMMENT{Note that each part in a family of partitions are non-empty} \eqref{eq: Bj wb size} and \eqref{eq: Q'2 bwb def} imply that $O'^{(j)}_2(\hat{\bw},b)$ is nonempty for each $j\in [k-1]\setminus\{1\}$ and $(\hat{\bw},b)\in \hat{A}(j,j-1,\ba^{\sO})\times [a_j^{\sO}]$.
The construction of $\sO'_2$ also gives rise to a natural description of all polyads.
Indeed, as in \eqref{eq:hatPconsistsofP(x,b)}, we define for all $j\in [k-1]$ and $\hat{\bw}\in \hat{A}(j+1,j,\ba^{\sO})$
\begin{eqnarray}\label{eq: hat O'2 naturally define}
\hat{O'}^{(j)}_2(\hat{\bw})
&:=& \bigcup_{\hat{\bz}\leq_{j,j-1}\hat{\bw}} O'^{(j)}_2(\hat{\bz},\bw^{(j)}_{\bz^{(1)}_*}) \\
&\stackrel{\eqref{eq: Q'2 bwb def}}{=}& \bigcup_{\hat{\bz}\leq_{j,j-1}\hat{\bw}} \bigcup_{ (\hat{\bv},b')\in B_{j}(\hat{\bz},\bw^{(j)}_{\bz^{(1)}_*}) }P^{(j)}_2(\hat{\bv},b').\label{eq: hat O' j-1 2 hat bw consists} 
\end{eqnarray}

\begin{claim} \label{eq: sO'2 family of partitions}
$\sO'_2$ is a family of partitions on $V_2$. Moreover, 
for all $j\in [k-1]$ and $\hat{\bw} \in \hat{A}(j+1,j,\ba^{\sO})$, we have
\begin{align}\label{eq: cKj+1 hat O'2(j) is..}
\cK_{j+1}({\hat{O}_2}^{\prime(j)}(\hat{\bw})) = \bigcup_{\hat{\bu}\in \hat{B}_{j+1}(\hat{\bw})} \cK_{j+1}(\hat{P}^{(j)}_2(\hat{\bu})).
\end{align}
\end{claim}
\begin{proof}
We will prove Claim~\ref{eq: sO'2 family of partitions} by applying the criteria in Lemma~\ref{lem: family of partitions construction}.
For each $j\in [k-1]\sm \{1\}$, $\hat{\bw} \in \hat{A}(j,j-1,\ba^{\sO})$ and $b\in [a_{j}^{\sO}]$, let $\phi^{(j)}( O'^{(j)}_2(\hat{\bw},b)) := b$. 
Let $\ell \in [k-1]$ be the largest number satisfying the following.
\begin{itemize}
\item[(OP1)$_{\ell}$] $\{\sO_2'^{(j)}\}_{j=1}^{\ell}$ is a family of partitions,
\item[(OP2)$_{\ell}$] Let $O^{(j)}_*(\cdot,\cdot)$ and $\hat{O}_*^{(j)}(\cdot)$ be the maps defined as in \eqref{eq: hat P def}--\eqref{eq: hat P P relations emptyset} for $\{\sO_2'^{(j)}\}_{j=1}^{k}$ and $\{\phi^{(j)}\}_{j=2}^{k}$. Then $\{\phi^{(j)}\}_{j=2}^{\ell}$ is an $(a_1^\sO,\dots, a_\ell^\sO)$-labelling of $\{\sO'^{(j)}_2\}_{j=1}^{\ell}$ such that for each $j\in [\ell]\setminus\{1\}$ and $(\hat{\bw},b)\in \hat{A}(j,j-1,\ba^{\sO}) \times [a_j^{\sO}]$, we have 
$$O_*^{(j)}(\hat{\bw},b) = O_2'^{(j)}(\hat{\bw},b),$$ 
and for each $j \in [\ell]$ and $\hat{\bw} \in \hat{A}(j+1,j,\ba^{\sO})$ we have
\begin{align*}
	\hat{O}_*^{(j)}(\hat{\bw}) = {\hat{O}_2}^{\prime(j)}(\hat{\bw}).
\end{align*}
\end{itemize}
\COMMENT{First we check that (OP1)$_{1}$--(OP2)$_{1}$ hold.
It is clear that $\sO_2'^{(1)}$ forms a family of partitions and $O_*^{(1)}(b,b) = O_2'^{(1)}(b,b)$ for each $b\in [a_1^{\sO}]$.  By \eqref{eq:hatPconsistsofP(x,b)} and \eqref{eq: hat O'2 naturally define}, 
for each $\hat{\bw} \in \hat{A}(2,1,\ba^{\sO})$ we have
$$\hat{O_*}^{(1)}(\hat{\bw}) = \bigcup_{ b\in \bw^{(1)}_*} O_*^{(1)}(b,b) =  \bigcup_{ b\in \bw^{(1)}_*} O_2'^{(1)}(b,b) = \hat{O'_2}^{(1)}(\hat{\bw}).$$}
It is easy to check that (OP1)$_{1}$--(OP2)$_{1}$ hold and thus $\ell\geq 1$.
Since $\{\sO_2'^{(j)}\}_{j=1}^{\ell}$ is a family of partitions, $\hat{\sO'_2}^{(j)}$ is well-defined for each $j\in [\ell]$.
Claim~\ref{eq: two sets are same}(ii) now allows us to express (the cliques spanned by) these polyads in terms of those of $\sP_2^{(j)}$.
\setcounter{subclaim}{0}
\begin{subclaim} \label{eq: cKjhatO' = bigcup hat bu b'}
For each $j\in [\ell]$ and $\hat{\bw} \in \hat{A}(j+1,j,\ba^{\sO})$, we have
$$\cK_{j+1}({\hat{O}_2}^{\prime(j)}(\hat{\bw})) = \bigcup_{\hat{\bu}\in \hat{B}_{j+1}(\hat{\bw})} \cK_{j+1}(\hat{P}^{(j)}_2(\hat{\bu})).$$
\end{subclaim}
\begin{proof}
Consider $j\in [\ell]$ and $\hat{\bw} \in \hat{A}(j+1,j,\ba^{\sO})$. Note that 
$$ \bigcup_{\hat{\bu}\in \hat{B}_{j+1}(\hat{\bw})} \hat{P}^{(j)}_2(\hat{\bu}) \stackrel{\eqref{eq:hatPconsistsofP(x,b)}}{=} \bigcup_{\hat{\bu}\in \hat{B}_{j+1}(\hat{\bw})}\bigcup_{ \hat{\bv}\leq_{j,j-1}\hat{\bu} } P_2^{(j)}(\hat{\bv},\bu^{(j)}_{\bv^{(1)}_*}).$$ 
This together with \eqref{eq: hat O' j-1 2 hat bw consists} and Claim~\ref{eq: two sets are same}(ii) implies that 
$ \bigcup_{\hat{\bu}\in \hat{B}_{j+1}(\hat{\bw})} \hat{P}^{(j)}_2(\hat{\bu}) \subseteq {\hat{O}_2}^{\prime(j)}(\hat{\bw})$, thus we obtain
\begin{align}\label{eq: subset eq holds}
\bigcup_{\hat{\bu}\in \hat{B}_{j+1}(\hat{\bw})} \cK_{j+1}(\hat{P}^{(j)}_2(\hat{\bu})) \subseteq \cK_{j+1}({\hat{O}_2}^{\prime(j)}(\hat{\bw})).
\end{align}
On the other hand, we have\COMMENT{To see the first equality below, consider any $J\in \cK_{j+1}(\sO_2'^{(1)})$. Then $J\in \cK_{j+1}(\sP_2^{(1)})$. Thus there exists $\hat{\bu'}\in \hat{A}(j+1,j,\ba^{\sP})$ with $J\in \cK_{j+1}(\hat{P}_2^{(j)}(\hat{\bu'})$. 
Then definition of $\hat{B}_{j+1}$ now implies that $\hat{\bu'}\in \hat{B}_{j+1}$. Thus $\cK_{j+1}(\sO_2'^{(1)}) \subseteq \bigcup_{\hat{\bu}\in \hat{B}_{j+1}} \cK_{j+1}(\hat{P}_2^{(j)}(\hat{\bu}))$. The other direction follows form \eqref{eq: hat Bj if crossing in O2}.}
\begin{align}\label{eq:cliques}
\bigcup_{\hat{\bu} \in \hat{B}_{j+1} } \cK_{j+1}(\hat{P}^{(j)}_2(\hat{\bu})) 
\stackrel{\eqref{eq: hat Bj if crossing in O2}}{=} \cK_{j+1}(\sO_2'^{(1)}) = \bigcup_{\hat{\bw}\in \hat{A}(j+1,j,\ba^{\sO})} \cK_{j+1}({\hat{O}_2}^{\prime(j)}(\hat{\bw})).
\end{align}
(Here the final equality follows from (OP1)$_\ell$, (OP2)$_\ell$ and Proposition~\ref{prop: hat relation}(vi) applied to $\{ O_2'^{(j)}\}_{j=1}^{\ell}$.)
Consider a $(j+1)$-set $J \in \cK_{j+1}({\hat{O}_2}^{\prime(j)}(\hat{\bw}) )$. 
Then by~\eqref{eq:cliques}
there exists $\hat{\bu'} \in \hat{B}_{j+1}$ such that $J\in \cK_{j+1}(\hat{P}^{(j)}_2(\hat{\bu'}))$.

We claim that  $\hat{\bu'} \in \hat{B}_{j+1}(\hat{\bw})$.
Indeed if not, then by \eqref{eq: union of hat Bj is hat Bj}, there exists $\hat{\bw'} \in \hat{A}(j+1,j,\ba^{\sO})\setminus\{\hat{\bw}\}$ such that $\hat{\bu'} \in \hat{B}_{j+1}(\hat{\bw'})$, thus we have 
$$J \in \bigcup_{\hat{\bu}\in \hat{B}_{j+1}(\hat{\bw'})} \cK_{j+1}(\hat{P}^{(j)}_2(\hat{\bu})) \stackrel{\eqref{eq: subset eq holds}}{\subseteq} \cK_{j+1}({\hat{O}_2}^{\prime(j)}(\hat{\bw'}) ).$$ 
Hence $\cK_{j+1}({\hat{O}_2}^{\prime(j)}(\hat{\bw}) )\cap \cK_{j+1}({\hat{O}_2}^{\prime(j)}(\hat{\bw'}) )$ is nonempty (as it contains $J$). However, since $\{\sO_2'^{(i)}\}_{i=1}^{j}$ is a family of partitions, this contradicts Proposition~\ref{prop: hat relation}(vi) and (ix).\COMMENT{Here, (vi) alone does not forbid the possibility of $\cK_{j+1}({\hat{O'}_2}^{(j)}(\hat{\bw}) ) = \cK_{j+1}({\hat{O'}_2}^{(j)}(\hat{\bw'}) )$. Thus (ix) is included here. }
Hence, we have $\hat{\bu'} \in \hat{B}_{j+1}(\hat{\bw})$, thus $J\in  \bigcup_{\hat{\bu}\in \hat{B}_{j+1}(\hat{\bw})} \cK_{j+1}(\hat{P}^{(j)}_2(\hat{\bu}))$. 
The fact that this holds for arbitrary $J \in \cK_{j+1}({\hat{O}_2}^{\prime(j)}(\hat{\bw}) )$ combined with \eqref{eq: subset eq holds} proves the subclaim.
\end{proof}

Now, if $\ell=k-1$, then $\sO_2'$ is a family of partitions and Subclaim~\ref{eq: cKjhatO' = bigcup hat bu b'} implies the moreover part of Claim~\ref{eq: sO'2 family of partitions}.
Assume that $\ell < k-1$ for a contradiction. 
Now we show that $\{\sO_2'^{(j)}\}_{j=1}^{\ell+1}$ and the maps $\{ O_2'^{(j)}(\cdot, \cdot), \hat{O_2'}^{(j)}(\cdot)\}_{j=1}^{\ell+1}$ satisfy the conditions \ref{item:FP1}--\ref{item:FP3} in Lemma~\ref{lem: family of partitions construction}. 
Condition \ref{item:FP1} follows from (OP1)$_\ell$,  \eqref{eq: Bj wb size} and \eqref{eq: Q'2 bwb def}.  
Condition \ref{item:FP3} also holds because of \eqref{eq: hat O'2 naturally define} and the assumption that $\ell<k-1$. 
Property (OP1)$_{\ell}$ implies that \ref{item:FP2} holds when $j\in [\ell]$.
To check that \ref{item:FP2} also holds for $j=\ell+1$, consider $\hat{\bw} \in \hat{A}(\ell+1,\ell, \ba^{\sO})$. 
By Claim~\ref{eq: two sets are same}(i) and~\eqref{eq: Bj wb size},\COMMENT{We also use that the $B_{\ell+1}(\hat{\bw},b)$ are disjoint as $\sP_2$ is a family of partitions.} we have that
$\{ B_{\ell+1}(\hat{\bw},1),\dots, B_{\ell+1}(\hat{\bw},a_{\ell+1}^{\sO})\}$  forms a partition of  $\hat{B}_{\ell+1}(\hat{\bw})\times [a_{\ell+1}^{\sP}]$  into nonempty sets.%
\COMMENT{Since $\cK_j(\hat{P}_1^{(j-1)}(\hat{\bu}) ) = \bigcup_{b'\in [a_j^{\sP}] }P_1^{(j)}(\hat{\bu},b')$ }
Thus by \eqref{eq: Q'2 bwb def} and Subclaim~\ref{eq: cKjhatO' = bigcup hat bu b'}, 
$\{O'^{(\ell+1)}_2(\hat{\bw},1),\dots, O'^{(\ell+1)}_2(\hat{\bw},a_{\ell+1}^{\sO})\}$ forms a partition of  $\cK_{\ell+1}(\hat{O}_2^{\prime(\ell)}(\hat{\bw})) $ into nonempty sets.
Thus \ref{item:FP2} holds for $j=\ell+1$.

Hence, by \eqref{eq: O'(1)(b,b) def}
we can apply Lemma~\ref{lem: family of partitions construction} to see that\COMMENT{with $O_2'^{(1)}(b,b), \{\sO_2'^{(j)}\}_{j=1}^{\ell+1}$, $\{ O_2'^{(j)}(\cdot, \cdot)\}_{j=1}^{\ell+1}, \{ \hat{O_2'}^{(j)}(\cdot)\}_{j=1}^{\ell+1}$ playing the roles of $V_b, \sP, \{P'^{(j)}(\cdot,\cdot)\}_{j=1}^{\ell+1}, \{\hat{P'}^{(j)}(\cdot)\}_{j=1}^{\ell+1}$ to obtain that $\{\sO_2'^{(j)}\}_{j=1}^{\ell+1}$ is a family of partitions on $V_2$ and conclude that
$\{\phi^{(j)}\}_{j=2}^{\ell+1}$ is an $(a_1^{\sO},\dots, a_{\ell+1}^{\sO})$-labelling of $\{\sO_2'^{(j)}\}_{j=1}^{\ell+1}$ such that the maps $O_*^{(j)}(\cdot, \cdot)$ and $\hat{O_*}^{(j)}(\cdot)$ defined in  \eqref{eq: hat P def}--\eqref{eq: hat P P relations emptyset} for $\{\sO_2'^{(j)}\}_{j=1}^{\ell+1}$ and $\{\phi^{(j)}\}_{j=2}^{\ell+1}$ coincide with
$O_2'^{(j)}(\cdot, \cdot)$ and $\hat{O_2'}^{(j)}(\cdot)$ for each $j\in [\ell+1]$.} 
(OP1)$_{\ell+1}$ and (OP2)$_{\ell+1}$ hold, a contradiction to the choice of $\ell$.
Thus $\ell=k-1$, which proves the claim.
\end{proof}
By Claim~\ref{eq: sO'2 family of partitions}, $\sO_2'$ is a family of partitions and
\eqref{eq: Q'2 bwb def} implies that $\sP_2 \prec \sO_2'$. 
Consider any $j\in [k-1]$ and $\hat{\bw}\in \hat{A}(j+1,j,\ba^{\sO})$.
Note that 
\begin{align}\label{eq: cK hat O' size}
&|\cK_{j+1}(\hat{O}_1^{\prime(j)}(\hat{\bw}))|
\stackrel{(\ref{eq: cK hat O' bw is union of hat Bj+1})}{=} \sum_{\hat{\bu}\in \hat{B}_{j+1}(\hat{\bw})} |\cK_{j+1}(\hat{P}_1^{(j)}(\hat{\bu}))| 
\stackrel{(\ref{eq: 1 polyad size})}{=} |\hat{B}_{j+1}(\hat{\bw})| (1\pm \nu) \prod_{i=1}^{j}(a^{\sP}_i)^{-\binom{j+1}{i}} n^{j+1}, \nonumber \\
&|\cK_{j+1}(\hat{O}_2^{\prime(j)}(\hat{\bw}))|
\stackrel{(\ref{eq: cKj+1 hat O'2(j) is..})}{=} \sum_{\hat{\bu}\in \hat{B}_{j+1}(\hat{\bw})} |\cK_{j+1}(\hat{P}_2^{(j)}(\hat{\bu}))| 
\stackrel{(\ref{eq: 2 polyad size})}{=} |\hat{B}_{j+1}(\hat{\bw})| (1\pm \nu) \prod_{i=1}^{j}(a^{\sP}_i)^{-\binom{j+1}{i}} m^{j+1}.
\end{align}

For notational convenience, 
for each $\hat{\bw}\in \hat{A}(k,k-1,\ba^{\sO})$,
let 
$$
O'^{(k)}_1(\hat{\bw},1):= G^{(k)}_1\cap \cK_{k}(\hat{O}_1^{\prime(k-1)}(\hat{\bw}))
\enspace \text{ and }\enspace 
O'^{(k)}_2(\hat{\bw},1) := G^{(k)}_2\cap \cK_{k}(\hat{O}_2^{\prime(k-1)}(\hat{\bw})).
$$

\begin{claim}\label{cl: O'2 bw b is epsilon0 regular wrt}
For all $j\in [k-1]$, $\hat{\bw} \in \hat{A}(j+1,j,\ba^{\sO})$ and $b\in [a^{\sO}_{j+1}]$, we have that
$O_2^{\prime(j+1)}(\hat{\bw},b)$ is 
\begin{itemize}
\item $(\epsilon_0+3\nu^{1/20}, 1/a_{j+1}^{\sO})$-regular with respect to $\hat{O}_2^{\prime(j)}(\hat{\bw})$ if $j\leq k-2$, and
\item $(\epsilon_0+3\nu^{1/20}, d_{\ba^{\sO},k}(\hat{\bw}))$-regular with respect to $\hat{O}_2^{\prime(j)}(\hat{\bw})$ if $j=k-1$.
\end{itemize}
\end{claim}

\begin{proof} To prove Claim~\ref{cl: O'2 bw b is epsilon0 regular wrt},
we will apply Lemma~\ref{lem: mimic Kk} (see \ref{item:J1} and \ref{item:J2} and the preceding discussion), which allows us to transfer information about $\sO_1'$ and $\sP_1$ to $\sO_2'$ and $\sP_2$.
Fix $j\in [k-1]$, 
$\hat{\bw}\in \hat{A}(j+1,j,\ba^{\sO})$, and $b\in [a_{j+1}^{\sO}]$. Let 
\begin{align}\label{eq: def d def d ddddd}
d:= \left\{ \begin{array}{ll} 
1/a^{\sO}_{j+1} & \text{ if } j\leq k-2, \\
d_{\ba^{\sO},k}(\hat{\bw}) &\text{ if } j=k-1.
\end{array}\right.
\end{align}
Consider an arbitrary $j$-graph $F^{(j)}\subseteq \hat{O}_2^{\prime(j)}(\hat{\bw})$ with 
\begin{align}\label{eq: F j-1 big}
|\cK_{j+1}(F^{(j)})|\geq ( \epsilon_0+3\nu^{1/20}) |\cK_{j+1}(\hat{O}_2^{\prime(j)}(\hat{\bw}))|.
\end{align}
To prove the claim, it suffices to show that $d(O'^{(j+1)}_2(\hat{\bw},b)\mid F^{(j)})=d\pm (\epsilon_0+3\nu^{1/20})$.
For each $\hat{\by} \in \hat{A}(j+1,j,\ba^{\sP})$, let
\begin{align}\label{eq: def d def d ddddd hat by}
d(\hat{\by}):= \left\{ \begin{array} {ll}
\frac{1}{a^{\sP}_{j+1}} \left|\{ b': (\hat{\by},b')\in B_{j+1}(\hat{\bw},b)\}\right| &\text{ if }  j \leq k-2,\\
d_{\ba^{\sP},k}(\hat{\by}) &\text{ if } \hat{\by}\in \hat{B}_{j+1}(\hat{\bw}), j=k-1,\\
0 &\text{ if } \hat{\by}\notin \hat{B}_{j+1}(\hat{\bw}), j=k-1.
\end{array}\right.
\end{align}
Thus Claim~\ref{eq: two sets are same}(i) and the above definition implies that 
\begin{equation}\label{eq: when d hat by 0}
\begin{minipage}[c]{0.8\textwidth} \em
if $\hat{\by}\notin \hat{B}_{j+1}(\hat{\bw})$, then we have $d(\hat{\by})=0$.
\end{minipage}
\end{equation}
\begin{subclaim}\label{claim:O'regular}
For all $\hat{\by} \in \hat{A}(j+1,j,\ba^{\sP})$ and each $i\in [2]$, we have that
$O'^{(j+1)}_i(\hat{\bw},b)$ is $(\epsilon'^{1/2},d(\hat{\by}))$-regular with respect to $\hat{P}_i^{(j)}(\hat{\by})$.
\end{subclaim}
\begin{proof}
First, we note that if $j\leq k-2$, then by \eqref{eq: B hat bw b  def} and \eqref{eq: Q'2 bwb def},  for $i\in [2]$, we have
$$O'^{(j+1)}_i(\hat{\bw},b)\cap \cK_{j+1}(\hat{P}^{(j)}_i(\hat{\by})) 
= \bigcup_{b'\colon (\hat{\by},b')\in B_{j+1}(\hat{\bw},b)}P^{(j+1)}_i(\hat{\by},b'). $$
Together with \ref{item:G11}, \ref{item:G21} and \eqref{eq: def d def d ddddd hat by} this shows that $O'^{(j+1)}_i(\hat{\bw},b) \cap \cK_{j+1}(\hat{P}_i^{(j)}(\hat{\by}))$ is the disjoint union of $a_{j+1}^{\sP} d(\hat{\by})\leq \norm{\ba^{\sP}}$ hypergraphs, 
each of which is $(\epsilon',1/a_{j+1}^{\sP})$-regular with respect to $\hat{P}^{(j)}_i(\hat{\by})$.
Thus the union lemma~(Lemma~\ref{lem: union regularity}) together with the fact that $\epsilon'\ll 1/\norm{\ba^{\sP}} $ implies Subclaim~\ref{claim:O'regular} in this case.

If $j=k-1$, then we have $b=1$. 
If $\hat{\by}\in \hat{B}_{k}(\hat{\bw})$, then
$\cK_k( \hat{P}_i^{(k-1)}(\hat{\by})) \subseteq \cK_k(\hat{O}_i^{\prime(k-1)}(\hat{\bw}))$ by \eqref{eq: B hat bw b  def2} and Claim~\ref{eq: sO'2 family of partitions}.  Thus
  $$O'^{(k)}_i(\hat{\bw},1)\cap \cK_{k}(\hat{P}_i^{(k-1)}(\hat{\by})) = G^{(k)}_i \cap \cK_{k}(\hat{P}_i^{(k-1)}(\hat{\by})).$$
Together with \ref{item:G11}, \ref{item:G21} and \eqref{eq: def d def d ddddd hat by} this implies Subclaim~\ref{claim:O'regular} in this case.

If $\hat{\by}\notin \hat{B}_k(\hat{\bw})$, 
then by \eqref{eq: B hat bw b  def2}, \eqref{eq: cKjP precs cKjO} and Claim~\ref{eq: sO'2 family of partitions} we have
$${O'_i}^{(k)}(\hat{\bw},1)\cap \cK_{k}(\hat{P}_i^{(k-1)}(\hat{\by})) 
\subseteq \cK_{k}(\hat{O}_i^{\prime(k-1)}(\hat{\bw}))\cap  \cK_{k}(\hat{P}_i^{(k-1)}(\hat{\by})) = \emptyset.$$
Since $d(\hat{\by})=0$ in this case, 
this proves Subclaim~\ref{claim:O'regular}.
\end{proof}

In order to show that $O_2^{\prime(j+1)}(\hat{\bw},b)$ is $(\epsilon_0+3\nu^{1/20})$-regular with respect to $\hat{O}_2^{\prime(j)}(\hat{\bw})$,
we will transfer all the calculations about hypergraph densities from the hypergraphs on $V_2$
to the hypergraphs on $V_1$, 
because there we have much better control over their structure.
To this end we first use Lemma~\ref{lem: mimic Kk} to show the existence of two hypergraphs $J_1^{(j)},J_2^{(j)}$ on $V_1,V_2$, respectively, 
that exhibit a very similar structure in terms of their densities and where $J_2^{(j)}$ is very close to $F^{(j)}$.
Consequently, $J_1^{(j)}$ also resembles $F^{(j)}$.

More precisely, we now apply Lemma~\ref{lem: mimic Kk} with $F^{(j)},  \{\sP_2^{(i)}\}_{i=1}^{j}, \{\sP_1^{(i)}\}_{i=1}^{j}, j$ playing the roles of $H^{(k-1)}, \sP, \sQ, k-1$ respectively (we can do this by \ref{item:P11}, \ref{item:G21} and Claim~\ref{eq: sO'2 family of partitions}).
We obtain $j$-graphs $J_2^{(j)}\sub \cK_{j}(\sP_2^{(1)})$ on $V_2$ and $J_1^{(j)}\sub \cK_{j}(\sP_1^{(1)})$ on $V_1$ such that 
\begin{enumerate}[label=(J\arabic*)]
\item\label{item:J1} $|J_2^{(j)}\triangle F^{(j)}|\leq \nu \binom{m}{j}$, and
\item\label{item:J2}  $ d( \cK_{j+1}(J_2^{(j)}) \mid \hat{P}^{(j)}_2(\hat{\by})) = d( \cK_{j+1}(J_1^{(j)}) \mid \hat{P}^{(j)}_1(\hat{\by})) \pm \nu$ 
for each $\hat{\by}\in \hat{A}(j+1,j,\ba^{\sP})$.
\end{enumerate}
Our next aim is to estimate $|\cK_{j+1}(J_2^{(j)})|$ in terms of $|\cK_{j+1}(J_1^{(j)})|$. 
\begin{eqnarray}\label{eq: J2 J1} 
& & \hspace{-2cm} \left|\cK_{j+1}(J_2^{(j)})\right| \enspace = \sum_{\hat{\by}\in \hat{A}(j+1,j,\ba^{\sP})} \left|\cK_{j+1}(J_2^{(j)}) \cap \cK_{j+1}(\hat{P}^{(j)}_2(\hat{\by}))\right| \nonumber \\
&\stackrel{\eqref{eq: 2 polyad size}}{=}&  \hspace{-0.3cm} (1\pm \nu) \prod_{i=1}^{j} (a^{\sP}_i)^{-\binom{j+1}{i}} m^{j+1} \sum_{\hat{\by}\in \hat{A}(j+1,j,\ba^{\sP})} d( \cK_{j+1}(J_2^{(j)}) \mid \hat{P}^{(j)}_2(\hat{\by})) \nonumber \\
&\stackrel{\rm \ref{item:J2},\eqref{eq: 1 polyad size}}{=}&   \hspace{-0.3cm}
(1\pm 3\nu) \frac{m^{j+1}}{n^{j+1}}\left|\cK_{j+1}(\hat{P}^{(j)}_1(\hat{\by}))\right| 
\sum_{\hat{\by}\in \hat{A}(j+1,j,\ba^{\sP})} (d( \cK_{j+1}(J_1^{(j)}) \mid \hat{P}^{(j)}_1(\hat{\by}))\pm \nu ) \nonumber \\
&=& \hspace{-0.3cm} (1\pm 3\nu)\frac{m^{j+1}}{n^{j+1}}
 \left ( \sum_{\hat{\by}\in \hat{A}(j+1,j,\ba^{\sP})} \hspace{-0.6cm} \left|\cK_{j+1}(J_1^{(j)}) \cap \cK_{j+1}(\hat{P}^{(j)}_1(\hat{\by}))\right|
  \pm \nu \hspace{-0.6cm} \sum_{\hat{\by}\in \hat{A}(j+1,j,\ba^{\sP})} \hspace{-0.6cm} \left|\cK_{j+1}(\hat{P}^{(j)}_1(\hat{\by}))\right| \right) \nonumber \\
&=&  \hspace{-0.3cm} \frac{m^{j+1}}{n^{j+1}}( |\cK_{j+1}(J_1^{(j)})| \pm 5\nu n^{j+1} ).
\end{eqnarray}%
\COMMENT{As a second line on could insert $\sum_{\hat{\by}\in  \hat{A}(j+1,j,\ba^{\sP})} d( \cK_{j+1}(J_2^{(j)}) \mid \hat{P}^{(j)}_2(\hat{\by})) \cdot \left|\cK_{j+1}(\hat{P}^{(j)}_2(\hat{\by}))\right|$}
Similarly, we obtain
\begin{eqnarray}\label{eq: O2 cap J2}
& &\hspace{-2.2cm} \left|O'^{(j+1)}_2(\hat{\bw},b)\cap \enspace \cK_{j+1}(J_2^{(j)})\right| 
=  \sum_{\hat{\by}\in \hat{A}(j+1,j,\ba^{\sP})} \hspace{-0.3cm} \left|O'^{(j+1)}_2(\hat{\bw},b)\cap \cK_{j+1}(J_2^{(j)}) \cap \cK_{j+1}(\hat{P}_2^{(j)}(\hat{\by}))\right| \nonumber \\
&\stackrel{\rm Subcl.~\ref{claim:O'regular}}{=}& \hspace{-0.3cm}
\sum_{\hat{\by}\in \hat{A}(j+1,j,\ba^{\sP})} \left(d(\hat{\by})\left|\cK_{j+1}(J_2^{(j)} )\cap \cK_{j+1}(\hat{P}^{(j)}_2(\hat{\by}))\right| 
 \pm \epsilon'^{1/4} \left|\cK_{j+1}(\hat{P}^{(j)}_2(\hat{\by}))\right|\right) \nonumber \\
&\stackrel{\eqref{eq: 2 polyad size},\eqref{eq: when d hat by 0}}{=}& \hspace{-0.3cm}
\left( (1\pm \nu) \prod_{i=1}^{j} (a^{\sP}_i)^{-\binom{j+1}{i}} m^{j+1} \hspace{-0.3cm} \sum_{\hat{\by}\in  \hat{B}_{j+1}(\hat{\bw})} \hspace{-0.3cm} d(\hat{\by}) d( \cK_{j+1}(J_2^{(j)}) \mid \hat{P}^{(j)}_2(\hat{\by})) \right)\pm \epsilon'^{1/4} m^{j+1} \nonumber \\
&\stackrel{{\rm \ref{item:J2}},\eqref{eq: 1 polyad size}}{=}&\hspace{-0.3cm}
 \left((1\pm 3\nu)\frac{m^{j+1}}{n^{j+1}}\left|\cK_{j+1}(\hat{P}^{(j)}_1(\hat{\by}))\right|\sum_{\hat{\by}\in \hat{B}_{j+1}(\hat{\bw})}  \hspace{-0.3cm} d(\hat{\by})d( \cK_{j+1}(J_1^{(j)}) \mid \hat{P}^{(j)}_1(\hat{\by}) ) \right) \pm 4\nu m^{j+1} \nonumber \\
&=& \hspace{-0.3cm} (1\pm 3\nu)\frac{m^{j+1}}{n^{j+1}} 
\left ( \sum_{\hat{\by}\in  \hat{B}_{j+1}(\hat{\bw})} \hspace{-0.3cm} d(\hat{\by})\left|\cK_{j+1}(J_1^{(j)}) \cap \cK_{j+1}(\hat{P}^{(j)}_1(\hat{\by}))\right|
\right)\pm 4\nu m^{j+1}\nonumber \\
&\stackrel{\eqref{eq: when d hat by 0}, \rm Subcl.~\ref{claim:O'regular}}{=}& \hspace{-0.3cm}
(1\pm 3\nu)\frac{m^{j+1}}{n^{j+1}}\left( \hspace{-0.5cm} \sum_{ \hspace{0.5cm} \hat{\by}\in \hat{A}(j+1,j,\ba^{\sP})}  \hspace{-0.9cm} \left|O'^{(j+1)}_1(\hat{\bw},b)\cap \cK_{j+1}(J_1^{(j)}) \cap \cK_{j+1}(\hat{P}_1^{(j)}(\hat{\by}))\right| \right)
\pm 5\nu m^{j+1} \nonumber \\
&=& \hspace{-0.3cm} \frac{m^{j+1}}{n^{j+1}} \left(\left|O'^{(j+1)}_1(\hat{\bw},b)\cap \cK_{j+1}(J_1^{(j)})\right| \pm 10\nu n^{j+1}\right).
\end{eqnarray}
\COMMENT{Second equality:
To elaborate why it holds, the following is an explanation.
We let $Q^{(j-1)} := J_2^{(j-1)}\cap \hat{P}^{(j-1)}_2(\hat{\by})$.
Then $Q^{(j-1)} \subseteq \hat{P}^{(j-1)}_2(\hat{\by})$.
Since $O'^{(j)}_2$ is $(\epsilon'^{1/2},d(\hat{\by}))$- regular with respect to $\hat{P}^{(j-1)}_2(\hat{\by})$,
we have two cases. \newline
1. $|\cK_j(Q^{(j-1)}) | > \epsilon'^{1/2} | \cK_j( \hat{P}^{(j-1)}_2(\hat{\by}) )| $.
 Then $ |O'^{(j)}_2 \cap \cK_{j}(J_2^{(j-1)})\cap \cK_j(\hat{P}^{(j-1)}_2(\hat{\by}))| = |O'^{(j)}_2 \cap \cK_{j}( Q^{(j-1)} )| = (d(\hat{\by}) \pm \epsilon'^{1/2}) |\cK_{j}(J_2^{(j-1)})\cap \cK_j(\hat{P}^{(j-1)}_2(\hat{\by})) | = d(\hat{\by}) |\cK_{j}(J_2^{(j-1)})\cap \cK_j(\hat{P}^{(j-1)}_2(\hat{\by})) | \pm \epsilon'^{1/2} | \cK_j( \hat{P}^{(j-1)}_2(\hat{\by}) )|$.\newline
2. If $|\cK_j(Q^{(j-1)}) | < \epsilon'^{1/2} | \cK_j( \hat{P}^{(j-1)}_2(\hat{\by}) )| $, then $|O'^{(j)}_2 \cap \cK_{j}( Q^{(j-1)} )|  = 0 \pm |\cK_j(Q^{(j-1)}) |  =  0 \pm \epsilon'^{1/2} | \cK_j( \hat{P}^{(j-1)}_2(\hat{\by}) )|= d(\hat{\by}) |\cK_{j}(J_2^{(j-1)})\cap \cK_j(\hat{P}^{(j-1)}_2(\hat{\by})) | \pm \epsilon'^{1/4} | \cK_j( \hat{P}^{(j-1)}_2(\hat{\by}) )|.$
Thus we have it for both cases.
The key thing here is that if $H^{(k)}$ is $(\epsilon,d)$ regular w.r.t $H^{(k-1)}$, then
for any $Q^{(k-1)}$, we have $ H^{(k)}\cap \cK_k(Q^{(k-1)}) = d |\cK_k(Q^{(k-1)})| \pm 2\epsilon |\cK_{k}(H^{(k-1)})|$ even if $|\cK_k(Q^{(k-1)})| < \epsilon  |\cK_{k}(H^{(k-1)})|$.
(Note that here, we add $\pm 2\epsilon |\cK_{k}(H^{(k-1)})|$, instead of multiplying $(d\pm \epsilon)$. )
}
Note that \ref{item:J1} implies that
\begin{align}\label{eq: cKj J2 triangle F}
\left|\cK_{j+1}(J_2^{(j)}))\triangle \cK_{j+1}(F^{(j)})\right| \leq \nu \binom{m}{j} \cdot m \leq  \nu m^{j+1}
\end{align}
Since $F^{(j)}\subseteq \hat{O'}^{(j)}_2(\hat{\bw})$ by assumption, \eqref{eq: cKj J2 triangle F} implies that 
\begin{align}\label{eq: J2 hat O' bw}
\left|\cK_{j+1}(J_2^{(j)})\setminus \cK_{j+1}( \hat{O}^{\prime(j)}_2(\hat{\bw}))\right|\leq \nu m^{j+1}.
\end{align}
We can transfer \eqref{eq: J2 hat O' bw} to the corresponding graphs on $V_1$ as follows: 
\begin{eqnarray}\label{eq: J1 setminus O'}
&&\hspace{-2.8cm} \left|\cK_{j+1}(J_1^{(j)}) \setminus \cK_{j+1}(\hat{O}^{\prime(j)}_1(\hat{\bw}))\right|
\stackrel{\eqref{eq: B hat bw b  def2}}{=} \sum_{\hat{\by}\notin \hat{B}_{j+1}(\hat{\bw})} \left|\cK_{j+1}(J_1^{(j)})\cap \cK_{j+1}(\hat{P}^{(j)}_1(\hat{\by}))\right| \nonumber \\
&\stackrel{\eqref{eq: 1 polyad size}}{\leq}& (1+ \nu) \prod_{i=1}^{j} (a^{\sP}_i)^{-\binom{j+1}{i}} n^{j+1} \hspace{-0.3cm}\sum_{\hat{\by}\notin \hat{B}_{j+1}(\hat{\bw})}\hspace{-0.3cm} d( \cK_{j+1}(J_1^{(j)}) \mid \hat{P}^{(j)}_1(\hat{\by})) \nonumber \\
 &\stackrel{{\rm \ref{item:J2}}}{\leq }& 
\left((1+ \nu) \prod_{i=1}^{j} (a^{\sP}_i)^{-\binom{j+1}{i}} n^{j+1} \hspace{-0.3cm} \sum_{\hat{\by}\notin \hat{B}_{j+1}(\hat{\bw})}  \hspace{-0.3cm} d( \cK_{j+1}(J_2^{(j)}) \mid \hat{P}^{(j)}_2(\hat{\by}))\right)  + 2 \nu n^{j+1} \nonumber \\
 &\stackrel{\eqref{eq: 2 polyad size}, \eqref{eq: cKj+1 hat O'2(j) is..}}{\leq}& \frac{n^{j+1}}{m^{j+1}} \left|\cK_{j+1}(J_2^{(j)}) \setminus \cK_{j+1}(\hat{O}^{\prime(j)}_2(\hat{\bw}))\right| + 5\nu n^{j+1}\nonumber \\
&\stackrel{\eqref{eq: J2 hat O' bw}}{\leq}&  6 \nu n^{j+1}.
\end{eqnarray}%
\COMMENT{We may insert as a second line $\sum_{\hat{\by}\notin\hat{B}_{j+1}(\hat{\bw})}  \hspace{-0.3cm} d( \cK_{j+1}(J_1^{(j)}) \mid \hat{P}^{(j)}_1(\hat{\by}))\cdot \left|\cK_{j+1}(\hat{P}^{(j)}_1(\hat{\by}))\right|$.}
Next we show that $|\cK_{j+1}(J_1^{(j)})\cap \cK_{j+1}(\hat{O}^{\prime(j)}_1(\hat{\bw}))|$ is not too small:
\begin{eqnarray}\label{eq: cK J1 O' hat bw big}
\left|\cK_{j+1}(J_1^{(j)})\cap \cK_{j+1}(\hat{O'}^{(j)}_1(\hat{\bw}))\right|
&\stackrel{\eqref{eq: J1 setminus O'}}{\geq} & 
\left|\cK_{j+1}(J_1^{(j)})\right| - 6 \nu n^{j+1}\nonumber \\
&\stackrel{\eqref{eq: J2 J1}}{\geq} & 
\frac{n^{j+1}}{m^{j+1}} \left|\cK_{j+1}(J_2^{(j)})\right|  - 11 \nu n^{j+1} \nonumber \\
&\stackrel{ \eqref{eq: cKj J2 triangle F}}{\geq } & \frac{n^{j+1}}{m^{j+1}} \left|\cK_{j+1}(F^{(j)})\right|  - 12\nu n^{j+1} \nonumber \\
&\stackrel{ \eqref{eq: F j-1 big}}{\geq} & 
\frac{n^{j+1}}{m^{j+1}}(\epsilon_0+3\nu^{1/20})\left|\cK_{j+1}(\hat{O}^{\prime(j)}_2(\hat{\bw}))\right| - 12\nu n^{j+1} \nonumber \\
&\stackrel{ \eqref{eq: Bj hat bw size},\eqref{eq: cK hat O' size}}{\geq} & (\epsilon_0+2\nu^{1/20})\left|\cK_{j+1}(\hat{O}^{\prime(j)}_1(\hat{\bw}))\right|.
\end{eqnarray}

Recall that $d$ was defined in \eqref{eq: def d def d ddddd}. We now can combine our estimates to conclude that
\begin{eqnarray}\label{eq: estimate combined}
&& \hspace{-2.9cm} \left|O'^{(j+1)}_2(\hat{\bw},b)\cap \cK_{j+1}(F^{(j)})\right| 
\stackrel{\eqref{eq: O2 cap J2},\eqref{eq: cKj J2 triangle F}}{=}
\frac{m^{j+1}}{n^{j+1}}\left(\left|{O'_1}^{(j+1)}(\hat{\bw},b)\cap \cK_{j+1}(J_1^{(j)})\right| \pm 20\nu n^{j+1} \right)
\nonumber \\
& \hspace{-0.5cm} \stackrel{\eqref{eq: O'1 is also good partition too}, \eqref{eq: cK J1 O' hat bw big}, \text{\ref{item:O'12}}}{=}&  \hspace{-0.3cm}
\frac{m^{j+1}}{n^{j+1}} \left( (d\pm (\epsilon_0+ 2\nu^{1/20})) \left|\cK_{j+1}(J_1^{(j)})\cap\cK_{j+1}( \hat{O}_1^{\prime(j)}(\hat{\bw}))\right| \pm 20\nu n^{j+1} \right)\nonumber \\
&\hspace{-0.5cm}=& \hspace{-0.3cm}
\frac{m^{j+1}}{n^{j+1}} \hspace{-0.05cm}(d\pm (\epsilon_0\hspace{-0.05cm}+ \hspace{-0.05cm}2\nu^{1/20}))\hspace{-0.05cm} \left(\hspace{-0.05cm} |\cK_{j+1}(J_1^{(j)})|\hspace{-0.05cm}- \hspace{-0.05cm}|\cK_{j+1}(J_1^{(j)}) \hspace{-0.05cm}\setminus\hspace{-0.05cm} \cK_{j+1}(\hat{O}^{\prime(j)}_1(\hat{\bw})) |\hspace{-0.05cm}\right) \hspace{-0.05cm}\pm \hspace{-0.05cm}20\nu m^{j+1}  \nonumber\\
&\hspace{-0.5cm}\stackrel{\eqref{eq: J2 J1},\eqref{eq: J1 setminus O'}}{=} &\hspace{-0.3cm}  (d\pm (\epsilon_0+ 2\nu^{1/20})) \left|\cK_{j+1}(J_2^{(j)})\right| \pm 40\nu m^{j+1}  \nonumber\\
&\hspace{-0.5cm}\stackrel{\eqref{eq: cKj J2 triangle F}}{=} & \hspace{-0.3cm}
( d\pm (\epsilon_0+ 2\nu^{1/20})) \left|\cK_{j+1}(F^{(j)})\right| \pm 50\nu m^{j+1}  \nonumber \\
&\hspace{-0.5cm}=& \hspace{-0.3cm} (d\pm (\epsilon_0 + 3\nu^{1/20}))\left|\cK_{j+1}(F^{(j)})\right|.
\end{eqnarray}

Here, we obtain the final inequality since \eqref{eq: Bj hat bw size}, \eqref{eq: cK hat O' size} and \eqref{eq: F j-1 big} imply  
$|\cK_{j+1}(F^{(j)})| \geq \epsilon_0^{2} m^{j+1}$ and $\nu\ll \epsilon_0$.
\eqref{eq: estimate combined} holds for all $F^{(j)}\subseteq \hat{O}_2^{\prime(j)}(\hat{\bw})$ satisfying \eqref{eq: F j-1 big}, thus $O_2'^{(j+1)}(\hat{\bw},b)$ is $(\epsilon_0+3\nu^{1/20},d)$-regular with respect to $\hat{O}^{\prime(j)}_2(\hat{\bw})$. 
This with the definition of $d$ completes the proof of Claim~\ref{cl: O'2 bw b is epsilon0 regular wrt}. 
\end{proof}

Claim~\ref{cl: O'2 bw b is epsilon0 regular wrt} and \eqref{eq: how equitable O'2(1) is} show that $\sO'_2$ is a $(1/a_1^\sO,\epsilon_0+3\nu^{1/20}, \ba^{\sO}, 2\nu^{1/20})$-equitable family of partitions 
which is also an $(\epsilon_0+3\nu^{1/20},d_{\ba^{\sO},k})$-partition of $G_2^{(k)}$ (as defined in Section~\ref{sec: subsub density function}). 
Note that $(\epsilon_0+3\nu^{1/20})/3 \leq \epsilon_0/2$, thus $((\epsilon_0+3\nu^{1/20})/3,\ba^{\sO},d_{\ba^{\sO},k})$ is a regularity instance.
Since $|G_2^{(k)}\triangle H_2^{(k)}|\leq \nu \binom{m}{k}$ by \ref{item:G22}, this means that
we can apply Lemma~\ref{lem: slightly different partition regularity} with the following objects and parameters.\newline

{\small
\begin{tabular}{c|c|c|c|c|c|c|c}
object/parameter & $\sO'_2$ & $\sO'_2$ & $\nu$ & $\epsilon_0+ 3\nu^{1/20}$ & $d_{\ba^{\sO},k}$ & $G^{(k)}_2$ & $ H^{(k)}_2$  \\ \hline
playing the role of & $\sP$ & $\sQ$ & $\nu$ & $\epsilon$ & $d_{\ba,k}$ & $H^{(k)}$ & $G^{(k)}$ 
 \\ 
\end{tabular}
}\newline \vspace{0.2cm}

\noindent
Hence $\sO'_2$ is also an $(\epsilon_0+4\nu^{1/20},d_{\ba^{\sO},k})$-partition of $H_2^{(k)}$.

\begin{step}\label{step6}
Adjusting $\sO_2'$ into an equipartition $\sO_2$.
\end{step}
Finally, 
we modify $\sO_2'$ to turn it from an `almost' equipartition into an equipartition $\sO_2$.
For this we apply Lemma~\ref{lem: removing lambda} with $\sO'_2, H_2^{(k)} \epsilon_0+4\nu^{1/20}, 2\nu^{1/20}, d_{\ba^{\sO},k}$ playing the roles of $\sP, H^{(k)}, \epsilon,\lambda, d_{\ba,k}$ respectively. 
This guarantees an $(\epsilon_0+3\nu^{1/200}, \ba^{\sO},d_{\ba^{\sO},k})$-equitable partition $\sO_2$ of $H_2^{(k)}$, which completes the proof.
\end{proof}

\subsection{Random samples}\label{Random samples}

To prove our results about random samples of hypergraphs,
we will need the following lemma due to Czygrinow and Nagle.
It states that $\epsilon$-regularity of a random complex is inherited by a random sample (but with significantly worse parameters).

\begin{lemma}[Czygrinow and Nagle \cite{CN11}] \label{lem: grabbing}
Suppose $0< 1/m_0, 1/s, \epsilon \ll \epsilon', d_0, 1/\ell, 1/k\leq 1$ and $k, \ell\in \N\sm \{1\}$ with $\ell\geq k$. 
Suppose $\sH=\{H^{(j)}\}_{j=1}^{k}$ is an $(\epsilon,(d_2,\dots,d_k))$-regular $(\ell,k)$-complex with $H^{(1)}=\{V_1, \dots , V_\ell\}$ 
such that $d_i\in [d_0,1]$, and $|V_i|>m_0$ for all $i\in [\ell]$. Let $s_1,\dots, s_{\ell} \geq s$ be integers such that $|V_i|\geq s_i$.
Then for  subsets $S_i \in \binom{V_i}{s_i}$ chosen uniformly at random, $\{H^{(j)}[S_1\cup S_2\cup \dots \cup S_{\ell}]\}_{j=1}^{k}$ is an $(\epsilon',(d_2,\dots, d_k))$-regular $(\ell,k)$-complex with probability at least $1- e^{-\epsilon s}$.
\end{lemma}\COMMENT{
In their paper, they used $\epsilon$ for regularity constant, and used $1-e^{-c s}$ in the theorem. 
They said that `for given $\epsilon',d_0,1/k$, there exists $c, \epsilon, n_0, s$ such that...'  However, we can just take $\min\{c,\epsilon\}$ to be our $\epsilon$.
(Note that $\epsilon$-regularity trivially implies $\epsilon_0$-regular for bigger $\epsilon_0$. )}
Note that in \cite{CN11}, the lemma is only stated for the case $\ell=k$, but the case $\ell\geq k$ follows via a union bound.
The next lemma generalizes Lemma~\ref{lem: grabbing} and shows how an equitable partition of a $k$-graph transfers with high probability to a random sample.

\begin{lemma}\label{lem: random choice}
Suppose $0< 1/n<1/q\ll \epsilon\ll   \epsilon' \ll 1/t, 1/k,$ and $k\in \N\sm \{1\}$.
Suppose that $\sP=\sP(k-1,\ba)$ is an $(\epsilon,\ba,d_{\ba,k})$-equitable partition of a $k$-graph $H$ on vertex set $V$ with $|V|=n$ and $\ba\in [t]^{k-1}$.
Then for a set $Q\in \binom{V}{q}$ chosen uniformly at random, with probability at least $1 -e^{-\epsilon^3 q}$, 
there exists an $(\epsilon',\ba,d_{\ba,k})$-equitable family of partitions $\sQ$ of $H[Q]$.
\end{lemma}\COMMENT{It suffices to prove this holds with probability at least 0.99. However, the paper \cite{CN11} emphasize that the probability $1- e^{-cq}$ is meaningful, so we prove $1- e^{-\epsilon^3 q}$. }
The parameter $\epsilon'$ in Lemma~\ref{lem: random choice} will be too large for our purposes.
But we can combine Lemmas~\ref{lem: similar} and~\ref{lem: random choice} to obtain the stronger assertion stated in (Q1)$_{\ref{lem: random choice2}}$ of Theorem~\ref{lem: random choice2}.

For the proof of Lemma~\ref{lem: random choice} we also need the following lemma which is easy to show, for example using Azuma's inequality.
We omit the proof.

\begin{lemma}\label{lem: random subset edge size}
	Suppose $0<1/n \leq  1/q \ll 1/k \leq 1/2$ and $1/q \ll \nu$.
	Let $H$ be an $n$-vertex $k$-graph on vertex set $V$. Let $Q\in \binom{V}{q}$ be a $q$-vertex subset of $V$ chosen uniformly at random. Then 
	$$\mathbb{P}\left[|H[Q]| = \frac{q^{k}}{n^{k}}|H| \pm \nu \binom{q}{k} \right] \geq 1- 2 e^{\frac{-\nu^2q}{8k^2}}.$$
\end{lemma}
\COMMENT{
	\begin{proof}
		Suppose we reveal $q$ vertices $v_1,\dots, v_q$ of $V$ one by one and let $Q:=\{v_1,\dots, v_q\}$.
		Consider an exposure martingale $X_0,\dots, X_q$ such that 
		$X_i:=\mathbb{E}[ |H[Q]| \mid v_1,\dots, v_i]$. 
		Then it is easy to check that this forms a $\binom{q}{k-1}$-Lipschitz martingale. Moreover, $X_0 = \mathbb{E}[|H[Q]|] = \binom{q}{k}|H|/\binom{n}{k}$.
		Therefore, by Azuma's inequality, we obtain
		$$
		\mathbb{P}\left[|H[Q]| 
		= \frac{q^{k}}{n^{k}}|H| \pm \nu \binom{q}{k} \right] 
		\geq 1 - 2e^{-\nu^2 \binom{q}{k}^2 /(3\binom{q}{k-1}^2 q) } \geq 1- 2 e^{-\nu^2q/(8k^2)}.$$
	\end{proof}
}

\begin{proof}[Proof of Lemma~\ref{lem: random choice}]
We choose an additional constant $\nu$ such that 
$$0<\epsilon \ll \nu \ll 1/t,1/k, \epsilon'.$$ 
Let $Q$ be a set of $q$ vertices selected uniformly at random in $V$.
Write $\sP^{(1)} = \{V_1, \dots , V_{a_1}\}$ and let $S_i := Q\cap V_i$.
For $\bS=(s_1,\dots, s_{a_1})$ with $\sum_{i=1}^{a_1} s_i = q$ and $s_i\in \N\cup \{0\}$, 
let $\cE(\bS)$ be the event that $|S_i|=s_i$ for all $i\in [a_1]$, and let 
$$I:=\left\{ \bS : s_i = (1\pm \epsilon)\frac{q}{a_1} \text{ for each }i\in [a_1]\right\}.$$ 
By some standard concentration inequality, we conclude
\begin{align}\label{eq: P s1..}
\mathbb{P}\left[\bigvee_{\bS\in I} \cE(\bS)\right]
 \geq 1 - 2a_1 e^{-\epsilon^2 q^2/(a_1^2 q)} \geq 1 - e^{-\epsilon^{5/2}q}.
\end{align}
Recall that for $\hat{\bx}\in \hat{A}(k,k-1,\ba)$, $\hat{\cP}(\hat{\bx})$ denotes the $(k,k-1)$-complex induced by $\hat{\bx}$ in $\sP$ as defined in \eqref{eq: complex definition by address} (as remarked at \eqref{eq: complex definition by address}, for a family of partitions $\sP$, $\hat{\cP}(\hat{\bx})$ is indeed a $(k,k-1)$-complex). Let 
$$A:= \{\hat{\bx}\in \hat{A}(k,k-1,\ba): d_{\ba,k}(\hat{\bx})\geq \nu\} \enspace \text{ and } \enspace 
G:= \bigcup_{\hat{\bx}\in A}\left( H\cap \cK_{k}(\hat{P}^{(k-1)}(\hat{\bx}))\right) \cup \left( H\setminus \cK_{k}(\sP^{(1)})\right) .$$
It is easy to see that $G\subseteq H$ and $|G\triangle H|\leq 2\nu \binom{n}{k}$.\COMMENT{
$|G\triangle H| \leq \sum_{ \hat{\bx}\in \hat{A}(k,k-1,\ba)} (\nu+\epsilon) |\cK_{k}(\hat{P}^{(k-1)}(\hat{\bx}))| \leq 2\nu \binom{n}{k}.$
}

For each $\hat{\bx}\in \hat{A}(k,k-1,\ba)$, let 
$$\hat{\cP}'(\hat{\bx}):= \hat{\cP}(\hat{\bx}) \cup \left\{G\cap \cK_{k}(\hat{P}^{(k-1)}(\hat{\bx}))\right\}.$$ 
Note that $\hat{\cP}'(\hat{\bx})$ is an $(\epsilon,(1/a_2,\dots, 1/a_{k-1},d_{\ba,k}(\hat{\bx})))$-regular $(k,k)$-complex for each $\hat{\bx} \in A$ and $\hat{\cP}(\hat{\bx})$ is a $(\epsilon,(1/a_2,\dots, 1/a_{k-1}))$-regular $(k,k-1)$-complex for each $\hat{\bx}\in \hat{A}(k,k-1,\ba)$.

For each $\hat{\bx}\in A$, 
we define the following event:
\begin{itemize}
\item[($\hat{\cE}(\hat{\bx})$)] $\hat{\cP}'(\hat{\bx})[Q]$ is an $(\epsilon'/2,(1/a_2,\dots, 1/a_{k-1}, d_{\ba,k}(\hat{\bx})))$-regular $(k,k)$-complex.
\end{itemize}
For each $\hat{\bx} \in \hat{A}(k,k-1,\ba)\setminus A$,
we also define the following event:
\begin{itemize}
\item[($\hat{\cE}(\hat{\bx})$)]\em$\hat{\cP}(\hat{\bx})[Q]$ is an $(\epsilon'/2,(1/a_2,\dots, 1/a_{k-1}))$-regular $(k,k-1)$-complex.
\end{itemize}
Note that for each $\hat{\bx} \in \hat{A}(k,k-1,\ba)\setminus A$, the event $\hat{\cE}(\hat{\bx})$ implies that the complex $\hat{\cP}'(\hat{\bx})[Q]$ is an $(\epsilon'/2,(1/a_2,\dots, 1/a_{k-1}, d_{\ba,k}(\hat{\bx}))$-regular complex as  we have $d_{\ba,k}(\hat{\bx}) \leq \nu  \ll  \epsilon'$ and $(G\cap \cK_{k}(\hat{P}^{(k-1)}(\hat{\bx})))[Q] = \emptyset$.\COMMENT{Note that an empty set is $(\epsilon'/2, \nu)$-regular with respect to any $(k-1)$-graph as long as $\nu < \epsilon'/2$.}
Thus we have that
\begin{equation}\label{eq: sP[Q] is good for G[Q]}
\begin{minipage}[c]{0.8\textwidth}\em
$\bigwedge_{\hat{\bx}\in \hat{A}(k,k-1,\ba)} \hat{\cE}(\hat{\bx})$ implies that  $\sP[Q]$ is an $(\epsilon'/2, d_{\ba,k})$-partition of $G[Q]$.
\end{minipage}
\end{equation}

Consider any $\hat{\bx} \in A$. Since $q$ is sufficiently large, we may apply Lemma~\ref{lem: grabbing} with the following objects and parameters.\newline

{\small
\begin{tabular}{c|c|c|c|c|c|c}
object/parameter & $\hat{\cP}'(\hat{\bx}) $ & $S_i$ & $1/a_1,\dots, 1/a_{k-1}$ & $d_{\ba,k}(\hat{\bx})$ & $\epsilon'/2$ & $\nu/2$   \\ \hline
playing the role of & $\cH$ & $S_i $ & $d_1,\dots,d_{k-1}$ & $d_{k}$ & $\epsilon'$ & $d_0$
 \\ 
\end{tabular}
}\newline \vspace{0.2cm}

\noindent
We obtain for any fixed $\bS\in I$, that
\begin{align}\label{eq:prop}
	\mathbb{P}[ \hat{\cE}(\hat{\bx}) \mid \cE(\bS)] \geq 1-e^{-\epsilon^2 q}.
\end{align}
In a similar way, for each $\hat{\bx} \in \hat{A}(k,k-1,\ba)\setminus A$, we can apply Lemma~\ref{lem: grabbing} to $\hat{\cP}(\hat{\bx})$ to obtain that \eqref{eq:prop} holds, too.
Thus for each $\hat{\bx}\in  \hat{A}(k,k-1,\ba)$, we obtain
\begin{eqnarray*}
\mathbb{P}[\hat{\cE}(\hat{\bx}) ] 
&=& \sum_{ \bS \in I} \mathbb{P}[ \hat{\cE}(\hat{\bx})  \mid \cE(\bS)] \mathbb{P}[  \cE(\bS)] + \sum_{ \bS \notin I}  \mathbb{P}[\hat{\cE}(\hat{\bx})  \mid \cE(\bS)] \mathbb{P}[  \cE(\bS)]\\
&\stackrel{\eqref{eq:prop}}{\geq}& (1-e^{ - \epsilon^2 q}) \sum_{ \bS \in I}   \mathbb{P}[  \cE(\bS)]  
\stackrel{\eqref{eq: P s1..}}{\geq} (1 - e^{-\epsilon^2 q}) (1- e^{-\epsilon^{5/2} q})
\geq 1 - 2 e^{-\epsilon^{5/2} q}.
\end{eqnarray*}
Let $\cE_0$ be the event that 
\begin{align}\label{eq: def event E0}
|(H\setminus G)[Q]|\leq 3\nu \binom{q}{k}.
\end{align}
Since $|H\setminus G| \leq 2\nu \binom{n}{k}$, 
we may apply Lemma~\ref{lem: random subset edge size} with $n, H\setminus G,Q, \nu/2$ playing the roles of $n, H,Q,\nu$ to obtain
$$\mathbb{P}[\cE_0] 
\geq 1 - e^{-\nu^{3}q}.$$
As $ |\hat{A}(k,k-1,\ba)|\leq t^{2^k}$ by Proposition~\ref{prop: hat relation}(viii), 
we conclude
\begin{align}\label{eq:Qprob}
	\mathbb{P}\left[ \cE_0\wedge \bigwedge_{\hat{\bx}\in  \hat{A}(k,k-1,\ba)} \hat{\cE}(\hat{\bx})\right] 
	\geq 1- e^{-\nu^{3}q} - 2t^{2^k} e^{-\epsilon^{5/2}q} 
	\geq 1- e^{-\epsilon^{8/3}q}.
\end{align}

Now suppose that $\cE(\bS)$ holds for some $\bS\in I$ and that $\cE_0 \wedge\bigwedge_{\hat{\bx}\in  \hat{A}(k,k-1,\ba)} \hat{\cE}(\hat{\bx})$ holds.
Then $\sP$ induces a family of partitions $\sP[Q]$ on $Q$ which is $(1/a_1,\epsilon'/2,\ba,\epsilon)$-equitable.
Note $\epsilon'\ll 1/t, 1/k$, thus $(\epsilon'/6, \ba,d_{\ba,k})$ is a regularity instance.
Since $\nu \ll  \epsilon' \ll 1/t$, by using \eqref{eq: sP[Q] is good for G[Q]}, we can apply Lemma~\ref{lem: slightly different partition regularity} with the following objects and parameters.\newline

{\small
\begin{tabular}{c|c|c|c|c|c|c|c|c}
object/parameter & $\sP[Q]$ & $\sP[Q]$ & $3\nu$ & $\epsilon'/2$ & $d_{\ba,k}$ & $G[Q]$ & $H[Q]$  & $\epsilon$ \\ \hline
playing the role of & $\sP$ & $\sQ$ & $\nu$ & $\epsilon$ & $d_{\ba,k}$ & $H^{(k)}$ & $G^{(k)}$  & $\lambda$
 \\ 
\end{tabular}
}\newline \vspace{0.2cm}

This implies that $\sP[Q]$ is an $(1/a_1,\epsilon'/2+ \nu^{1/7},\ba,\nu^{1/7} )$-equitable family of partitions on $Q$ which is also an $(\epsilon'/2 +\nu^{1/7},d_{\ba,k})$-partition of $H[Q]$.\COMMENT{Note that $\epsilon + (3\nu)^{1/6}\leq \nu^{1/7}$.}

Finally, since $\nu \ll \epsilon'$, Lemma~\ref{lem: removing lambda} implies that there exists 
a family of partitions $\sQ$ which is an $(\epsilon', \ba,d_{\ba,k})$-equitable partition of $H[Q]$.
By~\eqref{eq: P s1..} and \eqref{eq:Qprob}, this completes the proof.
\end{proof}

Next we proceed with the proof of Theorem~\ref{lem: random choice2}.
To prove (Q1)$_{\ref{lem: random choice2}}$, we first apply the regular approximation lemma (Theorem~\ref{thm: RAL}) to obtain an $\epsilon$-equitable partition $\sP_1$ of a $k$-graph $G$ that is very close to $H$.
Lemma~\ref{lem: random choice} implies that (with high probability) $G[Q]$ has a regularity partition $\sP_2$ which has the same parameters as $\sP_1$, 
except for a much worse regularity parameter $\epsilon'$. 
However, we still have $\epsilon' \ll \epsilon_0$ and thus we can now apply Lemma~\ref{lem: similar}
to $G$, $G[Q]$ and $\sP_1, \sP_2, \sO_1$ to obtain an equitable partition $\sO_2$ of $G[Q]$ which reflects $\sO_1$.
By Lemma~\ref{lem: slightly different partition regularity}, $\sO_2$ is also an equitable partition of $H[Q]$.
To prove (Q2)$_{\ref{lem: random choice2}}$,
we again apply Lemma~\ref{lem: similar} but with the roles of $G$ and $G[Q]$ interchanged.

\begin{proof}[Proof of Theorem~\ref{lem: random choice2}]
Choose new constants $\eta, \nu$ so that 
$ c \ll \eta \ll  \nu \ll \delta.$ 
Let $\overline{\epsilon}:\mathbb{N}^{k-1}\rightarrow (0,1]$ be a function such that for all $\bb\in \mathbb{N}^{k-1}$, we have
$$\overline{\epsilon}(\bb) \ll \norm{\bb}^{-k}.$$
Let $t_0:= t_{\ref{thm: RAL}} (\eta,\nu,\overline{\epsilon})$.

By Theorem~\ref{thm: RAL}, there exists a $t_0$-bounded $(\eta,\overline{\epsilon}(\ba^{\sP}),\ba^{\sP})$-equitable family of partitions $\sP_1=\sP_1(k-1,\ba^{\sP})$, a $k$-graph $G$ and a density function $d_{\ba^{\sP},k}$ such that the following hold.
\begin{itemize}
\item[(G1)$_{\ref{lem: random choice2}}$]
$\sP_1$ is an $(\overline{\epsilon}(\ba^{\sP}),d_{\ba^{\sP},k})$-partition of $G$, and
\item[(G2)$_{\ref{lem: random choice2}}$] $|G\triangle H|\leq \nu \binom{n}{k}$.
\end{itemize}
(Here (G1)$_{\ref{lem: random choice2}}$ follows from Theorem~\ref{thm: RAL}(ii), \eqref{eq: perfectly regular is regular}, and Lemma~\ref{lem: maps bijections}.)

Let $\epsilon := \overline{\epsilon}(\ba^{\sP})$ and $T:= \norm{\ba^{\sP}}$. As $t_0$ only depends on $\eta,\nu,\overline{\epsilon}$, we may assume that $c\ll \epsilon$.%
\COMMENT{
Here we also use that $c\ll 1/k$ since $c \ll \delta\ll \epsilon_0$ and $\epsilon_0\leq \norm{\ba}^{-4^k}\leq 2^{-4^k}$ as $(2\epsilon_0/3,\ba,d_{\ba,k})$ is a regularity instance.
}
Together with the choice of $\overline{\epsilon}$ and the fact that  $1/T\leq 1/a_1^{\sP} \leq \eta$, this implies 
$$0 < 1/n< 1/q \ll c \ll \epsilon \ll 1/T, 1/a_1^{\sP}  \ll  \nu \ll \delta \ll \epsilon_0 \leq 1.$$
Additionally, we choose $\epsilon'$ so that
\begin{align}\label{eq: constants hierarchy}
0 < 1/n< 1/q \ll c \ll \epsilon \ll \epsilon' \ll 1/T, 1/a_1^{\sP}  \ll  \nu \ll \delta \ll \epsilon_0  \leq 1.
\end{align}
Let $\cE_0$ be the event that 
$$|G[Q]\triangle H[Q]|\leq 2\nu \binom{q}{k}.$$ 
Property (G2)$_{\ref{lem: random choice2}}$ and Lemma~\ref{lem: random subset edge size} imply that
\begin{align}\label{eq: mathbb prob E0}
\mathbb{P}[\cE_0] \geq 1- e^{-\nu^{3}q}.
\end{align}

Let $\cE_1$ be the event that there exists a family of partitions $\sP_2=\sP_2(k-1,\ba^{\sP})$  
which is an $(\epsilon',\ba^{\sP},d_{\ba^{\sP},k})$-equitable partition of $G[Q]$. 
Since $\epsilon \ll \epsilon'$, 
Lemma~\ref{lem: random choice} implies that 
\begin{align}\label{eq:E_1}
	\mathbb{P}[\cE_1] \geq 1 - e^{-\epsilon^{3}q}.
\end{align}
Thus \eqref{eq: mathbb prob E0} and \eqref{eq:E_1}  imply that
\begin{align}\label{eq: E0 wedge E1 holds with prob..}
\mathbb{P}[\cE_0 \wedge \cE_1 ] \geq 1 - 2e^{-\epsilon^{3}q} \geq 1- e^{-cq}.
\end{align}
Hence it suffices to show that the two statements (Q1)$_{\ref{lem: random choice2}}$ and (Q2)$_{\ref{lem: random choice2}}$ both hold 
if we condition on $\cE_0\wedge \cE_1$.

First, assume $\cE_0\wedge \cE_1$ holds and $\sO_1$ exists as in (Q1)$_{\ref{lem: random choice2}}$. 
As $\nu\ll \delta \ll \epsilon_0$, 
we can apply Lemma~\ref{lem: slightly different partition regularity} with $\sO_1$, $\sO_1$, $\nu$, $\epsilon_0$, $d_{\ba,k}$, $H$ and $G$ playing the roles of $\sP$, $\sQ$, $\nu$, $\epsilon$, $d_{\ba,k}$, $H^{(k)}$ and $G^{(k)}$, respectively, to 
conclude that  $\sO_1$ is also an 
$(\epsilon_0+\delta/3,d_{ \ba,k})$-partition of~$G$.

Note that $(\epsilon_0+\delta/3)/2 \leq 2\epsilon_0/3$, thus $((\epsilon_0+\delta/3)/2, \ba, d_{\ba,k})$ is a regularity instance.
By this and \eqref{eq: constants hierarchy}, we can apply Lemma~\ref{lem: similar} with the following objects and parameters. \newline

{\small
\begin{tabular}{c|c|c|c|c|c|c|c|c|c|c|c|c|c}
object/parameter & $n$ & $q$ & $\sO_1$ & $\sP_1$ & $\sP_2$ & $G$ & $G[Q]$ &$\epsilon'$ & $T$  & $\delta/3$ & $\epsilon_0+\delta/3$&   $ d_{\ba{^{\sP},k}}$ & $d_{\ba,k}$ \\ \hline
playing the role of  & $n$ & $m$& $\sO_1$& $\sQ_1$& $\sQ_2$& $H^{(k)}_1$& $H^{(k)}_2$ & $\epsilon$ & $T$ & $\delta$ &  $\epsilon_0$&  $ d_{\ba{^{\sQ},k}}$   & $d_{\ba^{\sO},k}$\\
\end{tabular}
}\newline \vspace{0.2cm}

\noindent Hence there exists an $(\epsilon_0+2\delta/3,\ba,d_{ \ba,k})$-equitable partition $\sO_2$ of $G[Q]$.
Since $\cE_0$ holds and $\nu\ll \delta \ll \epsilon_0$, 
we can apply Lemma~\ref{lem: slightly different partition regularity} with $\sO_2$, $\sO_2$, $2\nu$, $\epsilon_0+2\delta/3$, $d_{\ba,k}$, $H[Q]$ and $G[Q]$ playing the roles of $\sP$, $\sQ$, $\nu$, $\epsilon$, $d_{\ba,k}$, $G^{(k)}$ and $H^{(k)}$, respectively. Then we conclude that $\sO_2$ is an $(\epsilon_0+\delta,\ba,d_{ \ba,k})$-equitable partition of $H[Q]$. 
Thus $\cE_0\wedge \cE_1$ implies (Q1)$_{\ref{lem: random choice2}}$.

Now assume $\cE_0\wedge \cE_1$ holds and $\sO_2$ exists as in (Q2)$_{\ref{lem: random choice2}}$. 
As $\nu\ll \delta \ll \epsilon_0$, 
we can apply Lemma~\ref{lem: slightly different partition regularity} with  $\sO_2$, $\sO_2$, $2\nu$, $\epsilon_0$, $d_{\ba,k}$, $H[Q]$ and $G[Q]$ playing the roles of 
$\sP$, $\sQ$, $\nu$, $\epsilon$, $d_{\ba,k}$, $H^{(k)}$ and $G^{(k)}$, respectively.
Thus $\sO_2$ is an $(\epsilon_0+\delta/3,\ba,d_{ \ba,k})$-equitable partition of $G[Q]$.
By \eqref{eq: constants hierarchy} and the fact that $R$ is a regularity instance, 
we can apply Lemma~\ref{lem: similar} with the following objects and parameters. \newline

{\small
\begin{tabular}{c|c|c|c|c|c|c|c|c|c|c|c|c|c}
object/parameter & $q$ & $n$ & $\sO_2$ & $\sP_2$ & $\sP_1$ & $G[Q]$& $G$&  $\epsilon'$ & $T$& $\delta/3$ & $\epsilon_0+\delta/3$&  $ d_{\ba{^{\sP},k}}$ & $d_{\ba,k}$\\ \hline
playing the role of & $n$ & $m$& $\sO_1$& $\sQ_1$& $\sQ_2$& $H_1^{(k)}$&$H_2^{(k)}$& $\epsilon$ & $T$& $\delta$ & $\epsilon_0$&   $ d_{\ba{^{\sQ},k}}$   & $d_{\ba^{\sO},k}$    
\end{tabular}
}\newline \vspace{0.2cm}

Thus there exists a family of partitions $\sO_1$ which is an $(\epsilon_0+2\delta/3,\ba,d_{ \ba,k})$-equitable partition of $G$.
By (G2)$_{\ref{lem: random choice2}}$ and the fact that $\nu\ll \delta \ll \epsilon_0$, 
we can apply Lemma~\ref{lem: slightly different partition regularity} with $\sO_1$, $\sO_1$, $\nu$, $\epsilon_0+2\delta/3$, $d_{\ba,k}$, $H$ and $G$ playing the roles of 
$\sP$, $\sQ$, $\nu$, $\epsilon$, $d_{\ba,k}$, $G^{(k)}$ and $H^{(k)}$, respectively. 
We conclude that $\sO_1$ is an $(\epsilon_0+\delta,\ba,d_{ \ba,k})$-equitable partition of $H$. 
Thus $\cE_0\wedge \cE_1$ implies (Q2)$_{\ref{lem: random choice}}$.
\end{proof}

\bibliographystyle{amsplain}
\bibliography{littesting}

\end{document}